%% file: stochastic_representation_PF1.20.2013b.tex
\documentclass[11pt,reqno]{amsart}

\input{latexformat}

\input{latexnewmacros}

\numberwithin{equation}{section}

\usepackage{amssymb}
\usepackage{hyperref}
\usepackage{mathrsfs}
\usepackage[usenames]{color}
\usepackage{url}
\usepackage{verbatim}

\hypersetup{pdftitle={Stochastic representation of solutions to degenerate elliptic and parabolic boundary value and obstacle problems with Dirichlet boundary conditions}}
\hypersetup{pdfauthor={Paul M. N. Feehan and Camelia Pop}}

\newcount\hh
\newcount\mm
\mm=\time
\hh=\time
\divide\hh by 60
\divide\mm by 60
\multiply\mm by 60
\mm=-\mm
\advance\mm by \time
\def\hhmm{\number\hh:\ifnum\mm<10{}0\fi\number\mm}

\begin{document}

\title[Stochastic representation of solutions to Dirichlet variational inequalities]{Stochastic representation of 
solutions to degenerate elliptic and parabolic boundary value and obstacle problems with Dirichlet boundary conditions}

\author[P. Feehan]{Paul M. N. Feehan}
\address[PF]{Department of Mathematics, Rutgers, The State University of New Jersey, 110 Frelinghuysen Road, Piscataway, 
NJ 08854-8019}
\email[PF]{feehan@math.rutgers.edu}

\author[C. Pop]{Camelia Pop}
\address[CP]{Department of Mathematics, University of Pennsylvania, 209 South 33rd Street, Philadelphia, PA 19104-6395}
\email[CP]{cpop@math.upenn.edu}

\date{January 20, 2013}

\begin{abstract}
We prove existence and uniqueness of stochastic representations for solutions to elliptic and parabolic boundary value 
and obstacle problems associated with a degenerate Markov
diffusion process. In particular, our article focuses on the Heston stochastic volatility process, which is widely used 
as an asset price model in mathematical finance and a
paradigm for a degenerate diffusion process where the degeneracy in the diffusion coefficient is proportional to the 
square root of the distance to the boundary of the
half-plane. The generator of this process with killing, called the elliptic Heston operator, is a second-order, 
degenerate, elliptic partial differential operator whose
coefficients have linear growth in the spatial variables and where the degeneracy in the operator symbol is proportional 
to the distance to the boundary of the half-plane. In
mathematical finance, solutions to terminal/boundary value or obstacle problems for the parabolic Heston operator 
correspond to value functions for American-
style options on the underlying asset.
\end{abstract}

%



\thanks{PF was partially supported by NSF grant DMS-1059206. CP was partially supported by a Rutgers University 
fellowship. }

\maketitle
\tableofcontents

\section{Introduction}
\label{sec:stochastic_representation_Introduction}
Since its discovery by Mark Kac \cite{Kac_1949}, inspired in turn by the doctoral dissertation of Richard Feynman 
\cite{Feynman_thesis}, the \emph{Feynman-Kac} (or
\emph{stochastic representation}) \emph{formula} has provided a link between probability theory and partial differential 
equations which has steadily deepened and developed
during the intervening years. Moreover, judging by continuing interest in its applications to mathematical finance 
\cite{KaratzasShreve1998} and mathematical physics
\cite{Lorinczi_Hiroshima_Betz_2011, Simon_functionalintegration}, including non-linear parabolic equations 
\cite{Constantin_Iyer_2008}, this trend shows no sign of abating.
However, while stochastic representation formulae for solutions to linear, second-order elliptic and parabolic boundary 
and obstacle problems are well established when the
generator, $-A$, of the Markov stochastic process is \emph{strictly elliptic} \cite{Bensoussan_Lions, FriedmanSDE, 
KaratzasShreve1991, Oksendal_2003} in the sense of \cite[p.
31]{GilbargTrudinger}, the literature is far less complete when $A$ is \emph{degenerate elliptic}, that is, only has a 
\emph{non-negative definite characteristic form} in the
sense of \cite{Oleinik_Radkevic}, and its coefficients are unbounded.

In this article, we prove stochastic representation formulae for solutions to an \emph{elliptic boundary value problem},
\begin{equation}
\label{eq:stochastic_representation_HestonEllipticEqBVP}
Au = f \quad \hbox{on }\sO,
\end{equation}
and an \emph{elliptic obstacle problem},
\begin{equation}
\label{eq:stochastic_representation_Elliptic_obstacle_problem}
\min\{Au-f,u-\psi\} = 0 \quad \hbox{on }\sO,
\end{equation}
respectively, subject to a \emph{partial} Dirichlet boundary condition,
\begin{equation}
\label{eq:stochastic_representation_HestonEllipticBoundaryCondition_Gamma1}
u = g \quad\hbox{on } \Gamma_1.
\end{equation}
Here, the subset $\sO\subseteqq\HH$ is a (possibly unbounded) domain (connected, open subset) in the open upper half-space 
$\HH := \RR^{d-1}\times(0,\infty)$ (where $d\geq 2$),
$\Gamma_1 = \partial\sO\cap\HH$ is the portion of the boundary, $\partial\sO$, of $\sO$ which lies in $\HH$, $f:\sO\to\RR$ 
is a source function, the function $g:\Gamma_1\to\RR$
prescribes a Dirichlet boundary condition along $\Gamma_1$ and $\psi:\sO\cup\Gamma_1\to\RR$ is an obstacle function which 
is compatible with $g$ in the sense that
\begin{equation}
\label{eq:stochastic_representation_Elliptic_compatibility_g_psi_beta_geq_1}
\psi\leq g \quad\hbox{on }\Gamma_1,
\end{equation}
while $A$ is an elliptic differential operator on $\sO$ which is degenerate along the interior, $\Gamma_0$, of 
$\partial\HH\cap\partial\sO$ and may have unbounded coefficients.
We require $\Gamma_0$ to be non-empty throughout this article as, otherwise, if $\sO$ is bounded (and the coefficients of 
$A$ are, say, continuous on $\bar\sO$), then standard
results apply \cite{Bensoussan_Lions, FriedmanSDE, KaratzasShreve1991, Oksendal_2003}. However, an additional boundary 
condition is \emph{not} necessarily prescribed along
$\Gamma_0$. Rather, we shall see that our stochastic representation formulae will provide the unique solutions to 
\eqref{eq:stochastic_representation_HestonEllipticEqBVP} or
\eqref{eq:stochastic_representation_Elliptic_obstacle_problem}, together with 
\eqref{eq:stochastic_representation_HestonEllipticBoundaryCondition_Gamma1}, when we seek solutions
which are suitably smooth up to the boundary portion $\Gamma_0$, a property which is guaranteed when the solutions lie in 
certain weighted H\"
older spaces (by analogy with \cite{DaskalHamilton1998}), \emph{or} replace the boundary condition 
\eqref{eq:stochastic_representation_HestonEllipticBoundaryCondition_Gamma1}
with the full Dirichlet condition,
\begin{equation}
\label{eq:stochastic_representation_HestonEllipticBoundaryCondition_whole_boundary}
u = g \quad\hbox{on } \partial\sO,
\end{equation}
in which case the solutions are not guaranteed to be any more than continuous up to $\Gamma_0$ and $\psi:\bar\sO\to\RR$ is 
now required to be compatible with $g$ in the sense
that,
\begin{equation}
\label{eq:stochastic_representation_Elliptic_compatibility_g_psi_beta_leq_1}
\psi\leq g \quad\hbox{on }\partial\sO.
\end{equation}
We also prove stochastic representation formulae for solutions to a \emph{parabolic terminal/boundary value problem},
\begin{equation}
\label{eq:stochastic_representation_Parabolic_BVP}
-u_t + Au = f \quad \hbox{on }Q,
\end{equation}
and a \emph{parabolic obstacle problem},
\begin{equation}
\label{eq:stochastic_representation_Parabolic_obstacle_problem}
\min\{-u_t + Au-f,u-\psi\} = 0 \quad \hbox{on }Q,
\end{equation}
respectively, subject to the \emph{partial} terminal/boundary condition,
\begin{equation}
\label{eq:stochastic_representation_Parabolic_BVP_boundary_condition_Gamma_1}
u = g \quad\hbox{on } \eth^1 Q.
\end{equation}
Here, we define $Q:=(0,T)\times\sO$, where $0<T<\infty$, and define
\begin{equation}
\label{eq:stochastic_representation_Parabolic_boundary_without_degenerate_part}
\eth^1{Q}:= (0, T) \times \Gamma_1\cup \{T\} \times \left( \sO \cup \Gamma_1\right),
\end{equation}
to be a subset of the parabolic boundary of $Q$, and now assume given a source function $f:Q\to\RR$, a Dirichlet boundary 
data function $g:\eth^1 Q\to\RR$, and an obstacle
function $\psi:Q\cup \eth^1 Q\to\RR$ which is compatible with $g$ in the sense that,
\begin{equation}
\label{eq:stochastic_representation_Parabolic_compatibility_g_psi_beta_geq_1}
\psi\leq g \quad\hbox{on }\eth^1 Q.
\end{equation}
Just as in the elliptic case, we shall either consider solutions which are suitably smooth up to $(0, T)\times\Gamma_0$, 
but impose no explicit Dirichlet boundary condition along
$(0, T)\times\Gamma_0$, \emph{or} replace the boundary condition in 
\eqref{eq:stochastic_representation_Parabolic_BVP_boundary_condition_Gamma_1} with the full Dirichlet
condition
\begin{equation}
\label{eq:stochastic_representation_Parabolic_BVP_boundary_condition_whole_boundary}
u = g \quad\hbox{on } \eth{Q},
\end{equation}
where
\begin{equation}
\label{eq:stochastic_representation_Parabolic_boundary_classical}
\eth{Q} := (0, T) \times \partial \sO \cup \{T\} \times \bar\sO,
\end{equation}
is the full parabolic boundary of $Q$, in which case the solutions are not guaranteed to be any more than continuous up to 
$(0, T)\times\Gamma_0$ and $\psi:Q\cup \eth Q\to\RR$ is
now compatible with $g$ in the sense that
\begin{equation}
\label{eq:stochastic_representation_Parabolic_compatibility_g_psi_beta_leq_1}
\psi\leq g \quad\hbox{on }\eth Q.
\end{equation}
Before giving a detailed account of our main results, we summarize a few applications.

\subsection{Applications}
In mathematical finance, a solution, $u$, to the elliptic obstacle problem 
\eqref{eq:stochastic_representation_Elliptic_obstacle_problem},
\eqref{eq:stochastic_representation_HestonEllipticBoundaryCondition_Gamma1}, when $f=0$, can be interpreted as the value 
function for a \emph{perpetual American-style option}
with \emph{payoff} function given by the obstacle function, $\psi$, while a solution, $u$, to the corresponding 
\emph{parabolic} obstacle problem
\eqref{eq:stochastic_representation_Parabolic_obstacle_problem}, 
\eqref{eq:stochastic_representation_Parabolic_BVP_boundary_condition_Gamma_1}, when $f=0$, can be interpreted as
the value function for a \emph{finite-maturity} American-style option with payoff function given by a terminal condition 
function, $h = g(T,\cdot):\sO\to\RR$, which typically
coincides on $\{T\}\times\sO$ with the obstacle function, $\psi$. For example, in the case of an American-style put 
option, one chooses $\psi(x,y) = (E-e^x)^+$, $\forall
(x,y)\in\sO$, where $E>0$ is a positive
constant. While solutions to \eqref{eq:stochastic_representation_HestonEllipticEqBVP}, 
\eqref{eq:stochastic_representation_HestonEllipticBoundaryCondition_Gamma1} do not have an
immediate interpretation in mathematical finance, a solution, $u$, to the corresponding \emph{parabolic} terminal/boundary 
value problem
\eqref{eq:stochastic_representation_Parabolic_BVP}, 
\eqref{eq:stochastic_representation_Parabolic_BVP_boundary_condition_Gamma_1}, when $f=0$, can be interpreted as the 
value
function for a \emph{European-style option} with payoff function given by the terminal condition function, $h$. For 
example, in the case of a European-style put option, one
chooses $h(x,y) = (E-e^x)^+$, $\forall (x,y)\in\sO$.

Stochastic representation formulae underly Monte Carlo methods of numerical computation of value functions for option 
pricing in mathematical finance \cite{Glasserman}. As is
well-known to practitioners, the question of Monte Carlo simulation of solutions to the Heston stochastic differential 
equation is especially delicate \cite{Andersen_2007,
Lord_Koekkoek_vanDijk_2010}. We hope that our article sheds further light on these issues.

\subsection{Summary of main results}
\label{subsec:stochastic_representation_Summary}
In this article, we set $d=2$ and choose $-A$ to be the generator of the two-dimensional Heston stochastic volatility 
process with killing rate $r$ \cite{Heston1993}, a
degenerate diffusion process well known in mathematical finance,
\begin{equation}
\label{eq:stochastic_representation_HestonGenerator}
-Av := \frac{y}{2}\left(v_{xx} + 2\rho\sigma v_{xy} + \sigma^2 v_{yy}\right) + (r-q-y/2)v_x + \kappa(\theta-y)v_y - rv, 
\quad v\in C^\infty(\HH).
\end{equation}
Nonetheless, we expect that many of our results would extend to a much broader class of degenerate Markov processes and we 
shall address such questions elsewhere. Throughout this
article, the coefficients of $A$ are required to obey

\begin{assump}
[Ellipticity condition for the Heston operator coefficients]
\label{assump:HestonCoefficients}
The coefficients defining $A$ in \eqref{eq:stochastic_representation_HestonGenerator} are constants obeying
\begin{equation}
\label{eq:EllipticHeston}
\sigma \neq 0, -1< \rho < 1,
\end{equation}
and $\kappa>0$, $\theta>0$, and\footnote{We impose additional conditions, such as $q\geq 0$, $r\geq 0$, or $r>0$, 
depending on the problem under consideration; we only require
that $q\geq 0$ when deriving the supermartingale property in Lemma \ref{lem:stochastic_representation_PropertiesHeston} 
\eqref{item:SupermartingaleX}, a property used only in the
elliptic case.} $q, r \in \RR$.
\end{assump}

Let $(\Omega,\sF,\FF,\QQ)$ be a filtered probability space satisfying the usual conditions, where $\FF = \{\sF(s)\}_{s\geq 
0}$ is the $\QQ$-completion of the natural filtration
of $(W(s))_{s\geq 0}$, and $(W(s))_{s\geq 0}$ is a standard Brownian motion with values in $\RR^2$. For $0\leq 
t<T<\infty$, let $\sT_{t,T}$ denote the set of $\FF$-stopping times
with values in $[t,T]$. Let $(X^{t,x,y}(s),Y^{t,y}(s))_{s\geq t}$ denote a continuous version of the
strong solution to the Heston stochastic differential equation
\begin{equation}
\label{eq:stochastic_representation_HestonSDE}
\begin{aligned}
d X(s) &= \left(r-q-\frac{Y(s)}{2}\right) \,ds + \sqrt{Y(s)} \,dW_1(s), \quad s > t,\\
d Y(s) &= \kappa\left(\vartheta-Y(s)\right) \,ds + \sigma \sqrt{Y(s)}\left(\rho \,dW_1(s)+\sqrt{1-\rho^2}\,dW_2(s)\right), 
\quad s > t,\\
&\quad (X(t), Y(t)) = (x,y),
\end{aligned}
\end{equation}
which exists by Corollary \ref{cor:HestonSDEStrongSolution}, where the coefficients are as in Assumption 
\ref{assump:HestonCoefficients}. For brevity, we sometimes denote
$z=(x,y)$ and $(Z^{t,z}(s))_{s \geq t} = (X^{t,x,y}(s),Y^{t,y}(s))_{s\geq t}$. We omit the superscripts $(t,z)$ and 
$(t,x,y)$ when the initial condition is clear from the
context, or we omit the superscript $t$ when $t=0$. We let
\begin{align}
\label{defn:stochastic_representation_beta}
\beta &:= \frac{2\kappa\vartheta}{\sigma^2},\\
\label{defn:stochastic_representation_mu}
\mu &:= \frac{2\kappa}{\sigma^2},
\end{align}
denote the \emph{Feller parameters} associated with the Heston process.

\subsubsection{Existence and uniqueness of solutions to elliptic boundary value problems}
For an integer $k\geq 0$, we let $C^{k}(\sO)$ denote the vector space of functions whose derivatives up to order $k$ are 
continuous on $\sO$ and let $C^{k}(\bar\sO)$ denote the
Banach space of functions whose derivatives up to order $k$ are \emph{uniformly continuous} and \emph{bounded} on $\sO$ 
\cite[\S 1.25 \& \S 1.26]{Adams}. If $T \subsetneqq
\partial\sO$ is a relatively open set, we let $C^k_{\loc}(\sO\cup T)$ denote the vector space of functions, $u$, such 
that, for any precompact open subset $U\Subset \sO \cup T$,
we have $u \in C^k(\bar U)$.

We shall often appeal to the following

\begin{hyp}[Growth condition]
\label{hyp:stochastic_representation_Growth_elliptic}
If $v$ is a function then, for all $(x,y)$ in its domain of definition,
\begin{equation}
\label{eq:stochastic_representation_Growth_elliptic}
|v(x,y)| \leq C(1 + e^{M_1 y}+e^{M_2x}),
\end{equation}
where $C>0$, $0\leq M_1 < \min\left\{r/\left(\kappa\vartheta\right), \mu\right\}$, and $M_2 \in[0,1)$.
\end{hyp}

Let $U \subseteqq\HH$ be an open set. We denote
\begin{equation}
\label{eq:stochastic_representation_Stopping_time_open_set}
\tau^{t,z}_{U} := \inf \left\{s \geq t: Z^{t,z}(s) \notin U \right\},
\end{equation}
and we let
\begin{equation}
\label{eq:stochastic_representation_Stopping_time_domain_with_degenerate_boundary}
\nu^{t,z}_{U} := \inf \left\{s \geq t: Z^{t,z}(s) \notin U \cup \left(\bar U \cap \partial \HH\right)\right\}.
\end{equation}
Notice that if $\bar U \cap \partial \HH = \emptyset$, then $\tau^{t,z}_{U} = \nu^{t,z}_{U}$. We also have that 
$\tau^{t,z}_{U}=\nu^{t,z}_{U}$ when $\beta \geq 1$, because in
this case the process $Z^{t,z}$ does not reach the boundary $\partial\HH$, by Lemma 
\ref{lem:stochastic_representation_CIR_properties} \eqref{item:CIRbetaGreaterEqualOne}. By
\cite[p. $117$]{Oksendal_2003}, both $\tau^{t,z}_{U}$ and $\nu^{t,z}_{U}$ are stopping times with respect to the 
filtration $\FF$, since $\FF$ is assumed to satisfy the usual
conditions. When the initial condition, $(t,z)$, is clear from the context, we omit the superscripts in the preceding 
definitions
\eqref{eq:stochastic_representation_Stopping_time_open_set} and 
\eqref{eq:stochastic_representation_Stopping_time_domain_with_degenerate_boundary} of the stopping times. Also,
when $t=0$, we omit the superscript $t$ in the preceding definitions.

\begin{thm} [Uniqueness of solutions to the elliptic boundary value problem]
\label{thm:stochastic_representation_Uniqueness_BVP_elliptic}
Let $r>0$, $q \geq 0$, and $f$ belong to $C(\sO)$ and obey the growth condition \eqref{eq:stochastic_representation_Growth_elliptic} on $\sO$. Then
\begin{enumerate}
\item
If $\beta\geq 1$, assume $g\in C_{\loc}(\Gamma_1)$ obeys \eqref{eq:stochastic_representation_Growth_elliptic}. Let
$$
u \in C_{\loc}(\sO\cup\Gamma_1) \cap C^2(\sO)
$$
be a solution to the elliptic boundary value problem \eqref{eq:stochastic_representation_HestonEllipticEqBVP},
\eqref{eq:stochastic_representation_HestonEllipticBoundaryCondition_Gamma1} and which obeys 
\eqref{eq:stochastic_representation_Growth_elliptic} on $\sO$. Then, $u=u^*$ on
$\sO\cup\Gamma_1$, where
\begin{equation}
\label{eq:stochastic_representation_StochasticRep_BVP_all_beta}
\begin{aligned}
u^*(z) := \EE^z_\QQ\left[e^{-r\tau_{\sO}} g(Z(\tau_{\sO})) 1_{\{\tau_{\sO} < \infty\}}\right]
          + \EE^z_\QQ\left[\int_{0}^{\tau_{\sO}} e^{-rs} f(Z(s))\,ds \right],
\end{aligned}
\end{equation}
where $\tau_{\sO}$ is defined by \eqref{eq:stochastic_representation_Stopping_time_open_set}, for all $z\in 
\sO\cup\Gamma_1$.
\item
If $0<\beta<1$, assume $g\in C_{\loc}(\partial\sO)$ obeys \eqref{eq:stochastic_representation_Growth_elliptic} on 
$\partial\sO$, and let $u \in C_{\loc}(\bar\sO) \cap C^2(\sO)$
be a solution to the elliptic boundary value problem \eqref{eq:stochastic_representation_HestonEllipticEqBVP},
\eqref{eq:stochastic_representation_HestonEllipticBoundaryCondition_whole_boundary} and which obeys
\eqref{eq:stochastic_representation_Growth_elliptic} on $\sO$. Then, $u=u^*$ on $\bar\sO$, where $u^*$ is given by
\eqref{eq:stochastic_representation_StochasticRep_BVP_all_beta}.
\end{enumerate}
\end{thm}

Following \cite[Definition 2.2]{Daskalopoulos_Feehan_optimalregstatheston},
we let $C^{1,1}_{s,\loc}(\sO\cup\Gamma_0)$ denote the subspace of $C^{1,1}(\sO)\cap C^{1}_{\loc}(\sO\cup\Gamma_0)$ 
consisting of functions, $u$, such that, for any precompact open subset $U\Subset \sO \cup \Gamma_0$,
\begin{equation}
\label{eq:stochastic_representation_C_1_1_regularity}
\sup_{(x,y)\in U} |u(x,y)| + |Du(x,y)| + |yD^2u(x,y)| <\infty,
\end{equation}
where $Du$ denotes the gradient and $D^2 u$ the Hessian matrix of $u$, defined Lebesgue-a.e. on $\sO$.

\begin{rmk}[On the definition of $C^{1,1}_{s,\loc}(\sO\cup\Gamma_0)$]
The definition of the function space $C^{1,1}_{s,\loc}(\sO\cup\Gamma_0)$ can be relaxed in our article from that of \cite[Definition 2.2]{Daskalopoulos_Feehan_optimalregstatheston}. It is enough to say that a function $u$ belongs to $C^{1,1}_{s,\loc}(\sO\cup\Gamma_0)$ if $u \in C^{1,1}(\sO) \cap C_{\loc}(\sO\cup\Gamma_0)$ and \eqref{eq:stochastic_representation_C_1_1_regularity} holds on any precompact open subset $U\Subset \sO \cup \Gamma_0$.
\end{rmk}

\begin{thm} [Uniqueness of solutions to the elliptic boundary value problem 
\eqref{eq:stochastic_representation_HestonEllipticEqBVP},
\eqref{eq:stochastic_representation_HestonEllipticBoundaryCondition_Gamma1}, when $0<\beta<1$]
\label{thm:stochastic_representation_Uniqueness_BVP_elliptic_beta_less_than_1}
Let $r>0$, $q \geq 0$, $0<\beta<1$, and let $f$ be as in Theorem 
\ref{thm:stochastic_representation_Uniqueness_BVP_elliptic}. Let $g\in C_{\loc}(\Gamma_1)$ obey
\eqref{eq:stochastic_representation_Growth_elliptic} on $\Gamma_1$ and suppose that
$$
u \in C_{\loc}(\sO\cup\Gamma_1) \cap C^2(\sO)\cap C^{1,1}_{s,\loc}(\sO\cup\Gamma_0)
$$
is a solution to the elliptic boundary value problem \eqref{eq:stochastic_representation_HestonEllipticEqBVP},
\eqref{eq:stochastic_representation_HestonEllipticBoundaryCondition_Gamma1} which obeys 
\eqref{eq:stochastic_representation_Growth_elliptic} on $\sO$.
Then, $u=u^*$ on $\sO\cup\Gamma_1$, where $u^*$ is given by
\begin{equation}
\label{eq:stochastic_representation_StochasticRep_BVP_2}
\begin{aligned}
u^*(z) := \EE^z_\QQ\left[e^{-r\nu_{\sO}} g(Z(\nu_{\sO})) 1_{\{\nu_{\sO} < \infty\}}\right]
          + \EE^z_\QQ\left[\int_{0}^{\nu_{\sO}} e^{-rs} f(Z(s))\,ds \right],
\end{aligned}
\end{equation}
and $\nu_{\sO}$ is defined by \eqref{eq:stochastic_representation_Stopping_time_domain_with_degenerate_boundary}, for all 
$z\in \sO\cup\Gamma_1$.
\end{thm}

\begin{rmk}[Existence and uniqueness of strong solutions in weighted Sobolev spaces to the elliptic boundary value 
problem]
Existence and uniqueness of strong solutions in weighted Sobolev spaces to problem 
\eqref{eq:stochastic_representation_HestonEllipticEqBVP} with boundary condition
\eqref{eq:stochastic_representation_HestonEllipticBoundaryCondition_Gamma1} along $\Gamma_1$, for all $\beta>0$, is proved 
in \cite[Theorem
1.18]{Daskalopoulos_Feehan_statvarineqheston}, and H\"older continuity of such solutions up to $\Gamma_0$ is proved in 
\cite[Theorem 1.10]{Feehan_Pop_regularityweaksoln}.
\end{rmk}

\begin{rmk}[Comparison of uniqueness results]
To obtain uniqueness of solutions to the elliptic boundary value problem 
\eqref{eq:stochastic_representation_HestonEllipticEqBVP} with boundary condition
\eqref{eq:stochastic_representation_HestonEllipticBoundaryCondition_Gamma1} only specified along $\Gamma_1$, we need to 
assume the stronger regularity hypothesis
$$
u \in C_{\loc}(\sO\cup\Gamma_1)\cap C^2(\sO) \cap C^{1,1}_{s,\loc}(\sO\cup\Gamma_0)
$$
when $0<\beta<1$, while the regularity assumption
$$
u \in C_{\loc}(\sO\cup\Gamma_1)\cap C^2(\sO)
$$
suffices when $\beta \geq 1$. The analogous comments apply to the elliptic obstacle problems described in Theorems 
\ref{thm:stochastic_representation_Uniqueness_elliptic_OP} and
\ref{thm:stochastic_representation_Uniqueness_elliptic_OP_beta_less_than_1}, the parabolic terminal/boundary value 
problems described in Theorems
\ref{thm:stochastic_representation_Uniqueness_BVP_parabolic} and 
\ref{thm:stochastic_representation_Uniqueness_BVP_parabolic_beta_less_than_1}, and the parabolic obstacle value
problems described in Theorems \ref{thm:stochastic_representation_Uniqueness_parabolic_OP} and 
\ref{thm:stochastic_representation_Uniqueness_parabolic_OP_beta_less_than_1}.
\end{rmk}

For $\alpha \in (0,1)$, we let $C^{k+\alpha}(\sO)$ denote the subspace of $C^{k}(\sO)$ consisting of functions whose 
derivatives up to order $k$ are \emph{locally}
$\alpha$-H\"older continuous on $\sO$ (in the sense of \cite[p. 52]{GilbargTrudinger}) and let $C^{k+\alpha}(\bar \sO)$ 
denote the subspace of $C^{k}(\bar\sO)$ consisting of
functions whose derivatives up to order $k$ are \emph{uniformly} $\alpha$-H\"older continuous on $\sO$ \cite[p. 
52]{GilbargTrudinger}, \cite[\S 1.27]{Adams}. If $T \subsetneqq
\partial\sO$ is a relatively open set, we let $C^{k+\alpha}_{\loc}(\sO\cup T)$ denote the vector space of functions, $u$, 
such that, for any precompact open subset $U\Subset \sO
\cup T$, we have $u \in C^{k+\alpha}(\bar U)$.

We have the following result concerning \emph{existence} of solutions to the elliptic boundary value problem with traditional 
regularity on $\sO$.

\begin{thm} [Existence of solutions to the elliptic boundary value problem 
with continuous Dirichlet boundary condition]
\label{thm:stochastic_representation_Existence_elliptic_BVP_beta}
In addition to the hypotheses of Theorem \ref{thm:stochastic_representation_Uniqueness_BVP_elliptic}, assume that the 
domain $\sO\subset \HH$ has boundary portion $\Gamma_1$
which satisfies the exterior sphere condition, and that $f \in C^{\alpha}(\sO)$.
\begin{enumerate}
\item
If $\beta\geq 1$ and also $g\in C_{\loc}(\bar\Gamma_1)$, then the function $u^*$ in 
\eqref{eq:stochastic_representation_StochasticRep_BVP_all_beta} is a solution to problem
\eqref{eq:stochastic_representation_HestonEllipticEqBVP} with boundary condition 
\eqref{eq:stochastic_representation_HestonEllipticBoundaryCondition_Gamma1} along $\Gamma_1$.
In particular, $u^* \in C_{\loc}(\sO \cup \Gamma_1) \cap C^{2+\alpha}(\sO)$ and $u^*$ satisfies the growth assumption 
\eqref{eq:stochastic_representation_Growth_elliptic}.
\item
If $0<\beta< 1$ and also $g\in C_{\loc}(\partial\sO)$, then the function $u^*$ in 
\eqref{eq:stochastic_representation_StochasticRep_BVP_all_beta} is a solution to problem
\eqref{eq:stochastic_representation_HestonEllipticEqBVP} with boundary condition 
\eqref{eq:stochastic_representation_HestonEllipticBoundaryCondition_whole_boundary} along
$\partial\sO$. In particular,
$u^* \in C_{\loc}(\bar\sO) \cap C^{2+\alpha}(\sO)$ and $u^*$ satisfies the growth assumption 
\eqref{eq:stochastic_representation_Growth_elliptic}.
\end{enumerate}
\end{thm}

We next have existence of solutions to the elliptic boundary value problem when the boundary data $g$ is $C^{2+\alpha}$ up to $\Gamma_1\subset\partial\sO$.

\begin{thm} [Existence of solutions to the elliptic boundary value problem with $C^{2+\alpha}$ Dirichlet boundary condition]
\label{thm:stochastic_representation_Existence_elliptic_BVP_beta_Holder_continuous_boundary_data}
In addition to the hypotheses of Theorem \ref{thm:stochastic_representation_Uniqueness_BVP_elliptic}, let $\sO\subset \HH$ 
be a domain such that the boundary portion $\Gamma_1$
is of class $C^{2+\alpha}$, that $f \in C^{\alpha}_{\loc}(\sO\cup\Gamma_1)$ and $g\in 
C^{2+\alpha}_{\loc}(\sO\cup\Gamma_1)$.
\begin{enumerate}
\item
If $\beta\geq 1$, then $u^*$, given by \eqref{eq:stochastic_representation_StochasticRep_BVP_all_beta}, is a solution to 
problem
\eqref{eq:stochastic_representation_HestonEllipticEqBVP} with boundary condition 
\eqref{eq:stochastic_representation_HestonEllipticBoundaryCondition_Gamma1} along $\Gamma_1$.
In particular,
$$
u^* \in C^{2+\alpha}_{\loc}(\sO \cup \Gamma_1)
$$
and satisfies the growth assumption \eqref{eq:stochastic_representation_Growth_elliptic}.
\item
If $0<\beta< 1$ and $g\in C_{\loc}(\partial\sO)$, then $u^*$, given by 
\eqref{eq:stochastic_representation_StochasticRep_BVP_all_beta}, is a solution to problem
\eqref{eq:stochastic_representation_HestonEllipticEqBVP} with boundary condition 
\eqref{eq:stochastic_representation_HestonEllipticBoundaryCondition_whole_boundary} along
$\partial\sO$. In particular,
$$
u^* \in C_{\loc}(\bar\sO) \cap C^{2+\alpha}(\sO\cup\Gamma_1)
$$
and satisfies the growth assumption \eqref{eq:stochastic_representation_Growth_elliptic}.
\end{enumerate}
\end{thm}

\begin{rmk}[Existence of solutions with Daskalopoulos-Hamilton-K\"och H\"older regularity]
\label{rmk:EllipticExistenceDHKHolder}
When $f\in C^{\alpha}_s(\sO\cup\Gamma_0)$ and $g=0$, we establish in \cite[Theorem 1.11 \& Corollary 
1.13]{Feehan_Pop_elliptichestonschauder} that
the solutions to the elliptic boundary value problem \eqref{eq:stochastic_representation_HestonEllipticEqBVP},
\eqref{eq:stochastic_representation_HestonEllipticBoundaryCondition_Gamma1} lie in
$C^{2+\alpha}_s(\sO\cup\Gamma_0)\cap C_{\loc}(\bar\sO)$ for all $\beta>0$, where $C^{2+\alpha}_s(\sO\cup\Gamma_0)$ is an 
elliptic analogue of the parabolic
Daskalopoulos-Hamilton-K\"och H\"older spaces described in \cite{DaskalHamilton1998, Koch}. A function $u\in 
C^{2+\alpha}_s(\sO\cup\Gamma_0)$ has the property that $u, Du, yD^2u$
are $C^{\alpha}_s$ continuous up to $\Gamma_0$ and $yD^2u=0$ on $\Gamma_0$, where $C^{\alpha}_s(\sO\cup\Gamma_0)$ is 
defined by analogy with the traditional definition of
$C^{\alpha}(\sO)$ in \cite{GilbargTrudinger}, except that Euclidean distance between points in $\sO$ is replaced by the 
cycloidal distance function.
\end{rmk}

\begin{rmk}[Existence and uniqueness of solutions to elliptic boundary value problems]
\label{rmk:ClassicalEllipticPDEResults}
Existence and uniqueness of solutions to the elliptic boundary value problem 
\eqref{eq:stochastic_representation_HestonEllipticEqBVP} and
\eqref{eq:stochastic_representation_HestonEllipticBoundaryCondition_Gamma1}, \emph{provided} $\Gamma_1=\partial\sO$, 
follow from Schauder methods when the coefficient matrix,
$(a^{ij})$, of the second-order derivatives in $A$ is uniformly elliptic. For example, see \cite[Theorem 
6.13]{GilbargTrudinger} for the case where $\sO$ is bounded and $f$ and
the coefficients of $A$ are bounded and in $C^\alpha(\sO)$, $\alpha\in(0,1)$, giving a unique solution $u\in 
C^{2+\alpha}(\sO)\cap C(\bar\sO)$, while \cite[Theorem
6.14]{GilbargTrudinger} gives $u\in C^{2+\alpha}(\bar\sO)$ when $f$ and the coefficients of $A$ are in 
$C^\alpha(\bar\sO)$. See \cite[Corollary 7.4.4]{Krylov_LecturesHolder},
together with \cite[Corollary 7.4.9]{Krylov_LecturesHolder} or \cite[Theorem 7.6.4]{Krylov_LecturesHolder} or 
\cite[Theorem 7.6.5 \& Remark 7.6.6]{Krylov_LecturesHolder}, for
similar statements.
\end{rmk}

\subsubsection{Uniqueness of solutions to elliptic obstacle problems}
For $\theta_1,\theta_2 \in \sT$, we set
\begin{equation}
\label{eq:stochastic_representation_Functional_J}
\begin{aligned}
J_e^{\theta_1, \theta_2}(z)
& := \EE^z_\QQ\left[\int_{0}^{\theta_1\wedge\theta_2} e^{-rs} f(Z(s))\,ds \right] \\
&\quad
  + \EE^z_\QQ\left[e^{-r\theta_1} g(Z(\theta_1)) 1_{\{ \theta_1 \leq \theta_2 \}}\right]
  + \EE^z_\QQ\left[ e^{-r\theta_2} \psi(Z(\theta_2)) 1_{\{ \theta_2 < \theta_1 \}} \right].
\end{aligned}
\end{equation}
We then have the

\begin{thm}[Uniqueness of solutions to the elliptic obstacle problem]
\label{thm:stochastic_representation_Uniqueness_elliptic_OP}
Let $r>0$, $q \geq 0$, and $f$ be as in Theorem \ref{thm:stochastic_representation_Uniqueness_BVP_elliptic}, and $\psi$ belong to $C(\sO)$ and satisfy
\eqref{eq:stochastic_representation_Growth_elliptic} on $\sO$.
\begin{enumerate}
\item
If $\beta\geq 1$, let $\psi \in C_{\loc}(\sO\cup\Gamma_1)$ and $g\in C_{\loc}(\Gamma_1)$ obey 
\eqref{eq:stochastic_representation_Growth_elliptic} and
\eqref{eq:stochastic_representation_Elliptic_compatibility_g_psi_beta_geq_1} on $\Gamma_1$. Let
$$
u \in C_{\loc}(\sO\cup\Gamma_1) \cap C^2(\sO)
$$
be a solution to the elliptic obstacle problem \eqref{eq:stochastic_representation_Elliptic_obstacle_problem},
\eqref{eq:stochastic_representation_HestonEllipticBoundaryCondition_Gamma1}
such that $u$ and $Au$ obey \eqref{eq:stochastic_representation_Growth_elliptic} on $\sO$. Then, $u=u^*$ on 
$\sO\cup\Gamma_1$, where $u^*$ is given by
\begin{equation}
\label{eq:stochastic_representation_Stochastic_representation_EOP_1}
u^*(z) := \sup_{\theta \in \mathscr{T}} J_e^{\tau_{\sO},\theta}(z),
\end{equation}
and $\tau_{\sO}$ is defined by \eqref{eq:stochastic_representation_Stopping_time_open_set}, for all $z\in 
\sO\cup\Gamma_1$.
\item
If $0<\beta<1$, let $\psi \in C_{\loc}(\bar \sO)$ and $g\in C_{\loc}(\partial\sO)$ obey 
\eqref{eq:stochastic_representation_Growth_elliptic} and
\eqref{eq:stochastic_representation_Elliptic_compatibility_g_psi_beta_leq_1} on $\partial\sO$. Let
$$
u \in C_{\loc}(\bar\sO) \cap C^2(\sO)
$$
be a solution to the elliptic obstacle problem \eqref{eq:stochastic_representation_Elliptic_obstacle_problem},
\eqref{eq:stochastic_representation_HestonEllipticBoundaryCondition_whole_boundary},
such that $u$ and $Au$ obey \eqref{eq:stochastic_representation_Growth_elliptic} on $\sO$. Then, $u=u^*$ on $\bar\sO$, 
where $u^*$ is given by
\eqref{eq:stochastic_representation_Stochastic_representation_EOP_1}.
\end{enumerate}
\end{thm}

\begin{thm}[Uniqueness of solutions to the elliptic obstacle problem 
\eqref{eq:stochastic_representation_Elliptic_obstacle_problem},
\eqref{eq:stochastic_representation_HestonEllipticBoundaryCondition_Gamma1}, when $0<\beta<1$]
\label{thm:stochastic_representation_Uniqueness_elliptic_OP_beta_less_than_1}
Let $r>0$, $q \geq 0$, $0<\beta<1$, and $f$ be as in Theorem \ref{thm:stochastic_representation_Uniqueness_elliptic_OP}. 
Let $\psi \in C_{\loc}(\sO\cup\Gamma_1)$ obey
\eqref{eq:stochastic_representation_Growth_elliptic} on $\sO$ and let $g\in C_{\loc}(\Gamma_1)$ obey 
\eqref{eq:stochastic_representation_Growth_elliptic} and
\eqref{eq:stochastic_representation_Elliptic_compatibility_g_psi_beta_geq_1} on $\Gamma_1$. If
$$
u \in C_{\loc}(\sO\cup\Gamma_1) \cap C^2(\sO)\cap C^{1,1}_{s,\loc}(\sO\cup\Gamma_0)
$$
is a solution to the elliptic obstacle problem \eqref{eq:stochastic_representation_Elliptic_obstacle_problem},
\eqref{eq:stochastic_representation_HestonEllipticBoundaryCondition_Gamma1} such that $u$ and $Au$ obey 
\eqref{eq:stochastic_representation_Growth_elliptic},
then $u=u^*$ on $\sO\cup\Gamma_1$, where $u^*$ is given by
\begin{equation}
\label{eq:stochastic_representation_Stochastic_representation_EOP_2}
u^*(z) := \sup_{\theta \in \mathscr{T}} J_e^{\nu_{\sO},\theta}(z),
\end{equation}
and $\nu_{\sO}$ is defined by \eqref{eq:stochastic_representation_Stopping_time_domain_with_degenerate_boundary}, for all 
$z\in \sO\cup\Gamma_1$.
\end{thm}

\begin{rmk}[Existence and uniqueness of strong solutions in weighted Sobolev spaces to the elliptic obstacle problem]
Existence and uniqueness of strong solutions in weighted Sobolev spaces to problem 
\eqref{eq:stochastic_representation_Elliptic_obstacle_problem} with Dirichlet boundary
condition \eqref{eq:stochastic_representation_HestonEllipticBoundaryCondition_Gamma1} along $\Gamma_1$, for all $\beta>0$, 
is proved in \cite[Theorem
1.6]{Daskalopoulos_Feehan_statvarineqheston}, and H\"older continuity of such solutions up to boundary portion $\Gamma_0$ 
is proved in \cite[Theorem
1.13]{Feehan_Pop_regularityweaksoln}.
\end{rmk}

\subsubsection{Existence and uniqueness of solutions to parabolic terminal/boundary value problems}
We shall need to appeal to the following analogue of Hypothesis \ref{hyp:stochastic_representation_Growth_elliptic}:

\begin{hyp}[Growth condition]
If $v$ is a function then, for all $(t,x,y)$ in its domain of definition,
\begin{equation}
\label{eq:stochastic_representation_Growth_parabolic}
|v(t,x,y)| \leq C(1 + e^{M_1 y}+e^{M_2x}),
\end{equation}
where $C>0$, $0 \leq M_1 <  \mu$, and $M_2 \in[0,1]$.
\end{hyp}

We let $C(Q)$ denote the vector space of continuous functions on $Q$, while $C(\bar Q)$ denotes the Banach space of 
functions which are \emph{uniformly continuous} and
\emph{bounded} on $Q$. We let $Du$ denote the gradient and let $D^2 u$ denote the Hessian matrix of a function $u$ on $Q$ 
with respect to spatial variables. We let $C^1(Q)$
denote the vector space of functions, $u$, such that $u$, $u_t$, and $Du$ are continuous on $Q$, while $C^1(\bar Q)$ 
denotes the Banach space of functions, $u$, such that $u$,
$u_t$, and $Du$ are uniformly continuous and bounded on $Q$; finally, $C^{2}(Q)$ denotes the vector space of functions, 
$u$, such that $u_t$, $Du$, and $D^2u$ are continuous $Q$,
while $C^2(\bar Q)$ denotes the Banach space of functions, $u$, such that $u$, $u_t$, $Du$, and $D^2u$ are uniformly 
continuous and bounded on $Q$. If $T \subsetneqq \partial Q$
is a relatively open set, we let $C_{\loc}(Q \cup T)$ denote the vector space of functions, $u$, such that, for any 
precompact open subset $V\Subset
Q \cup T$, we have $u \in C(\bar V)$.

\begin{thm} [Uniqueness of solutions to the parabolic terminal/boundary value problem]
\label{thm:stochastic_representation_Uniqueness_BVP_parabolic}
Let $f$ belong to $C(Q)$ and obey \eqref{eq:stochastic_representation_Growth_parabolic}. Then
\begin{enumerate}
\item
If $\beta\geq 1$, assume $g\in C_{\loc}(\eth^1 Q)$ obeys \eqref{eq:stochastic_representation_Growth_parabolic} on 
$\eth^1 Q$. Let
$$
u \in C_{\loc}(Q\cup\eth^1 Q) \cap C^2(Q)
$$
be a solution to the parabolic terminal/boundary value problem \eqref{eq:stochastic_representation_Parabolic_BVP},
\eqref{eq:stochastic_representation_Parabolic_BVP_boundary_condition_Gamma_1} which obeys 
\eqref{eq:stochastic_representation_Growth_parabolic} on $Q$. Then, $u=u^*$ on
$Q\cup\eth^1 Q$, where $u^*$ is given by
\begin{equation}
\label{eq:stochastic_representation_Stochastic_representation_parabolic_BVP_1}
\begin{aligned}
u^*(t, z)
&:= \EE^{t,z}_\QQ\left[\int_{t}^{\tau_{\sO}\wedge T} e^{-r(s-t)} f(s,Z(s))\,ds \right]\\
&\quad + \EE^{t,z}_\QQ\left[e^{-r(\tau_{\sO}\wedge T-t)} g(\tau_{\sO}\wedge T, Z(\tau_{\sO}\wedge T))\right],
\end{aligned}
\end{equation}
and $\tau_{\sO}$ is defined by \eqref{eq:stochastic_representation_Stopping_time_open_set}, for all $(t,z)\in 
Q\cup\eth^1 Q$.
\item
If $0<\beta<1$, assume $g\in C_{\loc}(\eth Q)$ obeys \eqref{eq:stochastic_representation_Growth_parabolic} on $\eth Q$, 
and let
$$
u \in C_{\loc}(Q\cup\eth Q) \cap C^2(Q)
$$
be a solution to the parabolic terminal/boundary value problem \eqref{eq:stochastic_representation_Parabolic_BVP},
\eqref{eq:stochastic_representation_Parabolic_BVP_boundary_condition_whole_boundary} which obeys 
\eqref{eq:stochastic_representation_Growth_parabolic} on $Q$. Then, $u=u^*$ on
$Q\cup\eth Q$, where $u^*$ is given by \eqref{eq:stochastic_representation_Stochastic_representation_parabolic_BVP_1}.
\end{enumerate}
\end{thm}

By analogy with \cite[Definition 2.2]{Daskalopoulos_Feehan_optimalregstatheston},
we let $C^{1,1}_{s,\loc}((0,T)\times(\sO\cup\Gamma_0))$ denote the subspace of $C^{1,1}((0,T)\times\sO)\cap C^{1}_{\loc}((0,T)\times(\sO\cup\Gamma_0))$ consisting of functions, $u$, such that, for any precompact open subset $V\Subset (0,T)\times(\sO\cup\Gamma_0)$,
\begin{equation}
\label{eq:stochastic_representation_C_1_1_regularity_parabolic}
\sup_{(t,z)\in V} |u(t,z)| + |Du(t,z)| + |yD^2u(t,z)| <\infty.
\end{equation}

\begin{rmk}[On the definition of $C^{1,1}_{s,\loc}((0,T)\times(\sO\cup\Gamma_0)$]
The definition of the function space $C^{1,1}_{s,\loc}((0,T)\times \sO\cup\Gamma_0)$ can be relaxed in our paper from that implied by \cite[Definition 2.2]{Daskalopoulos_Feehan_optimalregstatheston}. It is enough to say that a function $u$ belongs to $C^{1,1}_{s,\loc}((0,T)\times\sO\cup\Gamma_0) \cap C_{\loc}((0,T)\times\sO\cup\Gamma_0)$ if $u \in C^{1,1}((0,T)\times \sO)$ and \eqref{eq:stochastic_representation_C_1_1_regularity_parabolic} holds on any precompact open subset $U\in (0,T)\times\Subset \sO \cup \Gamma_0$.
\end{rmk}

We have the following alternative uniqueness result.

\begin{thm} [Uniqueness of solutions to the parabolic terminal/boundary value problem when $0<\beta<1$]
\label{thm:stochastic_representation_Uniqueness_BVP_parabolic_beta_less_than_1}
Let $0<\beta<1$ and $f$ be as in Theorem \ref{thm:stochastic_representation_Uniqueness_BVP_parabolic}. Let $g\in 
C_{\loc}(\eth^1 Q)$ obey
\eqref{eq:stochastic_representation_Growth_parabolic} on $\eth^1 Q$, and
$$
u \in C_{\loc}(Q\cup\eth^1 Q) \cap C^2(Q)\cap C^{1,1}_{s,\loc}((0,T)\times(\sO\cup\Gamma_0))
$$
be a solution to the parabolic terminal/boundary value problem \eqref{eq:stochastic_representation_Parabolic_BVP},
\eqref{eq:stochastic_representation_Parabolic_BVP_boundary_condition_Gamma_1} which obeys 
\eqref{eq:stochastic_representation_Growth_parabolic} on $Q$.
Then, $u=u^*$ on $Q\cup\eth^1 Q$, where $u^*$ is given by
\begin{equation}
\label{eq:stochastic_representation_Stochastic_representation_parabolic_BVP_2}
\begin{aligned}
u^*(t, z)
&:= \EE^{t,z}_\QQ\left[\int_{t}^{\nu_{\sO}\wedge T} e^{-r(s-t)} f(s,Z(s))\,ds \right]\\
&\quad  + \EE^{t,z}_\QQ\left[e^{-r(\nu_{\sO}\wedge T-t)} g(\nu_{\sO}\wedge T, Z(\nu_{\sO}\wedge T)) \right],
\end{aligned}
\end{equation}
and $\nu_{\sO}$ is defined by \eqref{eq:stochastic_representation_Stopping_time_domain_with_degenerate_boundary}, for all 
$(t,z)\in Q\cup\eth^1 Q$.
\end{thm}

\begin{rmk}[Existence and uniqueness of strong solutions in weighted Sobolev spaces to the parabolic terminal/boundary 
value problem]
Existence and uniqueness of strong solutions in weighted Sobolev spaces to problem 
\eqref{eq:stochastic_representation_Parabolic_BVP} with Dirichlet boundary condition
\eqref{eq:stochastic_representation_Parabolic_BVP_boundary_condition_Gamma_1} along $\eth^1 Q$, for all $\beta>0$, is 
proved in \cite{Daskalopoulos_Feehan_evolvarineqheston}.
\end{rmk}

\begin{rmk}[Growth of solutions to parabolic terminal/boundary value problems]
Karatzas and Shreve allow faster growth of solutions when the growth on the coefficients of the differential operator is 
constrained \cite[Theorem $4.4.2$ \& Problem
$5.7.7$]{KaratzasShreve1991}, and polynomial growth of solutions is allowed for linear growth coefficients and source 
function $f$ with at most polynomial growth \cite[Theorem
$5.7.6$]{KaratzasShreve1991}.
\end{rmk}

\begin{rmk}[Barrier option pricing and discontinuous terminal/boundary conditions]
\label{rmk:stochastic_representation_European_option_prices}
In applications to finance, $\sO$ will often be a rectangle, $(x_0,x_1)\times(0,\infty)$, where $-\infty\leq 
x_0<x_1\leq\infty$; the growth exponents will be $M_1=0$ and $M_2=1$
--- indeed, the source function $f$ will always be zero and the spatial boundary condition function $g:(0,T)\times\Gamma_1 
\rightarrow\RR$ will often be zero. However, the
spatial boundary condition, $g:(0,T)\times\Gamma_1 \rightarrow\RR$, and terminal condition, $g:\{T\}\times\bar\sO 
\rightarrow\RR$, may be \emph{discontinuous} where they meet
along $\{T\}\times\partial\sO$, as in the case of the \emph{down-and-out put}, with
$$
g(t,x,y) = \begin{cases} 0, & 0<t<T, x=x_0, y>0,\\ (K-e^x)^+ & t=T, x_0<x<\infty, y>0,\end{cases}
$$
where $g$ is discontinuous at $(T,x_0,y)$ if $K-e^{x_0}>0$, that is, $x_0<\log K$. We shall consider the question of 
establishing stochastic representations for solutions to
parabolic terminal/value problems (European-style option prices) or parabolic obstacle problems (American-style option 
prices) with discontinuous data elsewhere.
\end{rmk}

For $\alpha \in (0,1)$, we let $C^{\alpha}(Q)$ denote the subspace of $C(Q)$ consisting of \emph{locally} 
$\alpha$-H\"older continuous functions, $u$, on $Q$, that is, for any
precompact open set $V\Subset Q$,
\begin{equation}
\label{eq:stochastic_representation_Holder_seminorm}
\left[u\right]_{C^{\alpha}(V)}:=\sup_{\substack{(t^i,z^i) \in V \\ i=1,2}} 
\frac{|u(t^1,z^1)-u(t^2,z^2)|}{\left(|z^1-z^2|+\sqrt{|t^1-t^2|}\right)^{\alpha}} < \infty,
\end{equation}
and we let $C^{\alpha}(\bar Q) \subset C(\bar Q)$ denote the Banach space of functions, $u$, which are \emph{uniformly} 
$\alpha$-H\"older continuous on $Q$, that is
$$
\left[u\right]_{C^{\alpha}(Q)} < \infty.
$$
When $Q$ is unbounded, we let $C^{\alpha}_{\loc}(\bar Q)$ denote the subspace of $C^{\alpha}(Q)$ consisting of functions, 
$u$, such that, for any precompact open set $V\Subset
\bar Q$, we have
$$
\left[u\right]_{C^{\alpha}(V)} < \infty.
$$
We let $C^{2+\alpha}(Q)$ denote the subspace of $C^{2}(Q)$ consisting of functions, $u$, such that $u$, $u_t$, and the 
components of $Du$ and $D^2u$ belong to $C^{\alpha}(Q)$,
and let $C^{2+\alpha}(\bar Q) \subset C^{2}(\bar Q)$ denote the Banach space of functions, $u$, such that $u$, $u_t$, and 
the components of $Du$ and $D^2u$ belong to
$C^{\alpha}(\bar Q)$.

\begin{thm}[Existence of solutions to the parabolic terminal/boundary value problem with continuous Dirichlet boundary condition]
\label{thm:stochastic_representation_Existence_parabolic_BVP_beta}
In addition to the hypotheses of Theorem \ref{thm:stochastic_representation_Uniqueness_BVP_parabolic}, let $\sO\subset 
\HH$ be a domain such that the boundary $\Gamma_1$ obeys an
exterior sphere condition, and $f \in C^{\alpha}_{\loc}(\bar Q)$.
\begin{enumerate}
\item
If $\beta\geq 1$ and $g \in C_{\loc}(\overline{\eth^1 Q})$, then $u^*$ in 
\eqref{eq:stochastic_representation_Stochastic_representation_parabolic_BVP_1} is a solution to problem \eqref{eq:stochastic_representation_Parabolic_BVP} with boundary condition \eqref{eq:stochastic_representation_Parabolic_BVP_boundary_condition_Gamma_1}. 
In particular, $u^* \in C_{\loc}( Q \cup \eth^1 Q) \cap 
C^{2+\alpha}(Q)$ and obeys the growth assumption
\eqref{eq:stochastic_representation_Growth_parabolic}.
\item
If $0<\beta< 1$ and $g \in C_{\loc}(\overline{\eth Q})$, then $u^*$ in 
\eqref{eq:stochastic_representation_Stochastic_representation_parabolic_BVP_1} is a solution to problem \eqref{eq:stochastic_representation_Parabolic_BVP} with boundary condition \eqref{eq:stochastic_representation_Parabolic_BVP_boundary_condition_whole_boundary}. 
In particular, $u^* \in C_{\loc}(Q \cup \eth Q) \cap 
C^{2+\alpha}(Q)$ and satisfies the growth assumption
\eqref{eq:stochastic_representation_Growth_parabolic}.
\end{enumerate}
\end{thm}

For $T \subsetneqq \eth Q$ a relatively open subset, we let $C^{2+\alpha}_{\loc}(Q \cup T)$ denote the subspace of 
$C^{2+\alpha}(Q)$ such that, for any precompact open set $U
\Subset Q \cup T$, we have $u \in C^{2+\alpha}(\bar U)$.

\begin{thm}[Existence of solutions to the parabolic terminal/boundary value problem with $C^{2+\alpha}$
Dirichlet boundary condition]
\label{thm:stochastic_representation_Existence_parabolic_BVP_beta_with_Holder_continuous_boundary_data}
In addition to the hypotheses of Theorem \ref{thm:stochastic_representation_Uniqueness_BVP_parabolic}, let $\sO\subset 
\HH$ be a domain such that the boundary portion $\Gamma_1$ is of class $C^{2+\alpha}$.
\begin{enumerate}
\item
If $\beta\geq 1$ and  $g \in C^{2+\alpha}_{\loc}(Q\cup\eth^1 Q)$ obeys
\begin{equation}
\label{eq:Parabolic_compatibility_f_g_Gamma1}
-g_t+Ag = f \hbox{  on  } \{T\}\times\Gamma_1.
\end{equation}
Then $u^*$ in \eqref{eq:stochastic_representation_Stochastic_representation_parabolic_BVP_1} is a solution to 
problem \eqref{eq:stochastic_representation_Parabolic_BVP} with boundary condition \eqref{eq:stochastic_representation_Parabolic_BVP_boundary_condition_Gamma_1}. 
In particular,
$$
u^* \in C^{2+\alpha}_{\loc}(Q \cup \eth^1 Q)
$$
and obeys the growth estimate \eqref{eq:stochastic_representation_Growth_parabolic}.
\item
If $0<\beta< 1$ and $g\in C^{2+\alpha}_{\loc}(Q \cup\eth^1 Q) \cap C_{\loc}(\bar Q)$ obeys
\begin{equation}
\label{eq:Parabolic_compatibility_f_g}
-g_t+Ag = f \hbox{  on  } \{T\}\times\partial\sO.
\end{equation}
Then $u^*$ in \eqref{eq:stochastic_representation_Stochastic_representation_parabolic_BVP_1} is a solution to 
problem \eqref{eq:stochastic_representation_Parabolic_BVP} with boundary condition \eqref{eq:stochastic_representation_Parabolic_BVP_boundary_condition_whole_boundary}. 
In particular,
$$
u^* \in C^{2+\alpha}_{\loc}(Q \cup\eth^1 Q) \cap C_{\loc}(\bar Q).
$$
and obeys the growth estimate \eqref{eq:stochastic_representation_Growth_parabolic}.
\end{enumerate}
\end{thm}

\begin{rmk}[Zero and first-order compatibility conditions for parabolic equations]
The conditions \eqref{eq:Parabolic_compatibility_f_g_Gamma1} and \eqref{eq:Parabolic_compatibility_f_g} are the analogues 
of the first-order compatibility condition
\cite[Equation (10.4.3)]{Krylov_LecturesHolder}. The analogue of the zero-order compatibility condition in \cite[Equation 
(10.4.2)]{Krylov_LecturesHolder} automatically holds at
$\{T\}\times\Gamma_1$ or $\{T\}\times\partial\sO$, since we always choose $h=g(T,\cdot)$ on $\Gamma_1$ or $\partial\sO$, 
respectively, in this article.
\end{rmk}

\begin{rmk}[Existence of solutions with Daskalopoulos-Hamilton-K\"och H\"older regularity]
\label{rmk:ParabolicExistenceDHKHolder}
As in the elliptic case, the solutions to the parabolic terminal/boundary value problem 
\eqref{eq:stochastic_representation_Parabolic_BVP},
\eqref{eq:stochastic_representation_Parabolic_BVP_boundary_condition_Gamma_1} should lie in $C^{2+\alpha}_s(\underline 
Q)\cap C_{\loc}(\bar Q)$ for all $\beta>0$, where
$C^{2+\alpha}_s(\underline Q)$ is the parabolic Daskalopoulos-Hamilton-K\"och H\"older space described in 
\cite{DaskalHamilton1998, Koch}. A function
$u\in C^{2+\alpha}_s(\underline Q)$ has the property that $u, Du, yD^2u$ are $C^{\alpha}_s$ continuous up to $\Gamma_0$ 
and $yD^2u=0$ on $(0,T)\times\Gamma_0$, where
$C^{\alpha}_s(\underline Q)$ is defined by analogy with the traditional definition \cite{Krylov_LecturesHolder} of 
$C^{\alpha}(Q)$, except that Euclidean distance between points
in $Q$ is replaced by the cycloidal distance function. When $Q=\HH\times(0,T)$, we establish this existence result in 
\cite[Theorem 1.1]{Feehan_Pop_mimickingdegen_pde}.
\end{rmk}

\begin{rmk}[Existence and uniqueness of solutions to parabolic terminal/boundary value problems]
\label{rmk:ClassicalParabolicPDEResults}
Existence and uniqueness of solutions to the parabolic terminal/boundary value problem 
\eqref{eq:stochastic_representation_Parabolic_BVP} and
\eqref{eq:stochastic_representation_Parabolic_BVP_boundary_condition_Gamma_1}, again \emph{provided} 
$\Gamma_1=\partial\sO$, follow from Schauder methods when the coefficient
matrix, $(a^{ij})$, of $A$ is strictly elliptic on $\bar\sO$. For example, see \cite[Theorems 5.9 \& 5.10]{Lieberman} for 
the case where $f$ and the coefficients of $A$ are
bounded and in $C^{\alpha}(Q)$, giving a unique solution $u\in C^{2+\alpha}(Q)\cap C(\bar Q)$.
\end{rmk}

\subsubsection{Uniqueness of solutions to parabolic obstacle problems}
For $\theta_1,\theta_2 \in \sT_{t,T}$, $0 \leq t \leq T$, we set
\begin{equation}
\label{eq:stochastic_representation_Functional_J_parabolic}
\begin{aligned}
J_p^{\theta_1,\theta_2}(t,z)
  &:= \EE^{t,z}_\QQ\left[\int_{t}^{\theta_1\wedge\theta_2} e^{-r(s-t)} f(s,Z(s))\,ds \right]
    + \EE^{t,z}_\QQ\left[e^{-r(\theta_2-t)} \psi(\theta_2,Z(\theta_2)) \mathbf{1}_{\{\theta_2 < \theta_1\}}\right]\\
  &\quad + \EE^{t,z}_\QQ\left[e^{-r(\theta_1-t)} g(\theta_1,Z(\theta_1)) \mathbf{1}_{\{\theta_1\leq\theta_2\}}\right].
\end{aligned}
\end{equation}
We have the following \emph{uniqueness} result of solutions to the parabolic obstacle problem with different possible 
boundary conditions, depending on the value of the parameter
$\beta>0$.

\begin{thm}[Uniqueness of solutions to the parabolic obstacle problem]
\label{thm:stochastic_representation_Uniqueness_parabolic_OP}
Let $f$ be as in Theorem \ref{thm:stochastic_representation_Uniqueness_BVP_parabolic} and $\psi$ belong to $C(Q)$ and satisfy
\eqref{eq:stochastic_representation_Growth_parabolic}.
\begin{enumerate}
\item
If $\beta\geq 1$, assume $\psi \in C_{\loc}(Q\cup\eth^1 Q)$ and $g\in C_{\loc}(\eth^1 Q)$ obeys 
\eqref{eq:stochastic_representation_Growth_parabolic} on $\eth^1 Q$ and
\eqref{eq:stochastic_representation_Parabolic_compatibility_g_psi_beta_geq_1}. Let
$$
u \in C_{\loc}(Q\cup\eth^1 Q) \cap C^2(Q)
$$
be a solution to the parabolic obstacle problem \eqref{eq:stochastic_representation_Parabolic_obstacle_problem},
\eqref{eq:stochastic_representation_Parabolic_BVP_boundary_condition_Gamma_1} such that $u$ and $Au$ obey 
\eqref{eq:stochastic_representation_Growth_parabolic} on $Q$. Then,
$u=u^*$ on $Q\cup\eth^1 Q$, where $u^*$ is given by
\begin{equation}
\label{eq:stochastic_representation_Stochastic_representation_POP_1}
u^*(t,z) := \sup_{\theta \in \sT_{t,T}} J_p^{\tau_{\sO}\wedge T,\theta}(t,z),
\end{equation}
and $\tau_{\sO}$ is defined by \eqref{eq:stochastic_representation_Stopping_time_open_set}, for all $(t,z)\in 
Q\cup\eth^1 Q$.
\item
If $0<\beta<1$, assume $\psi \in C_{\loc}(\bar Q)$ and $g\in C_{\loc}(\eth Q)$ obeys 
\eqref{eq:stochastic_representation_Growth_parabolic} on $\eth Q$ and
\eqref{eq:stochastic_representation_Parabolic_compatibility_g_psi_beta_leq_1}. Let
$$
u \in C_{\loc}( Q\cup \eth Q) \cap C^2(Q)
$$
be a solution to the parabolic obstacle problem \eqref{eq:stochastic_representation_Parabolic_obstacle_problem},
\eqref{eq:stochastic_representation_Parabolic_BVP_boundary_condition_whole_boundary} such that $u$ and $Au$ obey 
\eqref{eq:stochastic_representation_Growth_parabolic} on $Q$.
Then, $u=u^*$ on $ Q\cup \eth Q$, where $u^*$ is given by 
\eqref{eq:stochastic_representation_Stochastic_representation_POP_1}.
\end{enumerate}
\end{thm}

\begin{thm}[Uniqueness of solutions to the parabolic obstacle problem 
\eqref{eq:stochastic_representation_Parabolic_obstacle_problem},
\eqref{eq:stochastic_representation_Parabolic_BVP_boundary_condition_Gamma_1}, when $0<\beta<1$]
\label{thm:stochastic_representation_Uniqueness_parabolic_OP_beta_less_than_1}
Let $0<\beta<1$ and $f$ be as in Theorem \ref{thm:stochastic_representation_Uniqueness_BVP_parabolic}. Assume $\psi \in 
C_{\loc}(Q \cup \eth^1 Q)$, and $g\in C_{\loc}(\eth^1 Q)$
obey \eqref{eq:stochastic_representation_Growth_parabolic} on $\eth^1 Q$ and 
\eqref{eq:stochastic_representation_Parabolic_compatibility_g_psi_beta_geq_1}. Let
$$
u \in C_{\loc}(Q \cup \eth^1 Q) \cap C^2(Q)\cap C^{1,1}_{s,\loc}(Q\cup(0,T)\times(\sO\cup\Gamma_0))
$$
be a solution to the parabolic obstacle problem \eqref{eq:stochastic_representation_Parabolic_obstacle_problem},
\eqref{eq:stochastic_representation_Parabolic_BVP_boundary_condition_Gamma_1} such that $u$ and $Au$ obey 
\eqref{eq:stochastic_representation_Growth_parabolic}. Then, $u=u^*$ on
$Q \cup \eth^1 Q$, where $u^*$ is given by
\begin{equation}
\label{eq:stochastic_representation_Stochastic_representation_POP_2}
u^*(t,z) := \sup_{\theta \in \mathscr{T}_{t,T}} J_p^{\nu_{\sO}\wedge T,\theta}(t,z),
\end{equation}
and $\nu_{\sO}$ is defined by \eqref{eq:stochastic_representation_Stopping_time_domain_with_degenerate_boundary}, for all 
$(t,z)\in Q\cup\eth^1 Q$.
\end{thm}

\begin{rmk}[Existence and uniqueness of strong solutions in weighted Sobolev spaces to the parabolic obstacle problem]
Existence and uniqueness of strong solutions in weighted Sobolev spaces to problem 
\eqref{eq:stochastic_representation_Parabolic_obstacle_problem} with Dirichlet boundary
condition \eqref{eq:stochastic_representation_Parabolic_BVP_boundary_condition_Gamma_1} along $\eth^1 Q$, for all 
$\beta>0$, is proved in
\cite{Daskalopoulos_Feehan_evolvarineqheston}.
\end{rmk}

\subsection{Survey of previous results on stochastic representations of solutions to boundary value or obstacle problems}
Stochastic representations of solutions to elliptic and parabolic boundary value and obstacle problems discussed by 
Bensoussan and Lions \cite{Bensoussan_Lions} and Friedman
\cite{FriedmanSDE} are established under the hypotheses that the matrix of coefficients, $(a^{ij})$, of the second-order 
spatial derivatives in an elliptic linear, second-order
differential operator, $A$, is \emph{strictly elliptic} and that all coefficients of $A$ are \emph{bounded}. Relaxations 
of these hypotheses, as in \cite[Chapter $13$ \&
$15$]{FriedmanSDE}, and more recently \cite{Zhou_quasiderivdegenpde}, fail to include the Heston generator mainly because 
the matrix $(a^{ij})$ does \emph{not} satisfy

\begin{hyp}[Extension property for positive definite, $C^2$ matrix-valued functions]
\label{hyp:stochastic_representation_Extension_matrix}
Given a subdomain $V\subsetneqq (0,\infty)\times\RR^d$, for $d\geq 1$, we say that a matrix-valued function,
$$
a: V \to \RR^{d\times d},
$$
which is $C^2$ on $V$ and $a(t,z)$ is positive definite for each $(t,z) \in V$ has the \emph{extension property} if there 
is a matrix-valued function,
$$
\tilde a: [0,\infty)\times\RR^d \to \RR^{d\times d},
$$
which coincides with $a$ on $V$ but is $C^2$ on $[0,\infty)\times\RR^d$ and $\tilde a(t,z)$ is positive definite for each 
$(t,z) \in [0,\infty)\times\RR^d$.
\end{hyp}

Naturally, Hypothesis \ref{hyp:stochastic_representation_Extension_matrix} is also applicable when the matrix $a$ is 
constant with respect to time, that is, in elliptic problems.
Note that in the case of the Heston process, $d=2$, $V=(0,\infty)\times\HH$, and
\begin{align*}
a(t,x,y) &:=
\left(\begin{array}{ll}
y & \sigma \rho y\\
\sigma\rho y & \sigma^2 y
\end{array}\right), \quad \forall (x,y)\in \HH,
\end{align*}
and so the matrix $a$ does \emph{not} satisfy Hypothesis \ref{hyp:stochastic_representation_Extension_matrix}. We now give 
more detailed comparisons for each of the four main
problems which we consider in this article.

\subsubsection{Elliptic boundary value problems}
Uniqueness of stochastic representations of solutions to \emph{non-degenerate} elliptic partial differential equations is 
established in \cite[Theorem $6.5.1$]{FriedmanSDE},
\cite[Proposition $5.7.2$]{KaratzasShreve1991}, \cite[Theorem $9.1.1$ \& Corollary $9.1.2$, Theorem 
$9.3.2$]{Oksendal_2003}, and \cite[Theorem 2.7.1 \& Remarks 2.7.1,
2.7.2]{Bensoussan_Lions} (for a bounded domain $\sO$), and \cite[Theorem 2.7.2 \& Remarks 2.7.3--5]{Bensoussan_Lions} 
(when the domain is the whole space, $\sO=\RR^n$).

Existence (and uniqueness) of stochastic representations of solutions to \emph{non-degenerate} elliptic partial 
differential equations is established in \cite[Theorem
$6.5.1$]{FriedmanSDE}, \cite[Theorem $9.2.14$]{Oksendal_2003}, and \cite[Theorem 9.3.3 \& Remark, p. 196]{Oksendal_2003}.

Existence and uniqueness of solutions to a certain class of \emph{degenerate} elliptic partial differential equations are 
described by Friedman in \cite[Theorems 13.1.1 \&
13.3.1]{FriedmanSDE}, but those results do not apply to the Heston operator because a square root, $(\sigma^{ij})$, of the 
matrix $(a^{ij})$ cannot be extended as a uniformly
Lipschitz continuous function on $\RR^2$, that is, \cite[Condition (A), p. $308$]{FriedmanSDE} is not satisfied. Stroock 
and Varadhan \cite[\S $5$-$8$]{StroockVaradhan1972} also
discuss existence and uniqueness of solutions to degenerate elliptic partial differential equations, but their assumption 
that the matrix $(a^{ij})$ satisfies Hypothesis
\ref{hyp:stochastic_representation_Extension_matrix} does not hold for the Heston operator (see \cite[Theorem 
$2.1$]{StroockVaradhan1972}).

More recently, Zhou \cite{Zhou_quasiderivdegenpde} employs the method of quasiderivatives to establish the stochastic 
representation of solutions to a certain class of degenerate
elliptic partial differential equations, and obtains estimates for the derivatives of their solutions. However, his 
results do not apply to the Heston operator because
\cite[Assumptions $3.1$ \& Condition $(3.2)$]{Zhou_quasiderivdegenpde} are not satisfied in this case. Moreover, the 
Dirichlet condition is imposed on the whole boundary of the
domain (see \cite[Equation $(1.1)$]{Zhou_quasiderivdegenpde}), while we take into consideration the portion of the 
boundary, $\Gamma_0$, where the differential operator $A$
becomes degenerate.

\subsubsection{Elliptic obstacle problems}
We may compare Theorems \ref{thm:stochastic_representation_Uniqueness_elliptic_OP} and 
\ref{thm:stochastic_representation_Uniqueness_elliptic_OP_beta_less_than_1} with the
\emph{uniqueness} assertions (in increasing degrees of generality) for \emph{non-degenerate} elliptic operators in 
\cite[Theorems 3.3.1, 3.3.2, 3.3.4, 3.3.5, 3.3.8, 3.3.19,
3.3.20, \& 3.3.23]{Bensoussan_Lions}. See also \cite[Theorem 10.4.1]{Oksendal_2003} and \cite[Theorems 16.4.1, 16.4.2, 
16.7.1, \& 16.8.1]{FriedmanSDE} for uniqueness assertions
\emph{non-degenerate} elliptic operators, though with more limited applicability.

\subsubsection{Parabolic boundary value problems}
Uniqueness of solutions to \emph{non-degenerate} parabolic partial differential equations and their stochastic 
representations are described in \cite[Theorems 6.5.2,
6.5.3]{FriedmanSDE}, \cite[Theorem $5.7.6$]{KaratzasShreve1991} and \cite[Theorems $2.7.3$ \& $2.7.4$]{Bensoussan_Lions}.

Friedman obtains fundamental solutions and stochastic representations of solutions to certain degenerate parabolic partial 
differential equations in \cite{Friedman_1974}, while
he obtains uniqueness and stochastic representations of solutions to the Cauchy problem in \cite{Friedman_1973}; those 
results are summarized in \cite[Chapter $15$]{FriedmanSDE}.
Nevertheless, the results in \cite[Chapter $15$]{FriedmanSDE} and \cite{Friedman_1974} do not apply to the Heston operator 
because Hypothesis
\ref{hyp:stochastic_representation_Extension_matrix} does not hold, that is \cite[Condition (A), p. 389]{FriedmanSDE} is 
not satisfied. Therefore, the method of construction in
\cite[Theorem $1.2$]{Friedman_1974} of a candidate for a fundamental solution does not apply to the Heston operator. A 
stochastic representation for a solution to the Cauchy
problem for a degenerate operator is obtained in \cite[\S 15.10]{FriedmanSDE}, but the hypotheses of \cite[Theorem 
$15.10.1$]{FriedmanSDE} are again too restrictive and
exclude the Heston operator.

Ekstr\"om and Tysk \cite{Ekstrom_Tysk_bcsftse} consider the problem of pricing European-style options on an underlying 
process which is the solution to a degenerate,
one-dimensional stochastic differential equation which satisfies \cite[Hypothesis $2.1$]{Ekstrom_Tysk_bcsftse}, and so 
includes the \emph{Feller square root} (or
\emph{Cox-Ingersoll-Ross}) \emph{process}, \eqref{eq:stochastic_representation_CIR}. The option price is the classical 
solution in the sense of \cite[Definition
$2.2$]{Ekstrom_Tysk_bcsftse} to the corresponding parabolic partial differential equation \cite[Theorem 
$2.3$]{Ekstrom_Tysk_bcsftse}. Under their assumption that the payoff
function $g(T,\cdot)$ is in $C^1([0,\infty))$, they show that their classical solution has the regularity property,
$$
u \in C([0,T]\times [0,\infty)) \cap C^1([0,T)\times[0,\infty)) \cap C^2([0,T)\times(0,\infty)),
$$
and obeys the second-order boundary condition,
\[
\lim_{(t,y) \rightarrow (0,t_0)} y u_{yy}(t,y) =0,\quad\forall t_0\in(0,T)
\quad\hbox{(by \cite[Proposition $4.1$]{Ekstrom_Tysk_bcsftse}).}
\]
As a consequence, in the framework of our article, their solution obeys
$$
u \in C^{1,1}_{s,\loc}((0,t_0)\times[0,\infty)), \quad\forall t_0\in(0,T),
$$
where the vector space of functions $C^{1,1}_{s,\loc}((0,t_0)\times[0,\infty))$ is defined by analogy with 
\eqref{eq:stochastic_representation_C_1_1_regularity_parabolic}.

In \cite{Ekstrom_Tysk_bssvm}, Ekstr\"om and Tysk extend their results in \cite{Ekstrom_Tysk_bcsftse} to the case of 
two-dimensional stochastic volatility models for option
prices, where the variance process satisfies the assumptions of \cite[Hypothesis $2.1$]{Ekstrom_Tysk_bcsftse}.

Bayraktar, Kardaras, and Xing \cite{Bayraktar_Kardaras_Xing_2012} address the problem of \emph{uniqueness} of 
classical solutions, in the sense of \cite[Definitions $2.4$ \&
$2.5$]{Bayraktar_Kardaras_Xing_2012}, to a class of two-dimensional, degenerate parabolic partial differential 
equations. Their differential operator has a degeneracy which
is similar to that of the Heston generator, $-A$, and to the differential operator considered in 
\cite{Ekstrom_Tysk_bcsftse}, but the matrix of coefficients, $(a^{ij})$, of their
operator may have \emph{more than quadratic growth} with respect to the spatial variables (see \cite[Standing Assumption 
$2.1$]{Bayraktar_Kardaras_Xing_2012}). Therefore,
weak maximum principles for parabolic partial differential operators on unbounded domains such as \cite[Exercise 
$8.1.22$]{Krylov_LecturesHolder} do not guarantee uniqueness of
solutions in such situations. The main result of their article -- \cite[Theorem $2.9$]{Bayraktar_Kardaras_Xing_2012} 
--
establishes by probabilistic methods that uniqueness of classical solutions, obeying a natural growth condition, holds if 
and only if the asset price process is a martingale.

In our article, we consider the two-dimensional Heston stochastic process, \eqref{eq:stochastic_representation_HestonSDE}, 
where the component $Y$ of the process satisfies
\cite[Hypothesis $2.1$]{Ekstrom_Tysk_bcsftse} and \cite[Standing Assumption $2.1$]{Bayraktar_Kardaras_Xing_2012}. We 
only require the payoff function, $g(T,\cdot)$, to be
continuous with respect to the spatial variables and have \emph{exponential growth}, as in 
\eqref{eq:stochastic_representation_Growth_parabolic}. Notice that the conditions on
the payoff function are more restrictive in \cite[Hypothesis $2.1$]{Ekstrom_Tysk_bcsftse} and \cite[Standing Assumption 
$2.3$]{Bayraktar_Kardaras_Xing_2012} than in our
article. We consider the parabolic equation associated to the Heston generator, $-A$, on bounded or unbounded subdomains, 
$\sO$, of the upper half plane, $\HH$, with Dirichlet
boundary condition along the portion, $\Gamma_1$, of the boundary $\partial\sO$ contained in $\HH$. Along the portion, 
$\Gamma_0$, of the boundary
contained in $\partial \HH$, we impose a suitable Dirichlet boundary condition, depending on the value of the parameter 
$\beta$ in \eqref{defn:stochastic_representation_beta},
which governs the behavior of the Feller square-root process when it approaches the boundary point $y=0$. In each case, we 
establish \emph{uniqueness} of solutions by proving
that suitably regular solutions must have the stochastic representations in Theorems 
\ref{thm:stochastic_representation_Uniqueness_BVP_parabolic} and
\ref{thm:stochastic_representation_Uniqueness_BVP_parabolic_beta_less_than_1}, and we prove \emph{existence} and 
\emph{regularity} of solutions, in a special case, in Theorems
\ref{thm:stochastic_representation_Existence_parabolic_BVP_beta} and 
\ref{thm:stochastic_representation_Existence_parabolic_BVP_beta_with_Holder_continuous_boundary_data},
complementing the results of \cite{Ekstrom_Tysk_bcsftse}. In addition, we consider the parabolic \emph{obstacle} problem 
and establish uniqueness and the stochastic
representations of suitably regular solutions in Theorems \ref{thm:stochastic_representation_Uniqueness_parabolic_OP} and
\ref{thm:stochastic_representation_Uniqueness_parabolic_OP_beta_less_than_1}.

\subsubsection{Parabolic obstacle problems}
We may compare Theorems \ref{thm:stochastic_representation_Uniqueness_parabolic_OP} and 
\ref{thm:stochastic_representation_Uniqueness_parabolic_OP_beta_less_than_1} with the
\emph{uniqueness} assertions and stochastic representations of solutions (in increasing degrees of generality) for 
\emph{non-degenerate} operators in \cite[Theorems 3.4.1, 3.4.2,
3.4.3, 3.4.5, 3.4.6, 3.4.7, 3.4.8]{Bensoussan_Lions}.

\subsection{Further work}
The authors are developing an extension of the main results of this article to a broader class of degenerate Markov 
processes in higher dimensions and more general boundary
conditions (including Neumann and oblique boundary conditions).

\subsection{Outline of the article}
\label{subsec:Guide}
For the convenience of the reader, we provide a brief outline of the article. We begin in \S 
\ref{sec:HestonProcessProperties} by reviewing or proving some of the key properties
of the Feller square root and Heston processes which we shall need in this article. In \S 
\ref{sec:stochastic_representation_Elliptic_boundary_value_problem}, we prove existence
and uniqueness (in various settings) of solutions to the elliptic boundary value problem for the Heston operator, while in 
\S
\ref{sec:stochastic_representation_Elliptic_obstacle_value_problem}, we prove uniqueness (again in various settings) of 
solutions to the corresponding obstacle problem. We
proceed in \S \ref{sec:stochastic_representation_Parabolic_BVP}, to prove existence and uniqueness of solutions to the 
parabolic terminal/boundary value problem for the Heston
operator and in \S \ref{sec:stochastic_representation_Parabolic_Obstacle}, we prove uniqueness of solutions to the 
corresponding parabolic obstacle problem. Appendices
\ref{app:LocalBoundaryEstimates} and \ref{sec:RegularPoints} contain additional technical results which we shall need 
throughout our article.

\subsection{Notation and conventions}
\label{subsec:stochastic_representation_Notation}
When we label a condition an \emph{Assumption}, then it is considered to be universal and in effect throughout this 
article and so not referenced explicitly in theorem and
similar statements; when we label a condition a \emph{Hypothesis}, then it is only considered to be in effect when 
explicitly referenced. We let
$\NN:=\left\{1,2,3,\ldots\right\}$ denote the set of positive integers. For $x,y\in\RR$, we denote $x\wedge y : 
=\min\{x,y\}$, $x \vee y := \max\{x,y\}$ and $x^+:=x\vee 0$.

\subsection{Acknowledgments} We are very grateful to everyone who has provided us with comments on previous versions of 
this article or related conference or seminar
presentations. Camelia Pop thanks Daniel Ocone for many helpful discussions on probability theory. Finally, we thank the 
anonymous referee for many helpful suggestions and kind comments.

\section{Properties of the Heston stochastic volatility process}
\label{sec:HestonProcessProperties}
In this section, we review or develop some important properties of the Feller square root process and the Heston 
stochastic volatility process.

By \cite[Theorem $1.9$]{Feehan_Pop_mimickingdegen}, it follows that for any initial point $(t,y)\in 
[0,\infty)\times[0,\infty)$, the Feller stochastic differential equation,
\begin{equation}
\label{eq:stochastic_representation_CIR}
\begin{aligned}
d Y(s) &= \kappa\left(\vartheta-Y(s)\right)ds+\sigma\sqrt{\left|Y(s)\right|} d W(s), \quad s>t,\\
Y(t) &= y,
\end{aligned}
\end{equation}
admits a unique weak solution $(Y^{t,y}(s), W(s))_{s \geq t}$, called the Feller square root process, where $(W(s))_{s 
\geq t}$ is a one-dimensional Brownian motion on a filtered
probability space $\left(\Omega,\sF,\PP^{t,y},\FF\right)$ such that the filtration $\FF = \{\sF(s)\}_{s\geq 0}$ satisfies 
the usual conditions \cite[Definition
$1.2.25$]{KaratzasShreve1991}. Theorem $1.9$ in \cite{Feehan_Pop_mimickingdegen} also implies that the Heston stochastic 
differential equation
\eqref{eq:stochastic_representation_HestonSDE} admits a unique weak solution, $\left(Z^{t,z}(s),W(s)\right)_{s \geq t}$, 
for any initial point $(t,z) \in
[0,\infty)\times\bar\HH$, where $(W(s))_{s \geq t}$ is now an $\RR^2$-valued Brownian motion on a filtered probability 
space $\left(\Omega,\sF,\QQ^{t,z},\FF\right)$ such that the
filtration $\FF = \{\sF(s)\}_{s\geq 0}$ satisfies the usual conditions. When the initial condition $(t,y)$ or $(t,z)$ is 
clear from the context, we omit the superscripts in the
definition of the probability
measures $\PP^{t,y}$ and $\QQ^{t,z}$, respectively.

Moreover, the weak solutions to the Feller and Heston stochastic differential equations are \emph{strong}. To prove this, 
we begin by reviewing a result of Yamada
\cite{Yamada_1978}.

\begin{defn}[Coefficients for a non-Lipschitz stochastic differential equation]
\label{defn:NonLipschitzSDECoefficients}
\cite[p. 115]{Yamada_1978}
In this article we shall consider one-dimensional stochastic differential equations whose diffusion and drift 
coefficients, $\alpha,b$, obey the following properties:
\begin{enumerate}
\item The functions $\alpha, b:[0,\infty)\times\RR\to \RR$ are continuous.
\item (\emph{Yamada condition}) There is an increasing function $\varrho:[0,\infty)\to[0,\infty)$ such that 
    $\varrho(0)=0$, for some $\varepsilon>0$ one has $\int_0^\varepsilon
    \varrho^{-2}(y)\,dy =\infty$, and
\begin{equation}
\label{eq:PseudoLipschitzDiffusionCoefficient}
|\alpha(t,y_1) - \alpha(t,y_2)| \leq \varrho(|y_1-y_2|), \quad y_1,y_2 \in \RR, t\geq 0.
\end{equation}
\item There is a constant $C_1>0$ such that
\begin{equation}
\label{eq:LipschitzDrift}
|b(t,y_2) -  b(t,y_1)| \leq C_1|y_2-y_1|, \quad y_1,y_2 \in \RR, t\geq 0.
\end{equation}
\item There is a constant $C_2>0$ such that
\begin{equation}
\label{eq:LinearGrowth}
|\alpha(t,y)| + |b(t,y)| \leq C_2(1+|y|), \quad t\geq 0, y\in\RR.
\end{equation}
\end{enumerate}
\end{defn}

Clearly, the coefficients of the Feller stochastic differential equation obey the hypotheses in Definition 
\ref{defn:NonLipschitzSDE}, where $\alpha(t,y) = \sigma\sqrt{y}$ and
$b(t,y) = \kappa(\theta-y)$. Indeed, one can choose $C_1=\kappa$, $C_2=\max\{\kappa, \kappa\theta, \sigma\}$, and 
$\varrho(y)= \sigma\sqrt{y}$, as the mean value theorem yields
$$
\sqrt{y_2} - \sqrt{y_1} = c(y_1,y_2)(y_2-y_1),
$$
where
$$
c(y_1,y_2) = \frac{1}{2}\int_0^1\frac{1}{\sqrt{y_1 + t(y_2-y_1)}} \leq \frac{1}{\sqrt{y_2-y_1}},
$$
for $0<y_1<y_2$. See \cite[Remark 1]{Yamada_1978} for other examples of suitable functions $\varrho$.

\begin{rmk}
When $\varrho(u)=u^\gamma$, $\gamma\in [\frac{1}{2},1]$ \cite[Remark 1]{Yamada_1978}, then Definition 
\ref{defn:NonLipschitzSDECoefficients} implies that $\alpha(t,\cdot)$ is
H\"older continuous with exponent $\gamma$, uniformly with respect to $t\in [0,\infty)$.
\end{rmk}

\begin{defn}[Solution to a non-Lipschitz stochastic differential equation]
\label{defn:NonLipschitzSDE}
\cite[p. 115]{Yamada_1978}, \cite[Definitions IX.1.2 \& IX.1.5]{Revuz_Yor}
Let $(\Omega,\sF,\PP,\FF)$ be a filtered probability space satisfying the usual conditions. We call a pair 
$(Y(s),W(s))_{s\geq 0}$ a \emph{weak solution} to the non-Lipschitz
one-dimensional stochastic differential equation,
\begin{equation}
\label{eq:1DDegenSDE}
dY(s) = b(s,Y(s))\,ds + \alpha(s,Y(s))\,dW(s), \quad s\geq 0, Y(0)=y,
\end{equation}
where $y\in\RR$, if the following hold:
\begin{enumerate}
\item The processes $Y(s)$ and $W(s)$ are defined on $(\Omega,\sF,\PP,\FF)$;
\item The process $Y(s)$ is continuous with respect to $s\in[0,\infty)$ and is $\FF$-adapted;
\item The process $W(s)$ is a standard $\FF$-Brownian motion.
\end{enumerate}
We call $(Y(s),W(s))_{s\geq 0}$ a \emph{strong solution} to \eqref{eq:1DDegenSDE} if $Y$ is $\FF^W$-adapted, where $\FF^W$ 
is the $\PP$-completion of the filtration of $\sF$
generated by $(W(s))_{s\geq 0}$. (Compare \cite[Definition IV.1.2]{Ikeda_Watanabe}, \cite[Definition 
5.2.1]{KaratzasShreve1991}, and \cite[\S 5.3]{Oksendal_2003}.)
\end{defn}

\begin{thm}
\label{thm:1DDegenSDEWeakSolutionExists}
\cite[p. 117]{Yamada_1978}
There exists a weak solution $(Y,W)$ to \eqref{eq:1DDegenSDE}.
\end{thm}

\begin{rmk}
Yamada's main theorem \cite[p. 117]{Yamada_1978} asserts considerably more than Theorem 
\ref{thm:1DDegenSDEWeakSolutionExists}. In particular, his article shows that
\eqref{eq:1DDegenSDE} may be solved using the method of finite differences. Simpler results may suffice to merely 
guarantee the existence of a weak solution, as we need here; see
Skorokhod \cite{Skorokhod}.
\end{rmk}

\begin{prop}
\label{prop:1DDegenSDEStrongSolutionExistsUnique}
There exists a unique strong solution to \eqref{eq:1DDegenSDE}.
\end{prop}

\begin{proof}
Theorem \ref{thm:1DDegenSDEWeakSolutionExists} ensures that \eqref{eq:1DDegenSDE} admits a weak solution. Conditions 
\eqref{eq:PseudoLipschitzDiffusionCoefficient} and
\eqref{eq:LipschitzDrift} ensure that pathwise uniqueness holds for (weak) solutions to \eqref{eq:1DDegenSDE} by Revuz and 
Yor \cite[Theorem IX.3.5 (ii)]{Revuz_Yor}, while
Karatzas and Shreve \cite[Corollary 5.3.23]{KaratzasShreve1991} imply that \eqref{eq:1DDegenSDE} admits a strong solution; 
see \cite[p. 310]{KaratzasShreve1991}. Conditions
\eqref{eq:PseudoLipschitzDiffusionCoefficient} and \eqref{eq:LipschitzDrift} guarantee the uniqueness of strong solutions 
to \eqref{eq:1DDegenSDE} by Karatzas and Shreve
\cite[Proposition 5.2.13]{KaratzasShreve1991}; compare Yamada and Watanabe \cite{Yamada_Watanabe_1971a, 
Yamada_Watanabe_1971b}. (Pathwise uniqueness is also asserted for
\eqref{eq:1DDegenSDE} by \cite[Theorem IV.3.2]{Ikeda_Watanabe} when \eqref{eq:1DDegenSDE} is time-homogeneous, noting that 
the coefficients $\alpha, b$ are not required
to be bounded by Ikeda and Watanabe \cite[p. 168]{Ikeda_Watanabe}). We conclude that a strong solution to 
\eqref{eq:1DDegenSDE} exists and is unique.
\end{proof}

\begin{cor}
\label{cor:CIRSDEStrongSolution}
Given any initial point $(t,y)\in[0,\infty)\times[0,\infty)$, there exists a unique strong solution, 
$(Y^{t,y}(s),W(s))_{s\geq t}$, to the Feller stochastic differential
equation.
\end{cor}

\begin{proof}
Immediate from Proposition \ref{prop:1DDegenSDEStrongSolutionExistsUnique}.
\end{proof}

\begin{cor}
\label{cor:HestonSDEStrongSolution}
Given $(t,z)\in[0,\infty)\times\bar\HH$, there exists a unique strong solution, $(Z^{t,z}(s),W(s))_{s\geq t}$, to the 
Heston stochastic differential equation, where
$(W(s))_{s\geq 0}$ is a standard two-dimensional $\FF$-Brownian motion on $(\Omega,\sF,\PP,\FF)$.
\end{cor}

\begin{proof}
By Proposition \ref{prop:1DDegenSDEStrongSolutionExistsUnique}, the Cox-Ingersoll-Ross stochastic differential equation 
has a unique strong solution, $(Y^{t,y}(s),W_2(s))_{s\geq
t}$, where $(W_2(s))_{s\geq t}$ is a standard one-dimensional $\FF$-Brownian motion on $(\Omega,\sF,\PP,\FF)$ and 
$(Y^{t,y}(s))_{s\geq t}$ is $\FF^{W_2}$-adapted. But given
$(Y^{t,y}(s))_{s\geq t}$ and a standard two-dimensional $\FF$-Brownian motion, $(W(s))_{s\geq t} = (W_1(s),W_2(s))_{s\geq 
t}$ on $(\Omega,\sF,\PP,\FF)$, the process
$(X^{t,x,y}(s))_{s\geq t}$, and thus $(Z^{t,z}(s))_{s\geq t}= (X^{t,x,y}(s),Y^{t,y}(s))_{s\geq t}$, is uniquely determined 
by
\begin{align*}
X^{t,x,y}(s) &= x + \int_t^s\left(r-q-\frac{1}{2}Y^{t,y}(u)\right)\,du
\\
&\quad + \int_t^s \sqrt{Y^{t,y}(u)}\left(\sqrt{1-\rho^2}dW_1(u)+\rho dW_2(u)\right).
\end{align*}
This completes the proof.
\end{proof}

\begin{lem}[Properties of the Feller square-root process]
\label{lem:PropertiesCIR}
The unique strong solution of the Feller stochastic differential equation started at any 
$(t,y)\in[0,\infty)\times[0,\infty)$ satisfies
\begin{equation}
\label{eq:stochastic_representation_CIR_nonnegative}
Y(s) \geq 0\quad \PP^{t,y}\hbox{-a.s.},\quad \forall s\geq t,
\end{equation}
and also
\begin{align}
\label{eq:stochastic_representation_time_spend_at_zero}
\int_t^s \mathbf{1}_{\{Y(u) =0\}} \,du =0, &\quad \forall s\geq t,\\
\label{eq:stochastic_representation_CIR_local_time}
L(s,x) =0, &\quad \forall x \leq 0, \forall s\geq t,
\end{align}
where $L(\cdot, \cdot)$ is the local time of the Feller square-root process.
\end{lem}

\begin{proof}
Without loss of generality, we may assume that $t=0$. In \cite[Lemma $2.4$]{Bayraktar_Kardaras_Xing_2012}, it is 
proved that $L(s,0)=0$, for all $s\geq 0$, but it is not
clear to us why it also follows that
\[
L(s,0-):=\lim_{x \uparrow 0} L(s,x) = 0,\quad \forall s \geq 0,
\]
a property we shall need in our proof of \eqref{eq:stochastic_representation_CIR_nonnegative}. To complete the argument, 
we consider the following stochastic differential
equation,
\begin{align*}
\,d \widetilde{Y}(s) &= b(\widetilde{Y}(s)) \,d s + \alpha(\widetilde{Y}(s)) \,d W(s),\quad s > 0,\\
\widetilde Y(0) &= y,
\end{align*}
where we let
\begin{equation}
\label{eq:stochastic_representation_tilde_CIR_coeff}
b(y):=\kappa(\vartheta-y) \hbox{ and } \alpha(y):= \mathbf{1}_{\{y\geq 0\}}\sigma \sqrt{y},\quad\forall y\in\RR.
\end{equation}
This stochastic differential equation admits a unique strong solution by Proposition 
\ref{prop:1DDegenSDEStrongSolutionExistsUnique}. We will show that $\widetilde{Y}(s) \geq 0$
a.s., for all $s\geq 0$, so that uniqueness of solutions to the Feller stochastic differential equation 
\eqref{eq:stochastic_representation_CIR} implies that $\widetilde{Y}=Y$
a.s. and $Y$ will satisfy the same properties as $\widetilde{Y}$. Thus, it is enough to prove 
\eqref{eq:stochastic_representation_time_spend_at_zero} and
\eqref{eq:stochastic_representation_CIR_local_time} for $\widetilde{Y}$. Property 
\eqref{eq:stochastic_representation_CIR_nonnegative} is a consequence of the preceding two
properties of $\widetilde{Y}$.

Let $\widetilde{L}$ be the local time process for the continuous semimartingale $\widetilde{Y}$ (see \cite[Theorem 
$3.7.1$]{KaratzasShreve1991}). From \cite[Theorem $3.7.1$
(iii)]{KaratzasShreve1991}, we know that, for any Borel measurable function $k:\mathbb{R} \rightarrow [0,\infty)$, we 
have
\begin{equation}
\label{eq:stochastic_representation_application_local_time}
\int_{0}^{s} k(\widetilde{Y}(u)) \sigma^2 \widetilde{Y}^+(u) \,d u
= 2\int_{\mathbb{R}} k(x) \widetilde{L}(s,x) dx,\quad\forall s \geq 0 .
\end{equation}
Assume, to obtain a contradiction, that $\widetilde{L}(s,0)>0$. From the right-continuity in the spatial variable of 
$\widetilde{L}(s, \cdot)$ \cite[Theorem $3.7.1$
(iv)]{KaratzasShreve1991}, there are positive constants $c$ and $x_0$ such that $\widetilde{L}(s,x)\geq c$, for all $x \in 
[0, x_0]$. For $\eps>0$, we define $k(x)=x^{-1}$, for
$x \in [\eps, x_0]$, and $0$ otherwise. With this choice of $k$, the left-hand-side in identity 
\eqref{eq:stochastic_representation_application_local_time} is bounded in absolute
value by $\sigma^2s$, for any $\eps>0$, while the right-hand-side of 
\eqref{eq:stochastic_representation_application_local_time} is greater or equal than $2c \log
\left(x_0/\eps\right)$, which diverges as $\eps$ tends to $0$. Therefore, our assumption that $\widetilde{L}(s,0) >0$ is 
false, and so $\widetilde{L}(s,0)=0$. Moreover, we notice
that for any bounded, Borel-measurable function $k$ with support in $(-\infty, 0)$ the left-hand-side in identity 
\eqref{eq:stochastic_representation_application_local_time} is
identically zero. Thus, we conclude that $\widetilde{L}(s,x)=0$, for all $x<0$, and also $\widetilde{L}(s,0-)=0$.

We use this result to show that $\mathbb{P} (\widetilde{Y}(s) \leq 0,\forall s \geq 0) = 0$. From \cite[p. $223$, third 
formula]{KaratzasShreve1991} and the fact that $\kappa,
\vartheta>0$, we see that
\[
0=\widetilde{L}(s,0)-\widetilde{L}(s,0-)=\kappa\vartheta \int_{0}^{s} \mathbf{1}_{\{\widetilde{Y}(u)=0\}}  \,du,
\]
which implies that $\mathbb{P} (\widetilde{Y}(s) =0,\forall s \geq 0 ) = 0$. It remains to show that $\mathbb{P} 
(\widetilde{Y}(s) \in (-\infty,0)) = 0$, for all $s\geq 0$, which
is equivalent to proving that for any $\eps>0$ and $s \geq 0$, we have $\mathbb{P} (\widetilde{Y}(s) \in (-\infty,-\eps)) 
= 0$. Let $\varphi:\mathbb{R}\rightarrow[0,1]$ be a
smooth cut-off function such that $\varphi|_{(-\infty, -\eps)} \equiv 1$ and $\varphi|_{(0, \infty)} \equiv 0$. We can 
choose $\varphi$ such that $\varphi' \leq 0$. Then, it
follows by It\^o's formula that
\begin{align*}
\varphi(\widetilde{Y}(s))
&= \varphi(\widetilde{Y}(0)) +
\int_{0}^{s} \left(\kappa(\vartheta- \widetilde{Y}(u)) 
\varphi'(\widetilde{Y}(u))+\frac{1}{2}\alpha^2(Y(u))\varphi''(Y(u))\right)\,d u\\
&\quad+ \int_{0}^{s} \alpha(\widetilde{Y}(u)) \varphi'(\widetilde{Y}(u))\,d W(u)\\
&= \varphi(\widetilde{Y}(0)) +
\int_{0}^{s} \kappa(\vartheta- \widetilde{Y}(u)) \varphi'(\widetilde{Y}(u))\,d u\quad \hbox{(as $\alpha(y)=0$ when 
$\varphi' \neq 0$)}.
\end{align*}
We notice that the right-hand-side is non-negative, while the left-hand-side is non-positive, as $\varphi' \leq 0$ on 
$\RR$, and $\varphi'=0$ on $(0,\infty)$. Therefore, we must
have $\varphi(\widetilde{Y}(s)) = 0$ a.s. which implies that $\PP (\widetilde{Y}(s) \in (-\infty,-\eps)) = 0$. This 
concludes the proof of the lemma.
\end{proof}

For $a,y,t \geq 0$, we let
\begin{equation}
\label{defn:stochastic_representation_CIR_First_hitting_time_a}
T^{t,y}_a := \inf\left\{s \geq t: Y^{t,y}(s)=a\right\}
\end{equation}
denote the first time the process $Y$ started at $y$ at time $t$ hits $a$. When the initial condition, $(t,y)$, is clear 
from the context, we omit the superscripts in the
preceding definition \eqref{defn:stochastic_representation_CIR_First_hitting_time_a}. Also, when $t=0$, we omit the 
superscript $t$.

\begin{lem}[Boundary classification at $y=0$ of the Feller square root process]
\label{lem:stochastic_representation_CIR_properties}
Let $Y^y$ be the unique strong solution to the Feller stochastic differential equation 
\eqref{eq:stochastic_representation_CIR} with initial condition $Y^y(0)=y$. Then
\begin{enumerate}
\item\label{item:CIRbetaGreaterEqualOne} For $\beta \geq 1$, $y=0$ is an entrance boundary point in the sense of 
    \cite[\S $15.6$(c)]{KarlinTaylor2}.
\item\label{item:CIRbetaLessOne} For $0<\beta<1$, $y=0$ is a regular, instantaneously reflecting boundary point in the 
    sense of \cite[\S $15.6$(a)]{KarlinTaylor2}, and
\begin{equation}
\label{eq:stochastic_representation_CIR_limit_hitting_time_zero}
\lim_{y \downarrow 0} T^y_0 = 0 \quad \hbox{a.s.},
\end{equation}
where $T^y_0$ is given by \eqref{defn:stochastic_representation_CIR_First_hitting_time_a}.
\end{enumerate}
\end{lem}

\begin{proof}
A direct calculation give us that the scale function, $\mathfrak{s}$, and the speed measure, $\mathfrak{m}$, of the Feller 
square root process are given by
\[
\mathfrak{s}(y)=y^{-\beta} e^{\mu y} \hbox{  and  } \mathfrak{m}(y)=\frac{2}{\sigma^2} y^{\beta-1} e^{-\mu y},\quad\forall 
y>0
\]
where $\beta=2\kappa\vartheta/\sigma^2$ and $\mu=2\kappa/\sigma^2$. We consider the following quantities, for 
$0<a<b<\infty$ and $x>0$,
\begin{align*}
S[a,b] &:= \int_a^b \mathfrak{s}(y) dy, \quad S(a,b] := \lim_{c \downarrow a} S[c,b],\\
M[a,b] &:= \int_a^b \mathfrak{m}(y) dy, \quad M(a,b] := \lim_{c \downarrow a} M[c,b],\\
N(0) &:= \int_0^x S[y,x] \mathfrak{m}(y) dy.
\end{align*}
Then, for $\beta \geq 1$, we have $S(0,x]=\infty$ and $N(0)<\infty$, which implies that $y=0$ is an entrance boundary 
point (\cite[p. $235$]{KarlinTaylor2}), while for
$0<\beta<1$, we have $S(0,x]<\infty$ and $M(0,x]<\infty$, and so $y=0$ is a regular boundary point (\cite[p. 
$232$]{KarlinTaylor2}).

Next, we consider the case $0<\beta<1$. To establish \eqref{eq:stochastic_representation_CIR_limit_hitting_time_zero}, we 
consider the following quantities
\begin{align*}
u_{a,b}(y) &:= \PP^y\left(T_b<T_a\right) = \frac{S[a,y]}{S[a,b]},\\
v_{a,b}(y) &:= \EE^y_{\PP}\left[T_a\wedge T_b\right]
= 2 u_{a,b}(y) \int_y^b S[z,b] m(z) dz + 2 \left(1-u_{a,b}(y)\right)\int_a^y S(a,z] m(z) dz,
\end{align*}
as in \cite[Equations ($15.6.1$) \& ($15.6.5$)]{KarlinTaylor2} and \cite[Equations ($15.6.2$) \& 
($15.6.6$)]{KarlinTaylor2}, respectively. Notice that $T^y_a \rightarrow T^y_0$,
when $y \downarrow 0$, by the continuity of the paths of $Y$. Then, for fixed $b>0$, we obtain
\begin{align*}
\lim_{y \downarrow 0} \PP^y (T_b < T_0) &=\lim_{y\downarrow 0} \lim_{a \downarrow 0} \PP^y (T_b<T_a)=0,\\
\lim_{y \downarrow 0} \EE^y_{\PP} \left[T_0 \wedge T_b\right] &=\lim_{y\downarrow 0} \lim_{a \downarrow 0} \EE^y_{\PP} 
\left[T_a \wedge T_b\right]=0,
\end{align*}
from where \eqref{eq:stochastic_representation_CIR_limit_hitting_time_zero} follows.
\end{proof}

Next, we have the following

\begin{lem}[Properties of the Heston process]
\label{lem:stochastic_representation_PropertiesHeston}
Let $\left(Z(s)\right)_{s \geq 0}$ be the unique strong solution to the Heston stochastic differential equation 
\eqref{eq:stochastic_representation_HestonSDE}.
\begin{enumerate}
\item\label{item:SupermartingaleX} Assume $q \geq 0$ and $r \in \RR$. Then, for any constant $c\in [0,1]$,
\begin{equation}
\label{eq:stochastic_representation_Supermartingale_property_X}
\left(e^{-rcs} e^{cX(s)}\right)_{s \geq 0} \hbox{ is a positive supermartingale.}
\end{equation}
\item\label{item:SupermartingaleY} For any positive constant $c\leq  \mu$,
\begin{equation}
\label{eq:stochastic_representation_Supermartingale_property_Y}
\left(e^{-c\kappa\vartheta s} e^{cY(s)}\right)_{s \geq 0} \hbox{ is a positive supermartingale.}
\end{equation}
\end{enumerate}
\end{lem}

\begin{proof}
To establish \eqref{eq:stochastic_representation_Supermartingale_property_X}, we use It\^o's formula to give
\begin{equation}
\label{eq:Sochastic_representation_SDE_X_supermartingale}
\begin{aligned}
d\left(e^{-rcs}e^{cX(s)}\right) &=  - e^{-rcs}e^{cX(s)}\left(cq+\frac{1}{2}c(1-c)Y(s)\right)\,ds\\
&\quad           + c e^{-rcs}e^{cX(s)}\sqrt{Y(s)}\,dW_1(s).
\end{aligned}
\end{equation}
Notice that the drift coefficient is non-positive, since $Y(s) \geq 0$ a.s. for all $s \geq 0$ by Lemma 
\ref{lem:PropertiesCIR}, and $q \geq 0$, and $c\in[0,1]$.

Similarly, to establish \eqref{eq:stochastic_representation_Supermartingale_property_Y} for the Feller square root 
process, we have
\begin{equation}
\label{eq:Sochastic_representation_SDE_Y_supermartingale}
\begin{aligned}
d\left( e^{- c \kappa\vartheta s} e^{c Y(s)}\right) &=
e^{- c \kappa\vartheta s} e^{c Y(s)} c\left(c\sigma^2/2-\kappa \right)Y(s) ds\\
&\quad + c \sigma  e^{- c \kappa\vartheta s} e^{c Y(s)} \sqrt{Y(s)} \left(\rho dW_1(s)+\sqrt{1-\rho^2}dW_2(s)\right).
\end{aligned}
\end{equation}
When $c\leq \mu$, we see that the drift coefficient in the preceding stochastic differential equation is non-negative.

The supermartingale properties \eqref{eq:stochastic_representation_Supermartingale_property_X} and 
\eqref{eq:stochastic_representation_Supermartingale_property_Y} follow if we
show in addition that the processes are integrable random variables for each time $s \geq 0$. For simplicity, we let 
$Q(s)$ denote either one of the processes we consider, and we
let $\theta_n$ be the first exit time of the Heston process $\left(X(s), Y(s)\right)_{s \geq 0}$ from the rectangle 
$(-n,n)\times(-n,n)$, where $n \in \NN$. We set
$Q_n(s):=Q(s\wedge\theta_n)$, for all $s\geq 0$. We then have
\[
d Q_n(s) = \mathbf{1}_{\{s \leq \theta_n\}} d Q_n(s),\quad \forall s >0, \quad \forall n \in \NN.
\]
Using equations \eqref{eq:Sochastic_representation_SDE_X_supermartingale} and 
\eqref{eq:Sochastic_representation_SDE_Y_supermartingale}, it is clear that $(Q_n(s))_{s \geq 0}$
are supermartingales, because the coefficients of the stochastic differential equations are bounded and the drift terms 
are non-positive. Therefore, we know that
\begin{equation}
\label{eq:stochastic_representation_Approximation_sipermatingale}
\EE^{x,y}_{\QQ}\left[Q_n(t) | \sF(s)\right] \leq Q_n(s), \quad \forall t\geq s,\quad\forall s\geq 0, \quad \forall n \in 
\NN.
\end{equation}
Clearly, we also have $Q_n(t)\rightarrow Q(t)$ a.s., as $n \rightarrow \infty$, for all $t\geq s$ and $s \geq 0$. Taking 
the limit as $n\rightarrow\infty$ in
\eqref{eq:stochastic_representation_Approximation_sipermatingale} and using the positivity of the processes, Fatou's lemma 
yields
\begin{align*}
\EE^{x,y}_{\QQ}\left[Q(t) | \sF(s)\right] &\leq \liminf_{n \rightarrow \infty}  \EE^{x,y}_{\QQ}\left[Q_n(t) | 
\sF(s)\right]\\
&\leq \liminf_{n \rightarrow \infty}  Q_n(s) \quad \hbox{ (by 
\eqref{eq:stochastic_representation_Approximation_sipermatingale})}\\
&= Q(s), \quad \forall t \geq s ,\quad\forall s\geq 0,
\end{align*}
and so \eqref{eq:stochastic_representation_Supermartingale_property_X} and 
\eqref{eq:stochastic_representation_Supermartingale_property_Y} follow.
\end{proof}

The next lemma is used to show that the functions $u^*$ given by 
\eqref{eq:stochastic_representation_StochasticRep_BVP_all_beta} and
\eqref{eq:stochastic_representation_StochasticRep_BVP_2} are well-defined and satisfy the growth assumption 
\eqref{eq:stochastic_representation_Growth_elliptic}.

\begin{lem}
\label{lem:stochastic_representation_WellDefinedStochRepVarEq}
Suppose $r>0$, and $f$, $g$, $\psi$ are Borel measurable functions on $\sO$ and satisfy assumption 
\eqref{eq:stochastic_representation_Growth_elliptic}. Then there is a positive
constant $\bar{C}$, depending on $r$, $\kappa$, $\vartheta$, $M_1$, $M_2$ and $C$ in 
\eqref{eq:stochastic_representation_Growth_elliptic}, such that for any $\theta_1, \theta_2
\in \sT$, the function $J_e^{\theta_1,\theta_2}$ in \eqref{eq:stochastic_representation_Functional_J} satisfies the growth 
assumption,
\[
|J_e^{\theta_1,\theta_2} (x,y)| \leq \bar C\left(1+e^{M_1 y} + e^{M_2 x}\right), \quad \forall (x,y) \in \sO,
\]
where $0<M_1<\min\left\{r/\left(\kappa\vartheta\right), \mu\right\}$ and $M_2 \in[0,1)$  are as in 
\eqref{eq:stochastic_representation_Growth_elliptic}.
\end{lem}

\begin{rmk}
The obstacle function $\psi$ in \eqref{eq:stochastic_representation_Functional_J} is only relevant for solutions to 
problem
\eqref{eq:stochastic_representation_Elliptic_obstacle_problem}.
\end{rmk}

\begin{proof}
The conclusion is a consequence of the properties of the Heston process given in Lemma 
\ref{lem:stochastic_representation_PropertiesHeston}. We first estimate the integral term
in \eqref{eq:stochastic_representation_Functional_J}. For $z \in \sO$, then
\begin{align*}
&\EE^z_\QQ\left[\int_0^{\theta_1\wedge\theta_2} e^{-rs} |f(Z(s))| ds\right] \\
&\quad
\leq C \EE^z_\QQ\left[\int_0^{\infty} e^{-rs} \left(1+e^{-rs} e^{M_1Y(s)} + e^{-rs} e^{M_2X(s)}\right) \,ds\right]
 \quad (\hbox{by \eqref{eq:stochastic_representation_Growth_elliptic}}) \\
&\quad\leq C \left(1+ \int_0^{\infty}e^{-(r-M_1\kappa\vartheta)s}\EE^z_\QQ\left[e^{-M_1\kappa\vartheta s} e^{M_1 
Y(s)}\right] \,ds\right.\\
&\quad\quad\left.   + \int_0^{\infty} e^{-(1-M_2)rs}\EE^z_\QQ\left[e^{-rM_2s} e^{M_2X(s)} ds\right]\,ds\right).
\end{align*}
Using the condition $M_1 < \min\left\{r/(\kappa\vartheta),\mu\right\}$ and 
\eqref{eq:stochastic_representation_Supermartingale_property_Y}, together with $M_2<1$ and
\eqref{eq:stochastic_representation_Supermartingale_property_X}, we see that
\begin{align}
\label{eq:stochastic_representation_Bound_integral_term_elliptic}
\EE^z_\QQ\left[\int_0^{\theta_1\wedge\theta_2} e^{-rs} |f(Z(s))| ds\right]
\leq \bar C \left(1 + e^{M_1 y} + e^{M_2 x }\right),
\end{align}
for a positive constant $\bar C$ depending on $r$, $M_1\kappa\vartheta$, $M_2$ and the constant $C$ in the growth 
assumption \eqref{eq:stochastic_representation_Growth_elliptic}
on $f$, $g$ and $\psi$.

Next, we show that the first non-integral term in \eqref{eq:stochastic_representation_Functional_J} can be written as
\begin{equation}
\label{eq:stochastic_representation_InterpretationStochRep1}
\EE^z_\QQ \left[ e^{-r\theta_1} g(Z(\theta_1)) \mathbf{1}_{\{\theta_1\leq\theta_2\}}\right] =
\EE^z_\QQ \left[ e^{-r\theta_1} g(Z(\theta_1)) \mathbf{1}_{\{\theta_1\leq\theta_2, \theta_1<\infty\}} \right],
\end{equation}
for any $\theta_1 \in \sT$ which is not necessarily finite. This is reasonable because by rewriting
\begin{align*}
 \EE^z_\QQ \left[ e^{-r\theta_1} g(Z(\theta_1)) \mathbf{1}_{\{\theta_1\leq\theta_2\}}\right] &=
 \EE^z_\QQ \left[ e^{-r\theta_1} g(Z(\theta_1)) \mathbf{1}_{\{\theta_1\leq\theta_2\wedge T\}}\right]\\
 &\quad +
 \EE^z_\QQ \left[ e^{-r\theta_1} g(Z(\theta_1)) \mathbf{1}_{\{T<\theta_1\leq\theta_2\}}\right],
\end{align*}
we shall see that the second term converges to zero, as $T\rightarrow\infty$. Using the growth assumption on $g$ in 
\eqref{eq:stochastic_representation_Growth_elliptic}, we have
\[
\EE^z_\QQ \left[ e^{-r\theta_1} \left|g(Z(\theta_1)) \right|\mathbf{1}_{\{T<\theta_1\leq\theta_2\}}\right]
\leq C \EE^z_\QQ \left[ e^{-r\theta_1} \left(1+ e^{M_1Y(\theta_1)}+e^{M_2X(\theta_1)}\right) 
\mathbf{1}_{\{T<\theta_1\}}\right],
\]
and so by Lemma \ref{lem:stochastic_representation_PropertiesHeston}, we obtain
\[
\EE^z_\QQ \left[ e^{-r\theta_1} g(Z(\theta_1)) \mathbf{1}_{\{T<\theta_1\leq\theta_2\}}\right]
\leq C\left(e^{-rT} + e^{-(r-M_1\kappa\vartheta)T}e^{M_1 y} + e^{-r(1-M_2)T} e^{M_2x}\right).
\]
Since $M_1 < r/(\kappa\vartheta)$ and $M_2<1$, we see that the right hand side converges to $0$, as $T\rightarrow\infty$. 
This justifies the identity
\eqref{eq:stochastic_representation_InterpretationStochRep1}.

Now, we use Fatou's lemma to obtain the bound \eqref{eq:stochastic_representation_Growth_elliptic} on the first 
non-integral term in
\eqref{eq:stochastic_representation_Functional_J}. For $z \in \sO$,
\begin{align*}
&\EE^z_\QQ\left[e^{-r\theta_1} \left|g(Z(\theta_1))\right|\mathbf{1}_{\{\theta_1\leq\theta_2\}}\right]\\
&\quad\leq \liminf_{n\to\infty} \EE^z_\QQ\left[e^{-r(\theta_1\wedge n)} |g(Z(\theta_1\wedge n))|\right]\\
&\quad\leq \liminf_{n\to \infty} C \left(1+ \EE^z_\QQ\left[e^{-r(\theta_1\wedge n)} e^{M_1 Y(\theta_1\wedge n)}\right]
                                           + \EE^z_\QQ\left[e^{-r(\theta_1\wedge n)} e^{M_2X(\theta_1\wedge 
                                           n)}\right]\right)
\quad(\hbox{by \eqref{eq:stochastic_representation_Growth_elliptic}}).
\end{align*}
Because $M_1<\mu$, we may apply the supermartingale property 
\eqref{eq:stochastic_representation_Supermartingale_property_Y} with $c:=M_2$. We use also that
$M_1<r/\left(\kappa\vartheta\right)$ to obtain $M_1\kappa\vartheta<r$, and so it follows by the Optional Sampling Theorem 
\cite[Theorem $1.3.22$]{KaratzasShreve1991} that
\begin{align*}
\EE^z_\QQ\left[e^{-r(\theta_1\wedge n)} e^{M_1 Y(\theta_1\wedge n)}\right]
&\leq \EE^z_\QQ\left[e^{-M_1\kappa\vartheta(\theta_1\wedge n)} e^{M_1 Y(\theta_1\wedge n)}\right] \\
&\leq e^{M_1y},\quad\forall n \in \NN.
\end{align*}
Using the fact that $M_2<1$, we see by the supermartingale property 
\eqref{eq:stochastic_representation_Supermartingale_property_X} applies with $c:=M_1$. By the Optional
Sampling Theorem \cite[Theorem $1.3.22$]{KaratzasShreve1991} we have
\begin{align*}
\EE^z_\QQ\left[e^{-r(\theta_1\wedge n)} e^{M_2X(\theta_1\wedge n)}\right]
&\leq \EE^z_\QQ\left[e^{-rM_2(\theta_1\wedge n)} e^{M_2X(\theta_1\wedge n)}\right]\\
&\leq e^{M_2 x}, \quad \forall n \in\NN.
\end{align*}
Therefore, we obtain
\[
\EE^z_\QQ\left[e^{-r\theta_1} \left|g(Z(\theta_1))\right|\mathbf{1}_{\{\theta_1\leq\theta_2\}}\right]
\leq
C\left(1+e^{M_1y}+e^{M_2x}\right).
\]
We obtain the same bound on the second non-integral term in \eqref{eq:stochastic_representation_Functional_J} because the 
obstacle function $\psi$ satisfies the same growth
condition \eqref{eq:stochastic_representation_Growth_elliptic} as the boundary data $g$.
\end{proof}

To prove Theorems \ref{thm:stochastic_representation_Uniqueness_BVP_parabolic} and 
\ref{thm:stochastic_representation_Uniqueness_BVP_parabolic_beta_less_than_1}, we make use of
the following auxiliary result

\begin{lem}
\label{lem:stochastic_representation_Power_X}
Let $z\in\overline\HH$ and $T\in(0,T_0]$, where $T_0$ is a positive constant. Let $(Z^z(s))_{s \geq 0}$ be the unique 
strong solution to the Heston stochastic differential
equation \eqref{eq:stochastic_representation_HestonSDE} with initial condition $Z^z(0)=z$. Then there is a positive 
constant $c$, depending on $y$, $\kappa$, $\vartheta$,
$\sigma$ and $T_0$, such that for any constant $p$ satisfying
\begin{equation}
\label{eq:stochastic_representation_Power_p_X}
0 \leq p< \frac{c}{2\sigma T},
\end{equation}
we have
\begin{equation}
\label{eq:stochastic_representation_Power_X}
\sup_{\theta \in \sT_{0,T}} \EE^z_{\QQ}\left[e^{pX^z(\theta)}\right] <\infty,
\end{equation}
where $\sT_{0,T}$ denotes the set of $(\Omega,\sF,\QQ^z,\FF)$-stopping times with values in $[0,T]$.
\end{lem}

\begin{proof}
We use the method of time-change. Denote
\[
M_i(t) := \int_0^t \sqrt{Y(s)} dW_i(s), \quad i=1,2,
\]
and observe that there is a two-dimensional Brownian motion $(B_1,B_2)$ \cite[Theorem 3.4.13]{KaratzasShreve1991} such 
that
\[
M_i(t) = B_i\left(\int_0^t Y(s) ds\right), \quad i=1,2.
\]
Thus, we may rewrite the solution of the Heston stochastic differential equation 
\eqref{eq:stochastic_representation_HestonSDE} in the form
\begin{align}
\label{eq:stochastic_representation_Heston_X}
X(t) = x + (r-q)s -\frac{1}{2} \int_0^t Y(s) ds + B_1\left(\int_0^t Y(s) ds\right),\\
\label{eq:stochastic_representation_Heston_Y}
Y(t) = y + \kappa\vartheta s - \kappa \int_0^t Y(s) ds + \sigma B_3 \left(\int_0^t Y(s) ds\right),
\end{align}
where $B_3 := \rho B_1 + \sqrt{1-\rho^2} B_2$ is a one-dimensional Brownian motion.

For any continuous stochastic process $(P(t))_{t \geq 0}$, we let
\begin{align*}
M_{P}(t) &:= \max_{0 \leq s \leq t} P(s),\quad\forall t\geq 0.
\end{align*}
We first prove the following estimate.
\begin{claim}
There are positive constants $n_0$ and $c$, depending on $y$, $\kappa$, $\vartheta$, $\sigma$ and $T_0$, such that
\begin{equation}
\label{eq:stochastic_representation_Estimate_max_Y}
\QQ^z\left( n \leq M_Y(T) \leq n+1 \right) \leq \frac{2}{\sqrt{\pi}} e^{-c n/(2\sigma^2 T)} \mathbf{1}_{\{n \geq n_0\}} + 
\mathbf{1}_{\{n < n_0\}},
\quad\forall n \in\NN.
\end{equation}
\end{claim}
\begin{proof}
Notice that if $M_Y(T) \leq n+1$, where $n\in\NN$, then
\[
\int_{0}^T Y(s) ds \leq (n+1)T,
\]
and so, for any positive constant $m$,
\begin{align}
\label{eq:stochastic_representation_maximum_Bm}
\left\{\max_{0\leq t \leq T} B_3\left(\int_{0}^t Y(s) ds\right) \geq m, M_Y(T) \leq n+1\right\}
&\subseteq
\left\{ M_{B_3}((n+1)T) \geq m\right\}.
\end{align}
Using the inclusion
\[
\left\{n \leq M_Y(T)\right\}
\subseteq
\left\{\max_{0\leq t \leq T} B_3\left(\int_{0}^t Y(s) ds\right) \geq \frac{n-y-\kappa\vartheta T}{\sigma}\right\}
\quad\hbox{(by \eqref{eq:stochastic_representation_Heston_Y})},
\]
we obtain by \eqref{eq:stochastic_representation_maximum_Bm},
\[
\QQ^z\left(n \leq M_Y(T) \leq n+1 \right)
\leq \QQ^z\left(M_{B_3}((n+1)T) \geq \frac{n-y-\kappa\vartheta T}{\sigma} \right).
\]
The expression for the density of the running maximum of Brownian motion \cite[Equation ($2.8.4$)]{KaratzasShreve1991} 
yields
\begin{align*}
\QQ^z\left(M_{B_3}((n+1)T) \geq \frac{n-y-\kappa\vartheta T}{\sigma} \right)
&\leq \int_{(n-y-\kappa\vartheta T)/(\sigma\sqrt{(n+1)T})}^{\infty} \frac{2}{\sqrt{2\pi}} e^{-x^2/2} dx.
\end{align*}
As in \cite[\S $7.1.2$]{AbramStegun}, we let
\[
\erfc (a) : = \frac{2}{\sqrt{\pi}} \int_a^{\infty} e^{-x^2/2} dx,\quad \forall a\in\RR,
\]
and so,
\begin{align*}
\QQ^z\left(n \leq M_Y(T) \leq n+1 \right)
&\leq \frac{1}{\sqrt{2}} \erfc\left( \frac{n-y-\kappa\vartheta T}{\sigma\sqrt{(n+1)T}}\right).
\end{align*}
Because for any $a \geq 1$,
\begin{align*}
\int_a^{\infty} e^{-x^2/2}dx &\leq \int_a^{\infty} x e^{-x^2/2}dx\\
&= e^{-a^2/2},
\end{align*}
we see that
\begin{align*}
\erfc(a) \leq  \frac{2}{\sqrt{\pi}} e^{-a^2/2}, \quad \forall a \geq 1.
\end{align*}
By hypothesis, $T\in(0,T_0]$, which implies that
\[
\frac{n-y-\kappa\vartheta T}{\sigma\sqrt{(n+1)T}} \geq \frac{n-y-\kappa\vartheta T_0}{\sigma\sqrt{(n+1) T_0}},\quad 
\forall n \in\NN.
\]
Hence, provided we have
\[
\frac{n-y-\kappa\vartheta T_0}{\sigma\sqrt{(n+1) T_0}} \geq 1,
\]
which is true for all $n\geq n_0(y, \kappa, \vartheta, \sigma, T_0)$, the smallest integer such that the preceding 
inequality holds, we see that
\begin{align}
\label{eq:sochastic_representation_estimate_M_Y_n_large}
\QQ^z\left(n \leq M_Y(T) \leq n+1 \right) \leq \frac{2}{\sqrt{\pi}}  e^{-(n-y-\kappa\vartheta T)^2/(2\sigma^2 (n+1)T)}, 
\quad \forall n \geq n_0.
\end{align}
Similarly, for a possibly larger $n_0(y, \kappa, \vartheta, \sigma, T_0)$, using again the fact that $T\in (0,T_0]$, we 
may choose a positive constant $c$, depending also on $y$,
$\kappa$, $\vartheta$, $\sigma$ and $T_0$, such that for all $n \geq n_0$, we have
\[
\frac{(n-y-\kappa\vartheta T)^2}{2\sigma^2 (n+1)T} \geq c \frac{n}{2\sigma^2 T}.
\]
Then, using the preceding inequality, we obtain the estimate \eqref{eq:stochastic_representation_Estimate_max_Y} from 
\eqref{eq:sochastic_representation_estimate_M_Y_n_large}.
This completes the proof of the claim.
\end{proof}
Next, we employ \eqref{eq:stochastic_representation_Estimate_max_Y} to obtain 
\eqref{eq:stochastic_representation_Power_X}. For any stopping time $\theta \in \sT_{0,T}$, we may
write
\[
e^{pX(\theta)} = \sum_{n=0}^{\infty} e^{pX(\theta)} \mathbf{1}_{\{M_Y(T) \leq n+1\}}
 \mathbf{1}_{\{n \leq M_Y(T) \leq n+1\}},
\]
and, by H\"older's inequality, it follows
\begin{equation}
\label{eq:stochastic_representation_Power_X_sum}
\EE^z_{\QQ} \left[e^{pX(\theta)}\right]
\leq \sum_{n=0}^{\infty} \EE^z_{\QQ} \left[e^{pX(\theta)} \mathbf{1}_{\{M_Y(T) \leq n+1\}}\right]^{1/2}
\QQ^z \left(n \leq M_Y(T) \leq n+1\right)^{1/2}.
\end{equation}
Using \eqref{eq:stochastic_representation_Heston_X} and the condition $p \geq 0$ in 
\eqref{eq:stochastic_representation_Power_p_X}, we have
\begin{align*}
&\EE^z_{\QQ}\left[e^{pX(\theta)} \mathbf{1}_{\{M_Y(T) \leq n+1\}}\right]\\
&\quad\leq e^{p(x+|r-q|T)}\EE^z_{\QQ}\left[\exp\left(2p B_1\left(\int_0^{\theta} Y(s)ds\right)\right)\mathbf{1}_{\{M_Y(T) 
\leq n+1\}}\right] \\
&\quad\leq e^{p(x+|r-q|T)}\EE^z_{\QQ}\left[\exp\left(2p \max_{0\leq t \leq T}B_1\left(\int_0^{t} 
Y(s)ds\right)\right)\mathbf{1}_{\{M_Y(T) \leq n+1\}}\right]\\
&\quad\leq e^{p(x+|r-q|T)}\EE^z_{\QQ}\left[e^{2p M_{B_1}((n+1)T)}\right],
\quad\forall n \in \NN
\quad\hbox{(by \eqref{eq:stochastic_representation_maximum_Bm})}.
\end{align*}
We see from the expression for the density of the running maximum of Brownian motion \cite[Exercise 
($2.8.4$)]{KaratzasShreve1991} that
\begin{align*}
\EE^z_{\QQ}\left[e^{2p M_{B_1}((n+1)T)}\right]
&=\int_{0}^{\infty} e^{2px} \frac{2}{\sqrt{2\pi(n+1)T}} e^{-x^2/(2(n+1)T)} dx\\
&\leq 2e^{2 p^2 (n+1)T},\quad\forall n \in \NN\quad\hbox{(by Mathematica)},
\end{align*}
and so,
\begin{equation}
\label{eq:stochastic_representation_Power_X_Y_less_n}
\EE^z_{\QQ}\left[e^{pX(\theta)} \mathbf{1}_{\{ M_Y(T) \leq n+1\}}\right]
\leq 2 e^{p(x+|r-q|T)} e^{2 p^2 (n+1)T}, \quad \forall n \in\NN.
\end{equation}
Inequalities \eqref{eq:stochastic_representation_Estimate_max_Y}, \eqref{eq:stochastic_representation_Power_X_sum} and 
\eqref{eq:stochastic_representation_Power_X_Y_less_n} give
us
\begin{align*}
\EE^z_{\QQ} \left[e^{pX(\theta)}\right]
&\leq
\sqrt{2} e^{p(x+|r-q|T)/2} \sum_{n=0}^{n_0-1} e^{p^2(n+1) T}\\
&\quad+\frac{2}{\pi^{1/4}} e^{p(x+|r-q|T)/2} \sum_{n=n_0}^{\infty}  e^{p^2(n+1) T} e^{-c n/(4\sigma^2T)}\\
& = \sqrt{2} e^{p(x+|r-q|T)/2} \sum_{n=0}^{n_0-1} e^{p^2(n+1) T}\\
&\quad + \frac{2}{\pi^{1/4}} e^{p(x+|r-q|T)/2+p^2 T} \sum_{n=n_0}^{\infty}  e^{(p^2 T-c/(4\sigma^2T))n}.
\end{align*}
We choose $p$ such that
\[
0\leq p<\frac{\sqrt{c}}{2\sigma T},
\]
that is, condition \eqref{eq:stochastic_representation_Power_p_X} is obeyed, and we obtain a bound on $\EE^z_{\QQ} 
\left[e^{pX(\theta)}\right]$ which is independent of the choice
of $\theta\in\sT_{0,T}$. Thus, \eqref{eq:stochastic_representation_Power_X} follows. (Note that 
\eqref{eq:stochastic_representation_Power_X} holds trivially when $p=0$.)
\end{proof}

\section{Elliptic boundary value problem}
\label{sec:stochastic_representation_Elliptic_boundary_value_problem}
In this section, we prove Theorem \ref{thm:stochastic_representation_Uniqueness_BVP_elliptic}. In addition to the 
uniqueness result in Theorem
\ref{thm:stochastic_representation_Uniqueness_BVP_elliptic} we establish the \emph{existence} and uniqueness of solutions 
in Theorem
\ref{thm:stochastic_representation_Existence_elliptic_BVP_beta}.

The \emph{existence} and \emph{uniqueness} of solutions to problem 
\eqref{eq:stochastic_representation_HestonEllipticEqBVP} with boundary condition
\eqref{eq:stochastic_representation_HestonEllipticBoundaryCondition_Gamma1} along $\Gamma_1$, when $\beta\geq 1$, and with 
boundary condition
\eqref{eq:stochastic_representation_HestonEllipticBoundaryCondition_whole_boundary} along $\partial\sO$, when $0<\beta<1$, 
are similar in nature. Therefore, we define
\begin{align}
\label{eq:stochastic_representation_Boundary_beta}
\partial_{\beta} \sO :=
\begin{cases}
\Gamma_1     & \hbox{ if } \beta \geq 1,\\
\partial \sO & \hbox{ if } 0<\beta <1.
\end{cases}
\end{align}
and treat the previous mentioned boundary value problems together as
\begin{equation}
\label{eq:stochastic_representation_HestonEllipticEqBVP_beta}
\begin{cases}
           Au=f &\hbox{ on }\sO,\\
           u=g &\hbox{ on }\partial_{\beta}\sO.
\end{cases}
\end{equation}
Now, we can give the

\begin{proof} [Proof of Theorem \ref{thm:stochastic_representation_Uniqueness_BVP_elliptic}]
Our goal is to show that if $u \in C_{\loc}(\sO\cup\partial_{\beta}\sO)\cap C^2(\sO)$ is a solution to problem 
\eqref{eq:stochastic_representation_HestonEllipticEqBVP_beta},
satisfying the pointwise growth condition \eqref{eq:stochastic_representation_Growth_elliptic}, then it admits the 
stochastic representation
\eqref{eq:stochastic_representation_StochasticRep_BVP_all_beta}.

We let $\left\{\sO_k: k \in \NN\right\}$ denote an increasing sequence of $C^{2+\alpha}$ subdomains of $\sO$ (see 
\cite[Definition \S $6.2$]{GilbargTrudinger}) such that each
$\sO_k$ has compact closure in $\sO$, and
\begin{equation*}
\bigcup_{k \in \NN} \sO_k = \sO.
\end{equation*}
By applying It\^o's lemma \cite[Theorems $3.3.3$ \& $3.3.6$]{KaratzasShreve1991}, we obtain for all $t>0$,
\begin{align*}
& d \left(e^{-r(t \wedge \tau_{\sO_k})} u(Z (t \wedge \tau_{\sO_k}))\right)\\
&\qquad= -\mathbf{1}_{\left\{t \leq \tau_{\sO_k}\right\}} e^{-rt} A u (Z(t)) dt\\
&\qquad\quad+\mathbf{1}_{\left\{t \leq \tau_{\sO_k}\right\}} e^{-rt}
\sqrt{Y(t)} \left(\left(u_x(Z(t)) + \sigma\rho u_y(Z(t))\right) d W_1(t) + \sigma \sqrt{1-\rho^2}u_y(Z(t))d 
W_2(t)\right).
\end{align*}
Since the subdomain $\sO_k\subset\sO$ is bounded and $u \in C^2(\sO)$, the $dW_i$-terms, $i=1,2$, in the preceding 
identity are martingales, and so we obtain
\begin{equation}
\label{eq:stochastic_representation_StatEq1}
\EE^z_\QQ \left[e^{-r(t \wedge \tau_{\sO_k})} u(Z (t \wedge \tau_{\sO_k}))\right]
= u(z) - \EE^z_\QQ \left[ \int_{0}^{t \wedge \tau_{\sO_k}} e^{-rs} f(Z(s)) ds \right].
\end{equation}
We take the limit as $k$ tends to $\infty$ in the preceding identity. By the growth estimate 
\eqref{eq:stochastic_representation_Bound_integral_term_elliptic}, we may apply the
Lebesgue Dominated Convergence Theorem to show that the integral term in \eqref{eq:stochastic_representation_StatEq1} 
converges to
\[
\EE^z_\QQ \left[ \int_{0}^{t\wedge\tau_{\sO}} e^{-rs} f(Z(s)) ds \right].
\]
For the non-integral term on the left hand side of \eqref{eq:stochastic_representation_StatEq1}, using the continuity of 
$u$ on $\sO\cup\partial_{\beta}\sO$ and of the sample
paths of the Heston process, we see that
\[
e^{-r(t \wedge \tau_{\sO_k})} u(Z(t \wedge \tau_{\sO_k}))
\rightarrow
e^{-r(t \wedge \tau_{\sO} )} u(Z(t \wedge \tau_{\sO} )), \quad \hbox{a.s.  as  } k \rightarrow \infty.
\]
Using \cite[Theorem $16.13$]{Billingsley_1986}, we prove that
\[
\EE^z_{\QQ}\left[e^{-r(t \wedge \tau_{\sO_k})} u(Z(t \wedge \tau_{\sO_k}))\right]
\rightarrow
\EE^z_{\QQ}\left[e^{-r(t \wedge \tau_{\sO} )} u(Z(t \wedge \tau_{\sO} ))\right], \quad \hbox{  as  } k \rightarrow 
\infty,
\]
by showing that
\[
\left\{e^{-r(t \wedge \tau_{\sO_k})} u(Z(t \wedge \tau_{\sO_k})): k \in \NN\right\}
\]
is a collection of uniformly integrable random variables. By \cite[Remark related to formula ($16.23$)]{Billingsley_1986}, 
it suffices to show that their $p$-th order moment is
uniformly bounded (independent of $k$), for some $p >1$. We choose $p>1$ such that  $p M_1<\mu$ and $pM_2<1$. Notice that 
this is possible because we assumed the coefficients $
M_1<\mu$ and $ M_2<1$. Then, from the growth estimate \eqref{eq:stochastic_representation_Growth_elliptic}, we have
\begin{align*}
\left|e^{-r(t \wedge \tau_{\sO_k})} u(Z)\right|^p
&\leq C e^{-rp(t \wedge \tau_{\sO_k})} \left(1+e^{pM_1 Y} + e^{pM_2 X}\right), \quad \forall k \in \NN.
\end{align*}
From the inequality \eqref{eq:stochastic_representation_Supermartingale_property_Y} with $c=p M_1<\mu$ and property
\eqref{eq:stochastic_representation_Supermartingale_property_X} applied with $c=p  M_2\in (0,1)$, we obtain using 
$M_1<r/(\kappa\vartheta)$
\begin{align*}
\EE^z_\QQ \left[\left|e^{-r(t \wedge \tau_{\sO_k})} u(Z (t \wedge \tau_{\sO_k}))\right|^p\right]
&\leq C \left(1+e^{pM_1 y} + e^{pM_2 x}\right), \quad \forall k \in \NN.
\end{align*}
Therefore, by taking limit as $k$ tends to $\infty$ in \eqref{eq:stochastic_representation_StatEq1} we obtain
\begin{equation}
\label{eq:stochastic_representation_Stochastic_representation_t}
\EE^z_\QQ \left[e^{-r (t\wedge\tau_{\sO})} u(Z (t\wedge\tau_{\sO} ))\right]
=    u(z) - \EE^z_\QQ \left[ \int_{0}^{t\wedge\tau_{\sO} } e^{-rs} f(Z(s)) ds \right].
\end{equation}
As we let $t$ tend to $\infty$, the integral term on the right-hand side in the preceding identity clearly converges to
\[
\EE^z_\QQ \left[\int_{0}^{\tau_{\sO}} e^{-rs} f(Z(s)) \,d s\right].
\]
It remains to consider the left-hand side of \eqref{eq:stochastic_representation_Stochastic_representation_t}. Keeping in 
mind that $u \in C_{\loc}(\sO\cup\partial_{\beta}\sO)$
solves \eqref{eq:stochastic_representation_HestonEllipticEqBVP_beta}, we rewrite this term as
\begin{equation*}
\EE^z_\QQ \left[e^{-r (t\wedge \tau_{\sO})} u(Z (t\wedge \tau_{\sO} ))\right]
= \EE^z_\QQ \left[e^{-r \tau_{\sO}} g(Z (\tau_{\sO})) \mathbf{1}_{\{\tau_{\sO} \leq t\}}\right]
+ \EE^z_\QQ \left[e^{-r t} u(Z (t)) \mathbf{1}_{\{\tau_{\sO} > t\}}\right].
\end{equation*}
Using the growth assumption \eqref{eq:stochastic_representation_Growth_elliptic}, we notice as above that both collections 
of random variables in the preceding identity,
\begin{equation*}
\left\{e^{-r \tau_{\sO}} g(Z (\tau_{\sO})) \mathbf{1}_{\{\tau_{\sO} \leq t\}}: t \geq 0\right\}
\hbox{  and  }
\left\{e^{-r t} u(Z (t)) \mathbf{1}_{\{\tau_{\sO} > t\}}: t \geq 0 \right\},
\end{equation*}
are uniformly integrable, and they converge a.s. to $e^{-r \tau_{\sO}} g(Z(\tau_{\sO})) \mathbf{1}_{\{\tau_{\sO} 
<\infty\}}$ and zero, respectively. Therefore, by \cite[Theorem
$16.13$]{Billingsley_1986}, letting $t$ tend to $\infty$ in 
\eqref{eq:stochastic_representation_Stochastic_representation_t}, we obtain
\[
\EE^z_\QQ \left[e^{-r \tau_{\sO}} g(Z(\tau_{\sO})) \mathbf{1}_{\{\tau_{\sO} <\infty\}}\right]
= u(z) -  \EE^z_\QQ \left[\int_{0}^{\tau_{\sO}} e^{-rs} f(Z(s)) \,d s\right],
\]
which implies that $u=u^*$ on $\sO\cup\partial_{\beta}\sO$, where $u^*$ is defined by 
\eqref{eq:stochastic_representation_StochasticRep_BVP_all_beta}.
\end{proof}

\begin{proof}[Proof of Theorem \ref{thm:stochastic_representation_Uniqueness_BVP_elliptic_beta_less_than_1}]
Our goal is to show that if $0<\beta<1$ and $u \in C_{\loc}(\sO\cup\Gamma_1)\cap C^2(\sO)\cap 
C^{1,1}_{s,\loc}(\sO\cup\Gamma_0)$ is a solution to problem
\eqref{eq:stochastic_representation_HestonEllipticEqBVP}, satisfying the growth estimate 
\eqref{eq:stochastic_representation_Growth_elliptic}, then it admits the stochastic
representation \eqref{eq:stochastic_representation_StochasticRep_BVP_2}.

We consider the following sequence of increasing subdomains of $\sO$,
\begin{align}
\label{eq:stochastic_representation_Sequence_subdomains_with_degenerate_boundary}
\sU_k &:= \left\{z \in \sO: |z|<k, \hbox{dist} \left(z, \Gamma_1 \right) > 1/k\right\}, \quad k\in\NN,
\end{align}
with non-empty boundary portions $\bar \Gamma_0 \cap \sU_k$. Let $\eps>0$ and denote
\begin{equation}
\label{eq:stochastic_representation_Z_eps}
Y^{\eps}:= Y+\eps,\hbox{  and  }
Z^{\eps} := \left(X,Y^{\eps}\right).
\end{equation}
By applying It\^o's lemma \cite[Theorems $3.3.3$ \& $3.3.6$]{KaratzasShreve1991}, we obtain
\begin{equation}
\label{eq:stochastic_representation_First_representation_BVP_2}
\EE^z_\QQ \left[e^{-r(t \wedge \nu_{\sU_k})} u(Z^{\eps} (t \wedge \nu_{\sU_k}))\right]
= u(z) - \EE^z_\QQ \left[ \int_{0}^{t \wedge \nu_{\sU_k}} e^{-rs} A^{\eps} u(Z^{\eps}(s)) ds \right],\quad \forall t>0,
\end{equation}
where $\nu_{\sU_k}$ is given by \eqref{eq:stochastic_representation_Stopping_time_domain_with_degenerate_boundary}, and 
$A^{\eps}$ denotes the elliptic differential operator,
\begin{align}
\label{eq:stochastic_representation_Definition_A_eps}
A^{\eps} v &:= A v + \frac{\eps}{2} v_x + \kappa \eps v_y - \frac{\eps}{2}\left(v_{xx}+2\rho\sigma v_{xy}+\sigma^2 
v_{yy}\right),
\quad \forall v \in C^2(\sO).
\end{align}
Using \eqref{eq:stochastic_representation_HestonEllipticEqBVP}, we can write 
\eqref{eq:stochastic_representation_First_representation_BVP_2} as \begin{equation}
\label{eq:stochastic_representation_Second_representation_BVP_2}
\begin{aligned}
\EE^z_\QQ \left[e^{-r(t \wedge \nu_{\sU_k})} u(Z^{\eps} (t \wedge \nu_{\sU_k}))\right]
&= u(z) - \EE^z_\QQ \left[ \int_{0}^{t \wedge \nu_{\sU_k}} e^{-rs} f(Z^{\eps}(s)) ds \right]\\
&\quad
- \EE^z_\QQ \left[ \int_{0}^{t \wedge \nu_{\sU_k}} e^{-rs} (A^{\eps}-A) u(Z^{\eps}(s)) ds \right].
\end{aligned}
\end{equation}
First, we take limit as $\eps$ tends to $0$ in the preceding identity.  We may assume without loss of generality that 
$\eps<1/k$, for any fixed $k \geq 1$. We evaluate the
residual term $(A^{\eps}-A)u$ with \eqref{eq:stochastic_representation_Definition_A_eps} to give
\begin{equation}
\label{eq:stochastic_representation_Estimate_A_eps_minus_A}
\left|\left(A^{\eps}-A\right)u(Z^{\eps}(s))\right|
\leq C \eps 
\|Du\|_{C(\bar\sU_{2k})} + C \left(\mathbf{1}_{\left\{Y^{\eps}(s) \leq \sqrt{\eps}\right\}}+\sqrt{\eps}\right) 
\|yD^2u\|_{C(\bar\sU_{2k})},
\end{equation}
for all $0\leq s\leq t\wedge \nu_{\sU_k}$, where $C$ is a positive constant depending only on the Heston constant 
coefficients.
This follows from the fact that
\begin{align*}
\eps D^2 u(Z^{\eps}(s))
&= \eps D^2 u(Z^{\eps}(s)) \mathbf{1}_{\left\{Y^{\eps}(s) \leq \sqrt{\eps}\right\}}
+ \eps D^2 u(Z^{\eps}(s))\mathbf{1}_{\left\{Y^{\eps}(s) > \sqrt{\eps}\right\}},
\quad \forall s \geq 0,
\end{align*}
and so,
\begin{align*}
\eps |D^2 u(Z^{\eps}(s)) |
&\leq  Y^{\eps}(s) |D^2 u(Z^{\eps}(s))|\mathbf{1}_{\left\{Y^{\eps}(s) \leq \sqrt{\eps}\right\}}
+ \eps \frac{Y^{\eps}(s)}{\sqrt{\eps}}|D^2 u(Z^{\eps}(s))|\mathbf{1}_{\left\{Y^{\eps}(s) > \sqrt{\eps}\right\}}\\
&\leq \left(\mathbf{1}_{\left\{Y^{\eps}(s) \leq \sqrt{\eps}\right\}} + \sqrt{\eps}\right)  Y^{\eps}(s) |D^2 
u(Z^{\eps}(s))|.
\end{align*}
Combining the preceding inequality with the definition \eqref{eq:stochastic_representation_Definition_A_eps} of 
$A^{\eps}$, we obtain
\eqref{eq:stochastic_representation_Estimate_A_eps_minus_A}. Since $u \in C^{1,1}_{s,\loc}(\sO\cup\Gamma_0)$, and
\[
 \mathbf{1}_{\left\{Y^{\eps}(s) \leq \sqrt{\eps}\right\}} \rightarrow 0, \quad\hbox{ as } \eps \downarrow 0,
\]
we see that by \eqref{eq:stochastic_representation_Estimate_A_eps_minus_A} yields, for each $k \geq 1$,
\begin{equation}
\label{eq:stochastic_representation_Limit_A_eps_minus_A}
\EE^z_\QQ \left[ \int_{0}^{t \wedge \nu_{\sU_k}} e^{-rs} (A^{\eps}-A) u(Z^{\eps}(s)) ds \right]
\rightarrow 0, \quad \hbox{ as } \eps \downarrow 0.
\end{equation}
In addition, using the continuity of $f$ and $u$ on compact subsets of $\sO\cup\Gamma_0$, we have
\begin{equation}
\label{eq:stochastic_representation_Convergence_u_f_BVP}
\begin{aligned}
&\EE^z_\QQ \left[e^{-r(t \wedge \nu_{\sU_k})} u(Z^{\eps} (t \wedge \nu_{\sU_k}))\right]
\rightarrow
\EE^z_\QQ \left[e^{-r(t \wedge \nu_{\sU_k})} u(Z(t \wedge \nu_{\sU_k}))\right],
\quad \hbox{ as } \eps \downarrow 0,\\
& \EE^z_\QQ \left[ \int_{0}^{t \wedge \nu_{\sU_k}} e^{-rs} f(Z^{\eps}(s)) ds \right]
\rightarrow
\EE^z_\QQ \left[ \int_{0}^{t \wedge \nu_{\sU_k}} e^{-rs} f(Z(s)) ds \right],
\quad \hbox{ as } \eps \downarrow 0.
\end{aligned}
\end{equation}
Therefore, using \eqref{eq:stochastic_representation_Limit_A_eps_minus_A} and the preceding limits, we find that 
\eqref{eq:stochastic_representation_Second_representation_BVP_2}
gives
\begin{equation}
\label{eq:stochastic_representation_Third_representation_BVP_2}
\EE^z_\QQ \left[e^{-r(t \wedge \nu_{\sU_k})} u(Z(t \wedge \nu_{\sU_k}))\right]
= u(z) - \EE^z_\QQ \left[ \int_{0}^{t \wedge \nu_{\sU_k}} e^{-rs} f(Z(s)) ds \right].
\end{equation}
Note that by letting $k$ and $t$ tend to $\infty$, we have
\begin{equation}
\label{eq:stochastic_representation_Convergence_stopping_times}
t \wedge \nu_{\sU_k} \rightarrow \nu_{\sO}, \quad \hbox{a.s.}
\end{equation}
By using the same argument as that used in the proof of Theorem 
\ref{thm:stochastic_representation_Uniqueness_BVP_elliptic} to take the limit as $k$ and $t$ tend to $\infty$ in
\eqref{eq:stochastic_representation_StatEq1}, we can take the limit as $k$ and $t$ tend to $\infty$ in 
\eqref{eq:stochastic_representation_Third_representation_BVP_2} to give
\[
\EE^z_\QQ \left[e^{-r \nu_{\sO}} g(Z(\nu_{\sO}))\right]
= u(z) - \EE^z_\QQ \left[ \int_{0}^{\nu_{\sO}} e^{-rs} f(Z(s)) ds \right].
\]
This establishes $u=u^*$, where $u^*$ is given by \eqref{eq:stochastic_representation_StochasticRep_BVP_2}, and completes 
the proof.
\end{proof}

Next, we prove \emph{existence} of solutions to problem \eqref{eq:stochastic_representation_HestonEllipticEqBVP_beta} when 
the boundary data $g$ is \emph{continuous} on suitable
portions of the boundary of $\sO$.

\begin{proof}[Proof of Theorem \ref{thm:stochastic_representation_Existence_elliptic_BVP_beta}]
Following the comments preceding problem \eqref{eq:stochastic_representation_HestonEllipticEqBVP_beta}, we need to show 
that $u^*$, given by
\eqref{eq:stochastic_representation_StochasticRep_BVP_all_beta}, is a solution to problem 
\eqref{eq:stochastic_representation_HestonEllipticEqBVP_beta}, that $u^* \in
C_{\loc}(\sO \cup \partial_{\beta}\sO) \cap C^{2}(\sO)$, and that $u^*$ satisfies the growth assumption 
\eqref{eq:stochastic_representation_Growth_elliptic}.

Notice that Lemma \ref{lem:stochastic_representation_WellDefinedStochRepVarEq}, applied with $\theta_1=\tau_{\sO}$, 
$\theta_2=\infty$ and $\psi\equiv 0$, shows that $u^*$ defined
by \eqref{eq:stochastic_representation_StochasticRep_BVP_all_beta} satisfies the growth assumption 
\eqref{eq:stochastic_representation_Growth_elliptic}. It remains to prove that
$u^*\in C_{\loc}(\sO\cup\partial_{\beta}\sO)\cap C^2(\sO)$. Notice that Theorem 
\ref{thm:stochastic_representation_Uniqueness_BVP_elliptic} implies that $u^*$ is the unique
solution to the elliptic boundary value problem \eqref{eq:stochastic_representation_HestonEllipticEqBVP_beta}, since any 
$C_{\loc}(\sO\cup\partial_{\beta}\sO)\cap C^2(\sO)$
solution must coincide with $u^*$.

By hypothesis and the definition of $\partial_{\beta}\sO$ in \eqref{eq:stochastic_representation_Boundary_beta}, we have 
$g \in C_{\loc}(\overline{\partial_{\beta}\sO})$. Since
$\overline{\partial_{\beta}\sO}$ is closed, we may use \cite[Theorem $3.1.2$]{FriedmanPDE} to extend $g$ to $\RR^2$ such 
that its extension $\tilde g \in C_{\loc}(\RR^2)$. We
organize the proof in two steps.

\setcounter{step}{0}
\begin{step} [$u^* \in C^{2+\alpha}(\mathscr{O})$]
\label{step:stochastic_representation_Elliptic_BVP_interior_regularity}
Let $\left\{\sO_k:k \in \NN\right\}$ be an increasing sequence of $C^{2+\alpha}$ subdomains of $\sO$ as in the proof of 
Theorem 
\ref{thm:stochastic_representation_Uniqueness_BVP_elliptic}. We notice that on each domain $\mathscr{O}_k$ the 
differential operator $A$ is uniformly elliptic with
$C^{\infty}(\bar{\mathscr{O}}_k)$ coefficients. From our hypotheses, we have $f \in C^{\alpha}(\bar{\sO}_k)$ and $\tilde g 
\in C(\bar{\sO}_k)$. Therefore,  \cite[Theorem
$6.13$]{GilbargTrudinger} implies that the elliptic boundary value problem
\begin{equation}
\label{eq:stochastic_representation_Elliptic_BVP_on_subdomains}
\begin{cases}
Au = f& \hbox{ on }\sO_k,
\\
u = \tilde g& \hbox{ on }\partial \sO_k.
\end{cases}
\end{equation}
admits a unique solution $u_k \in C(\bar \sO_k) \cap C^{2+\alpha}(\sO_k) $. Moreover, by\footnote{See also 
\cite[Proposition $5.7.2$]{KaratzasShreve1991}, \cite[Theorem $9.1.1$
\& Corollary $9.1.2$]{Oksendal_2003}.} \cite[Theorem $6.5.1$]{FriedmanSDE}, $u_k$ admits a stochastic representation on 
$\bar{\mathscr{O}}_k$,
\begin{equation}
\label{eq:stochastic_representation_StatEqApproxStochRep}
u_k(z) = \EE^z_\QQ \left[e^{-r\tau_{\sO_k}} \tilde g(Z(\tau_{\sO_k})) 1_{\{\tau_{\sO_k} < \infty\}}\right]
+ \EE^z_\QQ\left[\int_{0}^{\tau_{\sO_k}} e^{-rs} f(Z(s))\,ds \right].
\end{equation}
Our goal is to show that $u_k$ converges pointwise to $u^*$ on $\sO$. Recall that $\tau_k$ is an increasing sequence of 
stopping times which converges to $\tau_{\sO}$ almost
surely. Using $\tilde g \in C_{\loc}(\sO\cup\partial_{\beta}\sO)$ and the continuity of the sample paths of the Heston 
process, the growth estimate
\eqref{eq:stochastic_representation_Growth_elliptic} and Lemma 
\ref{lem:stochastic_representation_WellDefinedStochRepVarEq}, the same argument used in the proof of Theorem
\ref{thm:stochastic_representation_Uniqueness_BVP_elliptic} shows that the sequence $\left\{u_k:k \in \NN\right\}$ 
converges pointwise to $u^*$ on $\sO$.

Fix $z_0:=(x_0, y_0) \in \mathscr{O}$, and choose a Euclidean ball $B:=B(z_0, r_0)$ such that $\bar{B} \subset 
\mathscr{O}$. We denote $B_{1/2}= B(z_0, r_0/2)$. As in the proof
of Lemma \ref{lem:stochastic_representation_WellDefinedStochRepVarEq}, the sequence $u_k$ is uniformly bounded on $\bar B$ 
because it obeys
\[
|u_k(z)| \leq \bar{C}\left(1+e^{ M_1 y}+e^{ M_2 x}\right), \quad \forall z=(x, y) \in B, k \in \NN.
\]
From the interior Schauder estimates \cite[Corollary 6.3]{GilbargTrudinger}, the sequence $\left\{u_k:k \in \NN\right\}$ 
has uniformly bounded $C^{2+\alpha}(\bar{B}_{1/2})$
norms. Compactness of the embedding  $C^{2+\alpha}(\bar B_{1/2})\hookrightarrow C^{2+\gamma}(\bar B_{1/2})$, for $0\leq 
\gamma<\alpha$, shows that, after passing to a
subsequence, the sequence $\left\{u_k:k \in \NN\right\}$ converges in $C^{2+\gamma}(\bar B_{1/2})$ to $u^* \in 
C^{2+\gamma}(\bar B_{1/2})$, and so $A u^*=f$ on $\bar B_{1/2}$.
Because the subsequence has uniformly bounded $C^{2+\alpha}(\bar B_{1/2})$ norms and it converges strongly in $C^2(\bar 
B_{1/2})$ to $u^*$, we obtain that $u^* \in
C^{2+\alpha}(\bar B_{1/2})$.
\end{step}

\begin{step} [$u^* \in C_{\loc}(\mathscr{O} \cup \partial_{\beta} \sO)$]
\label{step:stochastic_representation_Elliptic_BVP_continuity}
From the previous step, we know that $u^*\in C(\sO)$, so it remains to show continuity of $u^*$ up to 
$\partial_{\beta}\sO$. We consider two cases.

\setcounter{case}{0}
\begin{case}[$u^* \in C_{\loc}(\sO\cup\Gamma_1)$, for all $\beta>0$]
\label{case:stochastic_representation_Elliptic_BVP_continuity_on_Gamma_1}
First, we show that $u^*$ is continuous up to $\Gamma_1$. We fix $z_0 \in \Gamma_1$, and let $B$ be an open ball centered 
at $z_0$, such that $\bar B \cap \partial\HH =
\emptyset$. Let $U:=B \cap \sO$. Let the function $\hat{g}$ be defined on $\partial U$ such that it coincides with $g$ on 
$\partial U\cap \partial\sO$, and it coincides with
$u^*$ on $\partial U\cap \sO$.

\begin{claim}
\label{claim:stochastic_representation_Markov_property}
The strong Markov property of the Heston process $(Z(s))_{s \geq 0}$ and the definition 
\eqref{eq:stochastic_representation_StochasticRep_BVP_all_beta} of $u^*$, implies that
\begin{equation}
\label{eqn:stochastic_representation_LocalRepresentation}
\begin{aligned}
 u^*(z) = \EE^z_\QQ \left[ e^{-r \tau_U} \hat{g}(Z(\tau_U))\right]
        + \EE^z_\QQ \left[\int_{0}^{\tau_U} e^{-rt} f(Z(t)) dt \right], \quad \forall z \in U.
\end{aligned}
\end{equation}
\end{claim}

\begin{proof}
By Corollary \ref{cor:HestonSDEStrongSolution}, the Heston stochastic differential equation 
\eqref{eq:stochastic_representation_HestonSDE} admits a unique strong solution, for
any initial point $(t,x,y) \in [0,\infty)\times\RR\times[0,\infty)$, and \cite[Theorem 
$1.16$(c)]{Feehan_Pop_mimickingdegen} shows that the solution satisfies the strong Markov
property.

Let $z \in U$, then $\tau^z_U \leq \tau^z_{\sO}$ a.s. Since $Z$ is a time-homogeneous strong Markov process, we obtain
\begin{align*}
\EE^z_{\QQ} \left[e^{-r\tau_{\sO}} g(Z(\tau_{\sO}))\right]
&=
\EE^z_{\QQ} \left[\EE^z_{\QQ} \left[e^{-r\tau_{\sO}} g(Z(\tau_{\sO}))\right] | \sF(\tau_U) \right]\\
&=
\EE^z_{\QQ} \left[e^{-r\tau_U} \EE^{Z(\tau_U)}_{\QQ} \left[e^{-r\tau_{\sO}} g(Z(\tau_{\sO}))\right] \right],
\end{align*}
which can be written as
\begin{equation}
\begin{aligned}
\label{eq:stochastic_representation_Conditioning_non_integral_term}
\EE^z_{\QQ} \left[e^{-r\tau_{\sO}} g(Z(\tau_{\sO}))\right]
&=
\EE^z_{\QQ} \left[e^{-r\tau_U}  g(Z(\tau_U))\mathbf{1}_{\{\tau_U=\tau_{\sO}\}} \right]\\
&\quad +
\EE^z_{\QQ} \left[e^{-r\tau_U} \EE^{Z(\tau_U)}_{\QQ} \left[e^{-r\tau_{\sO}} g(Z(\tau_{\sO})) \right] 
\mathbf{1}_{\{\tau_U<\tau_{\sO}\}} \right].
\end{aligned}
\end{equation}
Similarly, we have for the integral term
\begin{align*}
\EE^z_\QQ \left[\int_{0}^{\tau_{\sO}} e^{-rt} f(Z(t)) dt \right]
&=
\EE^z_\QQ \left[\int_{0}^{\tau_U} e^{-rt} f(Z(t)) dt \right]
+
\EE^z_\QQ \left[\mathbf{1}_{\{\tau_U<\tau_{\sO}\}}\int_{\tau_U}^{\tau_{\sO}} e^{-rt} f(Z(t)) dt \right],
\end{align*}
and so, by conditioning the second term in the preceding equality on $\sF(\tau_U)$ and applying the strong Markov 
property, we have
\begin{equation}
\begin{aligned}
\label{eq:stochastic_representation_Conditioning_integral_term}
\EE^z_\QQ \left[\mathbf{1}_{\{\tau_U<\tau_{\sO}\}}\int_{\tau_U}^{\tau_{\sO}} e^{-rt} f(Z(t)) dt \right]
&=
\EE^z_\QQ \left[\EE^z_\QQ \left[\mathbf{1}_{\{\tau_U\leq\tau_{\sO}\}}\int_{\tau_U}^{\tau_{\sO}} e^{-rt} f(Z(t)) dt 
\right]|\sF(\tau_U)\right]\\
&=
\EE^z_\QQ \left[\mathbf{1}_{\{\tau_U<\tau_{\sO}\}}e^{-r\tau_U}\EE^{Z(\tau_U)}_\QQ \left[\int_{0}^{\tau_{\sO}} e^{-rt} 
f(Z(t)) dt \right]\right].
\end{aligned}
\end{equation}
Combining \eqref{eq:stochastic_representation_Conditioning_non_integral_term} and 
\eqref{eq:stochastic_representation_Conditioning_integral_term} in
\eqref{eq:stochastic_representation_StochasticRep_BVP_all_beta}, we obtain
\begin{align*}
u(z) &=
\EE^z_{\QQ} \left[e^{-r\tau_U}  g(Z(\tau_U))\mathbf{1}_{\{\tau_U=\tau_{\sO}\}} \right]
+
\EE^z_\QQ \left[\int_0^{\tau_U} e^{-rt} f(Z(t)) dt \right]\\
&\quad
+\EE^z_{\QQ} \left[e^{-r\tau_U} \mathbf{1}_{\{\tau_U<\tau_{\sO}\}}\EE^{Z(\tau_U)}_{\QQ} \left[e^{-r\tau_{\sO}} 
g(Z(\tau_{\sO}))
+ \int_{0}^{\tau_{\sO}} e^{-rt} f(Z(t)) dt\right]  \right].
\end{align*}
Using again \eqref{eq:stochastic_representation_StochasticRep_BVP_all_beta} for $u^*(Z(\tau_U))$, the preceding equality 
yields 
\eqref{eqn:stochastic_representation_LocalRepresentation}. This completes the proof of Claim 
\ref{claim:stochastic_representation_Markov_property}.
\end{proof}

By \cite[Theorem $6.13$]{GilbargTrudinger} and a straightforward extension of\footnote{See also \cite[Proposition 
$5.7.2$]{KaratzasShreve1991}, \cite[Theorem $9.1.1$ \& Corollary
$9.1.2$]{Oksendal_2003}.} \cite[Theorem $6.5.1$]{FriedmanSDE} from domains with $C^{2}$ to domains with regular boundary, 
as in \cite[\S $6.2.6.A$]{Dynkin_diffsuperpde}, the
integral term in \eqref{eqn:stochastic_representation_LocalRepresentation} is the solution on $U$ of the uniformly 
elliptic partial differential equation $Au^*=f$ with
homogeneous Dirichlet boundary condition, and it is a continuous function up to $\partial U$. Notice that $\partial U$ 
satisfies the exterior sphere condition and thus $\partial
U$ is regular by\footnote{See also \cite[Proposition $4.2.15$ \& Theorem $4.2.19$]{KaratzasShreve1991}.} \cite[Theorem 
$2.4.4$]{Dynkin_diffsuperpde} (see \cite[Definitions
$2.4.1$ \& $6.2.3$]{Dynkin_diffsuperpde}
for the definition\footnote{See also \cite[Definition $4.2.9$]{KaratzasShreve1991}, \cite[Definition 
$9.2.8$]{Oksendal_2003}.}  of regular points of $\partial U$). The continuity
of the non-integral term in \eqref{eqn:stochastic_representation_LocalRepresentation} at $z_0$ follows from Corollary
\ref{cor:stochastic_representation_Continuity_stochastic_representations_with_killing_term}, as $\hat g$ is continuous at 
$z_0$ by hypotheses.
\end{case}

It remains to show that, when $0<\beta<1$, the solution $u^*$ is continuous up to $\bar\Gamma_0$.

\begin{case}[$u^*\in C_{\loc}(\sO\cup\bar\Gamma_0)$, for all $0<\beta<1$]
\label{case:stochastic_representation_Existence_elliptic_BVP_beta_continuity_Gamma_0}
Let $z_0=(x_0, 0) \in \bar \Gamma_0$. We denote by $\theta^z $ the first time the process started at $z=(x,y) \in 
\mathscr{O}$ hits $y=0$. Obviously, we have
\begin{equation}
\label{eq:stochastic_representation_Inequality_stopping_times}
\tau^z_{\sO} \leq \theta^z \leq T^y_0 \quad \hbox{ a.s.,}
\end{equation}
where $T^y_0$ is given by \eqref{defn:stochastic_representation_CIR_First_hitting_time_a}. For $\beta \in (0,1)$, it 
follows from
\eqref{eq:stochastic_representation_CIR_limit_hitting_time_zero} and the preceding inequality between stopping times, that 
$\theta^z$ converges to $0$, as $y$ goes to $0$,
uniformly in $x\in \RR$. Therefore, the integral term in (\ref{eqn:stochastic_representation_LocalRepresentation}) 
converges to zero. Next, we want to show that the non-integral
term in (\ref{eqn:stochastic_representation_LocalRepresentation}) converges to $g(z_0)$. We rewrite that term as
\begin{equation}
\label{eq:stochastic_representation_BVP_Cont_Ident1}
\begin{aligned}
\EE^z_\QQ \left[ e^{-r\tau_{\sO}} g(Z(\tau_{\sO}))\right] - g(z_0)
&= \EE^z_\QQ \left[ e^{-r\tau_{\sO}} \left(g(Z(\tau_{\sO})) - g(z_0)\right)\right]\\
&\quad + g(z_0) \left(1- \EE^z_\QQ \left[e^{-r\tau_{\sO}}\right]\right).
\end{aligned}
\end{equation}
From the observation that $\tau^z_{\sO} \leq \theta^z$ a.s., we see that
\begin{equation}
\label{eq:stochastic_representation_BVP_Cont_Conv1}
\EE^z_\QQ\left[e^{-r\tau_{\sO}}\right] \rightarrow 1, \hbox{ as } z \rightarrow z_0.
\end{equation}
By \eqref{eq:stochastic_representation_BVP_Cont_Ident1}, it remains to show that $\EE^z_\QQ \left[ e^{-r\tau_{\sO}} 
\left(g(Z(\tau_{\sO})) - g(z_0)\right)\right]$ converges to
zero, as $z \in \mathscr{O}$ converges to $z_0$. We fix $\eps > 0$ and choose $\delta_1>0$ such that
\begin{equation}
\label{eq:stochastic_representation_Cont_g}
|g(z)-g(z_0)| < \eps,\quad \forall z \in B(z_0, \delta_1) \cap \partial \sO.
\end{equation}
From \cite[Equation (5.3.18) in Problem 5.3.15 ]{KaratzasShreve1991}, there is a positive constant $C_1$, depending on 
$z_0$ and $\delta_1$, such that
\begin{equation*}
\sup_{z \in B(z_0, \delta_1) \cap \mathscr{O}} \EE^z_\QQ \left[ \sup_{0 \leq s \leq t}|Z(s)-z|\right] \leq C_1 \sqrt{t},
\end{equation*}
from where it follows
\begin{equation}
\label{eq:stochastic_representation_StatEqCont1}
\sup_{z \in B(z_0, \delta_1) \cap \mathscr{O}} \QQ^z \left(\sup_{0 \leq s \leq t}|Z(s)-z| > \delta_1/2 \right) \leq 
\frac{2C_1 \sqrt{t}}{\delta_1}.
\end{equation}
Next, we choose $t>0$ sufficiently small such that
\begin{equation}
\label{eq:stochastic_representation_StatEqCont2}
\frac{2C_1 \sqrt{t}}{\delta_1} < \eps,
\end{equation}
and, by \eqref{eq:stochastic_representation_Inequality_stopping_times} and 
\eqref{eq:stochastic_representation_CIR_limit_hitting_time_zero}, we may choose $\delta_2>0$
sufficiently small such that
\begin{equation}
\label{eq:stochastic_representation_StatEqCont3}
\QQ\left(T^{\delta_2}_0 > t\right) < \eps.
\end{equation}
Let $\delta:= \min \{\delta_1/2, \delta_2\}$. We rewrite
\begin{align*}
\label{eq:stochastic_representation_Elliptic_BVP_Eq1}
e^{-r\tau_{\sO}} \left(g(Z(\tau_{\sO})) - g(z_0)\right)
&=  e^{-r\tau_{\sO}} \left(g(Z(\tau_{\sO})) - g(z_0)\right) \mathbf{1}_{\{\tau_{\sO} \leq t\}}\\
&\quad  + e^{-r\tau_{\sO}} \left(g(Z(\tau_{\sO})) - g(z_0)\right) \mathbf{1}_{\{\tau_{\sO} > t\}}
\end{align*}
to give
\begin{equation}
\label{eq:stochastic_representation_Elliptic_BVP_Eq1}
\begin{aligned}
e^{-r\tau_{\sO}} \left(g(Z(\tau_{\sO})) - g(z_0)\right)
&=  e^{-r\tau_{\sO}} \left(g(Z(\tau_{\sO})) - g(z_0)\right) \mathbf{1}_{\{\tau_{\sO} \leq t, \sup_{0\leq s\leq t} |Z(s)-z| 
< \delta_1/2\}}  \\
&\quad+
    e^{-r\tau_{\sO}} \left(g(Z(\tau_{\sO})) - g(z_0)\right) \mathbf{1}_{\{\tau_{\sO} \leq t, \sup_{0\leq s\leq t} |Z(s)-z| 
    \geq \delta_1/2\}}  \\
&\quad +
    e^{-r\tau_{\sO}} \left(g(Z(\tau_{\sO})) - g(z_0)\right) \mathbf{1}_{\{\tau_{\sO} > t\}}\\
\end{aligned}
\end{equation}
By \eqref{eq:stochastic_representation_Cont_g}, we have for all $z \in B(z_0,\delta)\cap\sO$
\begin{align}
\label{eq:stochastic_representation_Elliptic_BVP_Eq2}
\EE^z_\QQ \left[ \left|g(Z(\tau_{\sO})) - g(z_0)\right| \mathbf{1}_{\{\tau_{\sO} \leq t,  \hbox{ }\sup_{0\leq s\leq t} 
|Z(s)-z| < \delta_1/2\}} \right] <\eps.
\end{align}
We choose $p>1$ such that  $p M_1<\mu$ and $pM_2<1$. Notice that this is possible because we assumed the coefficients $ 
M_1<\mu$ and $ M_2<1$. Then, from the growth estimate
\eqref{eq:stochastic_representation_Growth_elliptic} for $g$, we have
\begin{align*}
\left|e^{-r \tau_{\sO}} g(Z(\tau_{\sO}))\right|^p
&\leq C e^{-rp\tau_{\sO}} \left(1+e^{pM_1 Y(\tau_{\sO})} + e^{pM_2 X(\tau_{\sO})}\right).
\end{align*}
From the inequality \eqref{eq:stochastic_representation_Supermartingale_property_Y} with $c=p M_1<\mu$ and property
\eqref{eq:stochastic_representation_Supermartingale_property_X} applied with $c=p  M_2\in (0,1)$, we obtain using the 
condition $M_1 \leq r/(\kappa\vartheta)$
\begin{align*}
\EE^z_\QQ \left[\left|e^{-r\tau_{\sO}} g(Z (\tau_{\sO}))\right|^p\right]
&\leq C \left(1+e^{pM_1 y} + e^{pM_2 x}\right).
\end{align*}
Let $C_2>0$ be an bound on the right-hand side of the preceding inequality, for all $z=(x,y)\in B(z_0,\delta)\cap\sO$.
By the Cauchy-Schwarz inequality, we have
\begin{align*}
\left|\EE^z_\QQ \left[ e^{-r\tau_{\sO}} \left(g(Z(\tau_{\sO})) - g(z_0)\right) \mathbf{1}_{\{\tau_{\sO} > t\}} \right]  
\right|
& \leq \EE^z_\QQ \left[ e^{-rp\tau_{\sO}} \left|g(Z(\tau_{\sO})) - g(z_0)\right|^p \right]^{1/p}
       \QQ^z\left(\tau_{\sO} > t\right)^{1/p'},
\end{align*}
where $p'>1$ denotes the conjugate exponent of $p$. Using the fact that $\tau^z_{\sO} \leq T^{\delta_2}_0$ from 
\eqref{eq:stochastic_representation_Inequality_stopping_times} and
\eqref{eq:stochastic_representation_StatEqCont3}, we obtain in the preceding inequality
\begin{equation}
\begin{aligned}
\label{eq:stochastic_representation_Elliptic_BVP_Eq3}
\left|\EE^z_\QQ \left[ e^{-r\tau_{\sO}} \left(g(Z(\tau_{\sO})) - g(z_0)\right) \mathbf{1}_{\{\tau_{\sO} > t\}} \right]  
\right|
&\leq  2C_2^{1/p} \QQ^z\left(T_0 > t\right)^{1/p'}\\
& \leq
  2C_2^{1/p} \eps^{1/p'},\quad \forall z \in B(z_0,\delta)\cap\sO,
\end{aligned}
\end{equation}
From the inequality,
\begin{align*}
&\left|\EE^z_\QQ \left[  e^{-r\tau_{\sO}} \left(g(Z(\tau_{\sO})) - g(z_0)\right) \mathbf{1}_{\{\tau_{\sO} \leq t, 
\sup_{0\leq s\leq t} |Z(s)-z|
\geq \delta_1/2\}} \right]\right|\\
&\qquad \leq \EE^z_\QQ \left[ e^{-rp\tau_{\sO}} \left|g(Z(\tau_{\sO})) - g(z_0)\right|^p \right]^{1/p}
             \QQ^z \left(\sup_{0\leq s\leq t} |Z(s)-z| \geq \delta_1/2\right)^{1/p'},
\end{align*}
the inequalities \eqref{eq:stochastic_representation_StatEqCont1} and \eqref{eq:stochastic_representation_StatEqCont2} and 
definition of $C_2$ yield
\begin{align}
\label{eq:stochastic_representation_Elliptic_BVP_Eq4}
\left|\EE^z_\QQ \left[  e^{-r\tau_{\sO}} \left(g(Z(\tau_{\sO})) - g(z_0)\right) \mathbf{1}_{\{\tau_{\sO} \leq t, 
\sup_{0\leq s\leq t} |Z(s)-z|
\geq \delta_1/2\}} \right]\right|
& \leq 2 C_2^{1/p} \eps^{1/p'}.
\end{align}
Substituting \eqref{eq:stochastic_representation_Elliptic_BVP_Eq2}, \eqref{eq:stochastic_representation_Elliptic_BVP_Eq3}, 
and
\eqref{eq:stochastic_representation_Elliptic_BVP_Eq4} in \eqref{eq:stochastic_representation_Elliptic_BVP_Eq1}, we obtain
\[
\EE \left[ e^{-r\tau_{\sO}} \left(g(Z(\tau_{\sO})) - g(z_0)\right)\right]< \left(1+4 C_2^{1/p}\right)\left( \eps + 
\eps^{1/p'}\right),
\quad \forall z \in B(z_0,\delta)\cap\sO,
\]
and so $u^*$ is continuous up to $\bar\Gamma_0$, when $0<\beta<1$.
\end{case}
This concludes the proof that $u^*\in C_{\loc}(\sO\cup\partial_{\beta}\sO)$, for all $\beta>0$.
\end{step}
This completes the proof of Theorem \ref{thm:stochastic_representation_Existence_elliptic_BVP_beta}.
\end{proof}

We now prove \emph{existence} of solutions to problem \eqref{eq:stochastic_representation_HestonEllipticEqBVP_beta} when 
the boundary data $g$ is \emph{H\"older continuous} on
suitable portions of the boundary of $\sO$.

\begin{proof}[Proof of Theorem \ref{thm:stochastic_representation_Existence_elliptic_BVP_beta_Holder_continuous_boundary_data}]
The proof of the theorem is the same as that of Theorem \ref{thm:stochastic_representation_Existence_elliptic_BVP_beta}, 
with the exception that Case
\ref{case:stochastic_representation_Elliptic_BVP_continuity_on_Gamma_1} of Step 
\ref{step:stochastic_representation_Elliptic_BVP_continuity} can be simplified by applying the
classical boundary Schauder estimates. Also, instead of using the sequence of subdomains $\left\{\sO_k:k\in\NN\right\}$ 
precompactly contained in $\sO$, as in the proof of
Theorem \ref{thm:stochastic_representation_Uniqueness_BVP_elliptic}, we consider an increasing sequence, $\left\{\sD_k: k 
\in \NN\right\}$, of $C^{2+\alpha}$ subdomains of $\sO$
(see \cite[Definition \S $6.2$]{GilbargTrudinger}) such that each $\sD_k$ satisfies
\begin{equation}
\label{defn:stochastic_representation_Sequence_D_k}
\sO \cap (-k,k)\times(1/k,k) \subset \sD_k \subset \sO \cap (-2k,2k)\times(1/(2k),2k), \quad \forall k \in \NN,
\end{equation}
and
\begin{equation*}
\bigcup_{k \in \NN} \sD_k = \sO.
\end{equation*}
Since $\Gamma_1$ is assumed to be of class $C^{2+\alpha}$, we may choose $\sD_k$ to be of class $C^{2+\alpha}$.

Let $z_0 \in \Gamma_1$, and $r_0>0$ small enough such that $B(z_0,r_0) \cap \Gamma_0=\emptyset$. Let
\[
D := B(z_0,r_0) \cap \sO \hbox{  and  } D' := B(z_0,r_0/2) \cap \sO.
\]
By \eqref{defn:stochastic_representation_Sequence_D_k}, we may choose $k_0\in\NN$ large enough such that $D \subset 
\sD_k$, for all $k \geq k_0$. Using $f \in C^{\alpha}(\bar
D)$, $g \in C^{2+\alpha}(\bar D)$ and applying \cite[Corollary $6.7$]{GilbargTrudinger}, and the fact that $u_k$ solves
\eqref{eq:stochastic_representation_Elliptic_BVP_on_subdomains} 
with $g$ replacing $\tilde g$ and $\sD_k$ replacing $\sO_k$, we have
\begin{equation}
\label{eq:stochastic_representation_BVP_boundary_estimates}
\|u_k\|_{C^{2+\alpha}(\bar D')} \leq C \left(\|u_k\|_{C(\bar D)}+
\|g\|_{C^{2+\alpha}(\bar D)}+\|f\|_{C^{\alpha}(\bar D)}\right),
\quad \forall k \geq k_0,
\end{equation}
where $C>0$ is a constant depending only on the coefficients of $A$, and the domains $D$ and $D'$.
Combining the preceding inequality with the uniform bound on the $C(\bar D)$ norms of the sequence 
$\left\{u_k:k\in\NN\right\}$, resulting from Lemma
\ref{lem:stochastic_representation_WellDefinedStochRepVarEq}, the compactness of the embedding of $C^{2+\alpha}(\bar D') 
\hookrightarrow C^{2+\gamma}(\bar D')$, when
$0\leq\gamma<\alpha$, implies that a subsequence converges strongly 
to $u^*$ in $C^{2+\gamma}(\bar D')$. Therefore, $u^* \in C^{2+\gamma}(\bar D')$, 
and $Au^*=f$ on $D'$ and $u^*=g$ on $\partial
D'\cap\Gamma_1$. Moreover, $u^* \in C^{2+\alpha}(\bar D')$, since $u_k \in C^{2+\alpha}(\bar D')$, for all $k \geq k_0$, 
and the sequence converges in 
$C^{2+\gamma}(\bar D')$ to $u^*$.
Combining the boundary estimate \eqref{eq:stochastic_representation_BVP_boundary_estimates} with Step 
\ref{step:stochastic_representation_Elliptic_BVP_interior_regularity} in the
proof of Theorem \ref{thm:stochastic_representation_Existence_elliptic_BVP_beta}, we obtain the conclusion that $u^* \in 
C^{2+\alpha}_{\loc}(\sO\cup\Gamma_1)$.
\end{proof}

\begin{rmk}[Validity of the stochastic representation for strong solutions]
The stochastic representation \eqref{eq:stochastic_representation_StochasticRep_BVP_2} for solutions to problem 
\eqref{eq:stochastic_representation_HestonEllipticEqBVP} with
boundary condition along $\Gamma_1$ is valid if we replace the condition $u \in C_{\loc}(\sO\cup\Gamma_1) \cap 
C^2(\sO)\cap C^{1,1}_{s,\loc}(\sO\cup\Gamma_0)$ in the hypotheses
of Theorem \ref{thm:stochastic_representation_Uniqueness_BVP_elliptic_beta_less_than_1}, with the weaker condition $u \in 
C_{\loc}(\sO\cup\Gamma_1) \cap W^{2,2}_{\loc}(\sO)\cap
C^{1,1}_{s,\loc}(\sO\cup\Gamma_0)$.
\end{rmk}

\section{Elliptic obstacle problem}
\label{sec:stochastic_representation_Elliptic_obstacle_value_problem}
This section contains the proofs of Theorems \ref{thm:stochastic_representation_Uniqueness_elliptic_OP} and
\ref{thm:stochastic_representation_Uniqueness_elliptic_OP_beta_less_than_1}. As in problem 
\eqref{eq:stochastic_representation_HestonEllipticEqBVP_beta}, the questions of
\emph{uniqueness} of solutions to problem \eqref{eq:stochastic_representation_Elliptic_obstacle_problem} with Dirichlet 
boundary condition along $\Gamma_1$, when $\beta\geq 1$,
and along $\partial\sO$, when $0<\beta<1$, are similar in nature. We can conveniently treat them together as
\begin{equation}
\label{eq:stochastic_representation_Elliptic_obstacle_problem_beta}
\begin{cases}
           \min\left\{Au-f,u-\psi\right\}=0 &\hbox{ on }\sO,\\
           u=g &\hbox{ on }\partial_{\beta}\sO,
\end{cases}
\end{equation}
where $\partial_{\beta}\sO$ is given by \eqref{eq:stochastic_representation_Boundary_beta}.

\begin{proof}[Proof of Theorem \ref{thm:stochastic_representation_Uniqueness_elliptic_OP}]
Lemma \ref{lem:stochastic_representation_WellDefinedStochRepVarEq} indicates that $u^*$ given by 
\eqref{eq:stochastic_representation_Stochastic_representation_EOP_1} satisfies
\eqref{eq:stochastic_representation_Growth_elliptic}, so the growth assumption on $u$ in Theorem 
\ref{thm:stochastic_representation_Uniqueness_elliptic_OP} is justified.

By the preceding remarks, it suffices to prove that the stochastic representation 
\eqref{eq:stochastic_representation_Stochastic_representation_EOP_1} holds for solutions $u \in
C_{\loc}(\sO\cup\partial_{\beta}\sO) \cap C^2(\sO) $ to problem 
\eqref{eq:stochastic_representation_Elliptic_obstacle_problem_beta}. We consider the two situations: $u \geq u^*$
and $u \leq u^*$ on $\sO\cup\partial_{\beta}\sO$, where $u^*$ is defined by 
\eqref{eq:stochastic_representation_Stochastic_representation_EOP_1}.

\setcounter{step}{0}
\begin{step}[Proof that $u \geq u^*$ on $\sO\cup\partial_{\beta}\sO$]
\label{step:stochastic_representation_OP_u_geq_u_*}
Let $\left\{\sO_k:k \in \NN\right\}$ be an increasing sequence of $C^{2+\alpha}$ subdomains of $\sO$ as in the proof of 
Theorem 
\ref{thm:stochastic_representation_Uniqueness_BVP_elliptic}. Since $u \in C^2(\sO)$, It\^o's lemma
\cite[Theorems $3.3.3$ \& $3.3.6$]{KaratzasShreve1991} yields, for any stopping time $\theta \in \mathscr{T}$,
\begin{equation}
\label{eq:stochastic_representation_StatIneqStochRep1}
\EE^z_\QQ \left[e^{-r(\theta \wedge\tau_{\sO_k})} u(Z (\theta \wedge \tau_{\sO_k}))\right]
=    u(z) - \EE^z_\QQ \left[ \int_{0}^{\theta \wedge \tau_{\sO_k}} e^{-rs} Au(Z(s)) ds \right].
\end{equation}
By splitting the right-hand side in the preceding identity,
\begin{align*}
&\EE^z_\QQ \left[e^{-r(\theta \wedge\tau_{\sO_k})} u(Z (\theta \wedge \tau_{\sO_k}))\right]\\
&\qquad= \EE^z_\QQ \left[e^{-r \tau_{\sO_k}} u(Z (\theta \wedge \tau_{\sO_k}))\mathbf{1}_{\{\tau_{\sO_k} \leq 
\theta\}}\right]
+ \EE^z_\QQ \left[e^{-r\theta} u(Z (\theta \wedge \tau_{\sO_k}))\mathbf{1}_{\{\tau_{\sO_k} > \theta\}}\right],
\end{align*}
and using  $u\geq\psi$ on $\sO$ and $Au \geq f$ a.e. on $\sO$, the identity 
\eqref{eq:stochastic_representation_StatIneqStochRep1} gives
\begin{equation}
\label{eq:stochastic_representation_Obstacle_ineq_stopping_times}
\begin{aligned}
u(z) &\geq  \EE^z_\QQ \left[e^{-r\theta} \psi(Z (\theta)) \mathbf{1}_{\{\theta < \tau_{\sO_k}\}}\right] \\
     &\quad
         + \EE^z_\QQ \left[e^{-r\tau_{\sO_k}} u(Z (\tau_{\sO_k})) \mathbf{1}_{\{\tau_{\sO_k} \leq \theta\}}\right]
         + \EE^z \left[ \int_{0}^{\theta \wedge \tau_{\sO_k}} e^{-rs} f(Z(s)) \,ds \right].
\end{aligned}
\end{equation}
Just as in the proof of Theorem \ref{thm:stochastic_representation_Uniqueness_BVP_elliptic}, the collections of random 
variables
\begin{equation*}
\left\{e^{-r\theta} \psi(Z (\theta)) \mathbf{1}_{\{\theta < \tau_{\sO_k}\}}: k \in \NN\right\}
\hbox{  and  }
\left\{e^{-r\tau_{\sO_k}} u(Z (\tau_{\sO_k})) \mathbf{1}_{\{\tau_{\sO_k} \leq \theta\}}: k \in \NN\right\}
\end{equation*}
are uniformly integrable because $u$ and $\psi$ satisfy the pointwise growth estimate 
\eqref{eq:stochastic_representation_Growth_elliptic}. From the continuity of $u$ and $\psi$
on $\sO\cup\partial_{\beta}\sO$, we also have the a.s. convergence,
\begin{align*}
& e^{-r\theta} \psi(Z (\theta)) \mathbf{1}_{\{\theta < \tau_{\sO_k}\}}
\rightarrow
e^{-r\theta} \psi(Z (\theta)) \mathbf{1}_{\{\theta < \tau_{\sO}\}},
\quad\hbox{as  } k \rightarrow \infty,\\
& e^{-r\tau_{\sO_k}} u(Z (\tau_{\sO_k})) \mathbf{1}_{\{\tau_{\sO_k} \leq \theta\}}
\rightarrow
e^{-r\tau_{\sO}} u(Z (\tau_{\sO})) \mathbf{1}_{\{\tau_{\sO}\leq \theta\}},
\quad\hbox{as  } k \rightarrow \infty.
\end{align*}
Therefore, by \cite[Theorem $16.13$]{Billingsley_1986}, we can take limit as $k$ tends to $\infty$ in inequality 
\eqref{eq:stochastic_representation_Obstacle_ineq_stopping_times}
to obtain, for all $\theta\in\sT$,
\begin{align*}
u(z) &\geq       \EE^z_\QQ \left[e^{-r\theta} \psi(Z (\theta)) \mathbf{1}_{\{\theta < \tau_{\sO}\}}\right]
               + \EE^z_\QQ \left[e^{-r\tau_{\sO}} u(Z (\tau_{\sO})) \mathbf{1}_{\{\tau_{\sO} \leq \theta\}}\right]\\
     &\quad    + \EE^z_\QQ \left[ \int_{0}^{\theta \wedge \tau_{\sO}} e^{-rs} f(Z(s)) \,ds \right],
\end{align*}
which yields $u \geq u^*$ on $\sO\cup\partial_{\beta}\sO$.
\end{step}

\begin{step}[Proof that $u \leq u^*$ on $\sO\cup\partial_{\beta}\sO$]
The continuation region,
\begin{equation}
\label{eq:stochastic_representation_Continuation_region}
\mathscr{C} := \{ u > \psi\},
\end{equation}
is an open set by the continuity of $u$ and $\psi$. We denote the first exit time of $Z^{t,z}$ from the continuation 
region, $\mathscr{C}$,  by
\begin{equation}
\label{eq:stochastic_representation_Exit_time_continuation_region}
\tilde{\tau}^{t,z} := \left\{s \geq t : Z^{t,z}(s) \notin \sC\right\},
\end{equation}
and write $\tilde \tau = \tilde\tau^{t,z}$ for brevity.
This is indeed a stopping time because the process $Z^{t,z}$ is continuous and $\sC$ is open. By the same argument used in 
Step \ref{step:stochastic_representation_OP_u_geq_u_*}
with $\theta$ replaced by $\tilde{\tau}$, we obtain that all inequalities hold with equalities because $u(Z(\tilde{\tau})) 
= \psi(Z(\tilde{\tau}))$ and $Au=f$ on the continuation
region, $\sC$. Therefore,
\begin{align*}
u(z) &= \EE^z_\QQ \left[e^{-r\tilde{\tau}} \psi(Z(\tilde{\tau})) \mathbf{1}_{\{\tilde{\tau} < \tau_{\sO}\}}\right]
     + \EE^z_\QQ \left[e^{-r\tau_{\sO}} g(Z (\tau_{\sO})) \mathbf{1}_{\{\tau_{\sO} \leq \tilde{\tau_{\sO}}\}}\right]\\
     &\quad+ \EE^z_\QQ \left[\int_{0}^{\tilde{\tau} \wedge \tau_{\sO}} e^{-rs} f(Z(s)) \,d s \right],
\end{align*}
which implies that $u \leq u^*$.
\end{step}

By combining the preceding two steps, we obtain the stochastic representation 
\eqref{eq:stochastic_representation_Stochastic_representation_EOP_1} of solutions to problem
\eqref{eq:stochastic_representation_Elliptic_obstacle_problem_beta}, and hence the uniqueness assertion.
\end{proof}

\begin{proof}[Proof of Theorem \ref{thm:stochastic_representation_Uniqueness_elliptic_OP_beta_less_than_1}]
Lemma \ref{lem:stochastic_representation_WellDefinedStochRepVarEq} indicates that $u^*$ given by 
\eqref{eq:stochastic_representation_Stochastic_representation_EOP_2} satisfies
\eqref{eq:stochastic_representation_Growth_elliptic}, so the growth assumption on $u$ in Theorem 
\ref{thm:stochastic_representation_Uniqueness_elliptic_OP} is justified.

Our goal is to show that if $0<\beta<1$ and $u \in C_{\loc}(\sO\cup\Gamma_1)\cap C^2(\sO)\cap 
C^{1,1}_{s,\loc}(\sO\cup\Gamma_0)$ is a solution to problem
\eqref{eq:stochastic_representation_Elliptic_obstacle_problem} with Dirichlet boundary condition 
\eqref{eq:stochastic_representation_Elliptic_compatibility_g_psi_beta_geq_1}
along $\Gamma_1$, and satisfying the growth estimate \eqref{eq:stochastic_representation_Growth_elliptic}, then it admits 
the stochastic representation
\eqref{eq:stochastic_representation_Stochastic_representation_EOP_2}. As in the proof of Theorem 
\ref{thm:stochastic_representation_Uniqueness_elliptic_OP}, we consider the
following two cases.

\setcounter{step}{0}
\begin{step}[Proof that $u \geq u^*$ on $\sO\cup\Gamma_1$]
\label{step:stochastic_representation_elliptic_OP_beta_less_than_1_u_geq_u_*}
Let $\eps>0$ and $\left\{\sU_k: k \in \NN\right\}$ be the collection of increasing subdomains as in
\eqref{eq:stochastic_representation_Sequence_subdomains_with_degenerate_boundary}. By applying It\^o's lemma, we obtain, 
for all $t>0$ and $\theta\in\sT$,
\begin{equation}
\label{eq:stochastic_representation_First_representation_EOP_2}
u(z)=\EE^z_\QQ \left[e^{-r(t \wedge \nu_{\sU_k}\wedge\theta)} u(Z^{\eps} (t \wedge \nu_{\sU_k}\wedge\theta))\right]
+ \EE^z_\QQ \left[ \int_{0}^{t \wedge \nu_{\sU_k}\wedge\theta} e^{-rs} A^{\eps} u(Z^{\eps}(s)) ds \right],
\end{equation}
where $\nu_{\sU_k}$ is given by \eqref{eq:stochastic_representation_Stopping_time_domain_with_degenerate_boundary} and 
$Z^{\eps}$ is defined in
\eqref{eq:stochastic_representation_Z_eps}, and $A^{\eps}$ is defined by 
\eqref{eq:stochastic_representation_Definition_A_eps}. By
\eqref{eq:stochastic_representation_Elliptic_obstacle_problem} and \eqref{eq:stochastic_representation_Definition_A_eps}, 
preceding identity gives
\begin{equation}
\label{eq:stochastic_representation_Second_representation_EOP_2}
\begin{aligned}
u(z)&\geq \EE^z_\QQ \left[e^{-r(t \wedge \nu_{\sU_k}\wedge\theta)} u(Z^{\eps} (t \wedge \nu_{\sU_k}\wedge\theta))\right] 
\\
&\quad + \EE^z_\QQ \left[ \int_{0}^{t \wedge \nu_{\sU_k}\wedge\theta} e^{-rs} f(Z^{\eps}(s)) ds \right]
+ \EE^z_\QQ \left[ \int_{0}^{t \wedge \nu_{\sU_k}\wedge\theta} e^{-rs} (A^{\eps}-A) u(Z^{\eps}(s)) ds \right].
\end{aligned}
\end{equation}
First, we take the limit as $\eps$ tends to $0$ in \eqref{eq:stochastic_representation_Second_representation_EOP_2}. We 
can assume without loss of generality that $\eps<1/k$, for
any fixed $k \in \NN$. The residual term $(A^{\eps}-A)u$ then obeys estimate 
\eqref{eq:stochastic_representation_Estimate_A_eps_minus_A} because $u \in
C^{1,1}_{s,\loc}(\sO\cup\Gamma_0)$. Therefore, \eqref{eq:stochastic_representation_Limit_A_eps_minus_A} also holds in the 
present case. In addition, using the continuity of $f$,
$u$, $Du$ and $y D^2 u$ on compact subsets of $\sO\cup\Gamma_0$, we see that 
\eqref{eq:stochastic_representation_Convergence_u_f_BVP} holds, and so, by taking limit as
$\eps\downarrow 0$ in \eqref{eq:stochastic_representation_Second_representation_EOP_2},
\begin{equation}
\label{eq:stochastic_representation_Third_representation_EOP_2}
u(z)\geq \EE^z_\QQ \left[e^{-r(t \wedge \nu_{\sU_k}\wedge\theta)} u(Z (t \wedge \nu_{\sU_k}\wedge\theta))\right]
 + \EE^z_\QQ \left[ \int_{0}^{t \wedge \nu_{\sU_k}\wedge\theta} e^{-rs} f(Z(s)) ds \right].
\end{equation}
Finally, letting $k$ and $t$ tend to $\infty$ and using the convergence 
\eqref{eq:stochastic_representation_Convergence_stopping_times}, the same argument employed in the proof
of Theorem \ref{thm:stochastic_representation_Uniqueness_BVP_elliptic} can be applied to conclude that $u \geq u^*$ on 
$\sO\cup\Gamma_1$, where $u^*$ is given by
\eqref{eq:stochastic_representation_Stochastic_representation_EOP_2}.
\end{step}

\begin{step}[Proof that $u \leq u^*$ on $\sO\cup\Gamma_1$]
We choose $\theta=\tilde\tau$ in the preceding step, where $\tilde\tau$ is defined by 
\eqref{eq:stochastic_representation_Exit_time_continuation_region}. By the definition
\eqref{eq:stochastic_representation_Continuation_region} of the continuation region, $\sC$, and the obstacle problem
\eqref{eq:stochastic_representation_Elliptic_obstacle_problem}, we notice that inequalities 
\eqref{eq:stochastic_representation_Second_representation_EOP_2} and
\eqref{eq:stochastic_representation_Third_representation_EOP_2} hold with equality and so it follows as in Step
\ref{step:stochastic_representation_elliptic_OP_beta_less_than_1_u_geq_u_*} that
$u \leq u^*$ on $\sO\cup\Gamma_1$.
\end{step}
This completes the proof.
\end{proof}

\begin{rmk}[Validity of the stochastic representation for strong solutions]
The stochastic representation \eqref{eq:stochastic_representation_Stochastic_representation_EOP_1} of solutions to 
problem
\eqref{eq:stochastic_representation_Elliptic_obstacle_problem_beta}, when $\beta>0$, holds under the weaker assumption 
that $u \in C_{\loc}(\sO\cup\partial_{\beta}\sO) \cap
W^{2,2}_{\loc}(\sO)$. Similarly, the stochastic representation 
\eqref{eq:stochastic_representation_Stochastic_representation_EOP_2} of solutions to problem
\eqref{eq:stochastic_representation_Elliptic_obstacle_problem} with Dirichlet boundary condition 
\eqref{eq:stochastic_representation_Elliptic_compatibility_g_psi_beta_geq_1}
along $\Gamma_1$, when $0<\beta<1$, holds under the weaker assumption that $u \in C_{\loc}(\sO\cup\Gamma_1) \cap 
C^{1,1}_{s,\loc}(\sO\cup\Gamma_0)\cup W^{2,2}_{\loc}(\sO)$. In
each case, we would replace the application of the classical It\^o lemma
\cite[Theorems $3.3.3$ \& $3.3.6$]{KaratzasShreve1991} with \cite[Identity ($8.62$) in Theorem $2.8.5$]{Bensoussan_Lions}, or 
we could apply an approximation argument involving
$C^2(\sO)$ functions.
\end{rmk}

\section{Parabolic terminal/boundary value problem}
\label{sec:stochastic_representation_Parabolic_BVP}
This section contains the proofs of Theorems \ref{thm:stochastic_representation_Uniqueness_BVP_parabolic} and
\ref{thm:stochastic_representation_Uniqueness_BVP_parabolic_beta_less_than_1} and an \emph{existence} result in Theorem
\ref{thm:stochastic_representation_Existence_parabolic_BVP_beta}. Because the Heston process satisfies the strong Markov 
property, it suffices to prove the stochastic
representation of solutions to the terminal value problem for $T$ as small as we like. In particular, without loss of 
generality, we can choose $T$ such that

\begin{hyp}
\label{hyp:stochastic_representation_T_small_parabolic}
There is a constant $p_0 >1$ such that
\begin{enumerate}
\item Condition \eqref{eq:stochastic_representation_Power_p_X} in Lemma \ref{lem:stochastic_representation_Power_X} is 
    satisfied for $p:=p_0 M_2$, where $ M_2 \in [0,1]$ is the
    constant appearing in \eqref{eq:stochastic_representation_Growth_parabolic};
\item One has $p_0 M_1 \leq \mu$, where $M_1 \in [0,\mu)$ in \eqref{eq:stochastic_representation_Growth_parabolic}.
\end{enumerate}
\end{hyp}

As in \S \ref{sec:stochastic_representation_Elliptic_boundary_value_problem}, we first prove \emph{uniqueness} of 
solutions to the parabolic terminal/boundary value problems
\eqref{eq:stochastic_representation_Parabolic_BVP} with different possible Dirichlet boundary conditions depending on the 
parameter $\beta$. The proofs are similar those of
Theorems \ref{thm:stochastic_representation_Uniqueness_BVP_elliptic} and 
\ref{thm:stochastic_representation_Uniqueness_BVP_elliptic_beta_less_than_1}.

The \emph{existence} and \emph{uniqueness} of solutions to problem \eqref{eq:stochastic_representation_Parabolic_BVP} with 
boundary condition
\eqref{eq:stochastic_representation_Parabolic_BVP_boundary_condition_Gamma_1}, when $\beta\geq 1$, and with boundary 
condition
\eqref{eq:stochastic_representation_Parabolic_BVP_boundary_condition_whole_boundary}, when $0<\beta<1$, are similar in 
nature. By analogy with our treatment of problem
\eqref{eq:stochastic_representation_HestonEllipticEqBVP_beta}, we define
\begin{align}
\label{eq:stochastic_representation_Boundary_beta_parabolic}
\eth_{\beta} Q :=
\begin{cases}
\eth^1 Q     & \hbox{ if } \beta \geq 1,\\
\eth Q       & \hbox{ if } 0<\beta <1,
\end{cases}
\end{align}
where we recall that $Q:=(0,T)\times\sO$. The preceding problems can then be formulated as
\begin{align}
\label{eq:stochastic_representation_Parabolic_BVP_beta}
           -u_t+Au&=f     \quad\hbox{ on }Q,\\
           u&=g          \quad\hbox{ on }\eth_{\beta}Q.
\end{align}
We now have the

\begin{proof}[Proof of Theorem \ref{thm:stochastic_representation_Uniqueness_BVP_parabolic}]
We choose $T>0$ small enough and $p_0>1$ as in Hypothesis \ref{hyp:stochastic_representation_T_small_parabolic}. The 
pattern of the proof is the same as that of Theorem
\ref{thm:stochastic_representation_Uniqueness_BVP_elliptic}. For completeness, we outline the main steps of the argument.

We need to show that if $u \in C_{\loc}(Q\cup\eth_{\beta} Q)\cap C^2(Q)$ is a solution to problem 
\eqref{eq:stochastic_representation_Parabolic_BVP_beta}, satisfying the growth
bound \eqref{eq:stochastic_representation_Growth_parabolic}, then it admits the stochastic representation
\eqref{eq:stochastic_representation_Stochastic_representation_parabolic_BVP_1}. We choose a collection of increasing 
subdomains, $\left\{\sO_k: k \in \NN\right\}$, as in the
proof of Theorem \ref{thm:stochastic_representation_Uniqueness_BVP_elliptic}. By applying 
It\^o's lemma \cite[Theorems $3.3.3$ \& $3.3.6$]{KaratzasShreve1991}, we obtain, for all
$t>0$ and $k \in \NN$,
\begin{equation}
\label{eq:stochastic_representation_First_stochastic_representation_BVP_parabolic}
\begin{aligned}
&\EE^{t,z}_\QQ \left[e^{-r( \tau_{\sO_k}\wedge T-t)} u( \tau_{\sO_k}\wedge T, Z( \tau_{\sO_k}\wedge T))\right] \\
&\qquad= u(t,z) - \EE^{t,z}_\QQ \left[ \int_{t}^{ \tau_{\sO_k}\wedge T} e^{-r(s-t)} f(s, Z(s)) ds \right].
\end{aligned}
\end{equation}
We now take limit as $k$ tends to $\infty$ in the preceding identity. Using 
\eqref{eq:stochastic_representation_Growth_parabolic} and Lemma
\ref{lem:stochastic_representation_PropertiesHeston}, we obtain
\begin{equation}
\label{eq:stochastic_representation_Convergence_f_BVP_parabolic}
\EE^{t,z}_\QQ \left[ \int_{t}^{ \tau_{\sO_k}\wedge T} e^{-r(s-t)} f(s, Z(s)) ds \right]
\rightarrow
\EE^{t,z}_\QQ \left[ \int_{t}^{\tau_{\sO}\wedge T} e^{-r(s-t)} f(s, Z(s)) ds \right],
\hbox{  as  } k \rightarrow \infty.
\end{equation}
From the continuity of $u$ and of the sample paths of $Z$, we obtain the a.s. convergence as $k$ tends to $\infty$,
\begin{equation*}
  e^{-r( \tau_{\sO_k}\wedge T-t)} u( \tau_{\sO_k}\wedge T, Z( \tau_{\sO_k}\wedge T))
  \rightarrow
  e^{-r(\tau_{\sO}\wedge T)} g(\tau_{\sO}\wedge T, Z(\tau_{\sO}\wedge T)).
\end{equation*}
In order to prove that, as $k$ tends to $\infty$,
\begin{equation}
\label{eq:stochastic_representation_Convergence_u_h_BVP_parabolic}
\begin{aligned}
  \EE^{t,z}_{\QQ}\left[e^{-r( \tau_{\sO_k}\wedge T-t)} u( \tau_{\sO_k}\wedge T, Z( \tau_{\sO_k}\wedge T))\right]
  &\rightarrow
  \EE^{t,z}_{\QQ}\left[e^{-r(\tau_{\sO}\wedge T)} g(\tau_{\sO}\wedge T, Z(\tau_{\sO}\wedge T))\right],
\end{aligned}
\end{equation}
using \cite[Theorem $16.13$]{Billingsley_1986}, it is enough to show that the collection of random variables,
\begin{equation}
\label{eq:stochastic_representation_Collection_rv_BVP_parabolic}
\left\{  e^{-r( \tau_{\sO_k}\wedge T-t)} u( \tau_{\sO_k}\wedge T, Z( \tau_{\sO_k}\wedge T)) : k \in \NN\right\}
\end{equation}
is uniformly integrable. For $p_0>1$ as in Hypothesis \ref{hyp:stochastic_representation_T_small_parabolic}, it is enough 
to show that their $p_0$-th order moments are uniformly
bounded (\cite[Observation following Equation ($16.23$)]{Billingsley_1986}), that is
\begin{equation}
\label{eq:stochastic_representation_Uniform_bound_p_0_moment}
\sup_{k \in \mathbb{N}}
\EE^{t,z}_\QQ \left[\left|e^{-r\tau_{\sO_k}} u(\tau_{\sO_k}, Z(\tau_{\sO_k}))\mathbf{1}_{\{\tau_{\sO_k}<T\}}\right|^{p_0} 
\right] < +\infty.
\end{equation}
From \eqref{eq:stochastic_representation_Growth_parabolic}, we have, for some constant $C$,
\begin{equation*}
\begin{aligned}
&\EE^{t,z}_\QQ \left[\left|e^{-r( \tau_{\sO_k}\wedge T-t)} u( \tau_{\sO_k}\wedge T, Z( \tau_{\sO_k}\wedge T))\right|^{p_0} 
\right]\\
&\qquad\leq C \left( 1+ \EE^{t,z}_\QQ\left[ e^{p_0  M_1 Y(\tau_{\sO_k}\wedge T)}\right]
        + \EE^{t,z}_\QQ \left[e^{p_0  M_2X(\tau_{\sO_k}\wedge T)} \right] \right).
\end{aligned}
\end{equation*}
Now, the uniform bound in \eqref{eq:stochastic_representation_Uniform_bound_p_0_moment} follows by applying the 
supermartingale property
\eqref{eq:stochastic_representation_Supermartingale_property_Y} with $c:=p_0 M_1$ to the first expectation in the 
preceding inequality, and by applying
\eqref{eq:stochastic_representation_Power_X} with $p:=p_0 M_2$ to the second expectation above. Therefore, by taking the 
limit as $k$ tends to $\infty$ in
\eqref{eq:stochastic_representation_First_stochastic_representation_BVP_parabolic}, with the aid of 
\eqref{eq:stochastic_representation_Convergence_f_BVP_parabolic} and
\eqref{eq:stochastic_representation_Convergence_u_h_BVP_parabolic}, we obtain the stochastic representation
\eqref{eq:stochastic_representation_Stochastic_representation_parabolic_BVP_1} of solutions to problem 
\eqref{eq:stochastic_representation_Parabolic_BVP_beta}.
\end{proof}

\begin{proof}[Proof of Theorem \ref{thm:stochastic_representation_Uniqueness_BVP_parabolic_beta_less_than_1}]
The need is to show that if $0<\beta<1$ and $u \in C_{\loc}(Q\cup\eth^1 Q)\cap C^2(Q) \cap 
C^{1,1}_{s,\loc}((0,T)\times(\sO\cup\Gamma_0))$ is a solution to problem
\eqref{eq:stochastic_representation_Parabolic_BVP} with boundary condition 
\eqref{eq:stochastic_representation_Parabolic_BVP_boundary_condition_Gamma_1}, satisfying the growth
bound \eqref{eq:stochastic_representation_Growth_parabolic}, then it admits the stochastic representation
\eqref{eq:stochastic_representation_Stochastic_representation_parabolic_BVP_2}.

Let $\eps>0$ and $\left\{\sU_k: k \in \NN\right\}$ be the collection of increasing subdomains as in
\eqref{eq:stochastic_representation_Sequence_subdomains_with_degenerate_boundary}. By applying It\^o's lemma, we obtain
\begin{equation*}
\EE^{t,z}_\QQ \left[e^{-r(T \wedge \nu_{\sU_k})} u(T \wedge \nu_{\sU_k}, Z^{\eps} (T \wedge \nu_{\sU_k}))\right]
= u(t,z) - \EE^{t,z}_\QQ \left[ \int_{t}^{T \wedge \nu_{\sU_k}} e^{-rs} A^{\eps} u(s,Z^{\eps}(s)) ds \right],
\end{equation*}
where $\nu_{\sU_k}$ is given by \eqref{eq:stochastic_representation_Stopping_time_domain_with_degenerate_boundary}, 
$Z^{\eps}$ by \eqref{eq:stochastic_representation_Z_eps} and
$A^{\eps}$ is defined by \eqref{eq:stochastic_representation_Definition_A_eps}. The proof now follows the same path as 
that of Theorem
\ref{thm:stochastic_representation_Uniqueness_BVP_elliptic_beta_less_than_1}, with the only modification being that we now 
take the limit as $k$ tends to $\infty$ in the
preceding identity in order to obtain \eqref{eq:stochastic_representation_Stochastic_representation_parabolic_BVP_2}.
\end{proof}

Analogous to Lemma \ref{lem:stochastic_representation_WellDefinedStochRepVarEq}, we have the following auxiliary result.

\begin{lem}
\label{lem:stochastic_representation_WellDefinedStochRep_parabolic_BVP}
Suppose $f$ and $g$ obey the growth assumption \eqref{eq:stochastic_representation_Growth_parabolic}. Then there are 
positive constants $\bar{C}$, $M_1 \leq\mu$ and $M_2
\in[0,1]$, such that for any stopping times $\theta_1,\theta_2\in \sT_{t,T}$ with values in $[t,T]$, the function 
$J_p^{\theta_1,\theta_2}$ given by
\eqref{eq:stochastic_representation_Functional_J_parabolic} obeys the growth assumption 
\eqref{eq:stochastic_representation_Growth_parabolic}.
\end{lem}

\begin{proof}
The proof follows as in Lemma \ref{lem:stochastic_representation_WellDefinedStochRepVarEq} with the aid of Lemma 
\ref{lem:stochastic_representation_PropertiesHeston}. Notice that
because the stopping times $\theta_1, \theta_2 \in \sT_{t,T}$ are bounded by $T$, we do not need the constant $r$ to be 
positive, as in Lemma
\ref{lem:stochastic_representation_WellDefinedStochRepVarEq}.
\end{proof}

\begin{rmk}
The function $\psi$ in \eqref{eq:stochastic_representation_Functional_J_parabolic} plays the role of the obstacle function 
and is relevant only for problem
\eqref{eq:stochastic_representation_Parabolic_obstacle_problem}.
\end{rmk}

Next, we have the following \emph{existence} results for solutions to the parabolic terminal/boundary value problem 
\eqref{eq:stochastic_representation_Parabolic_BVP_beta}, for all
$\beta>0$.

\begin{proof}[Proof of Theorem \ref{thm:stochastic_representation_Existence_parabolic_BVP_beta}]
We choose $T>0$ small enough and $p_0>1$ as in Hypothesis \ref{hyp:stochastic_representation_T_small_parabolic}.

By hypothesis, we have $g \in C_{\loc}(\overline{\eth_{\beta} Q})$. Since $\overline{\eth_{\beta} Q}$ is closed, we may 
use \cite[Theorem $3.1.2$]{FriedmanPDE} to extend $g$ to a
function on $[0,T]\times\RR^2$, again called $g$, such that $g \in C_{\loc}([0,T]\times\RR^2)$.

The proof follows the same pattern as that of Theorem \ref{thm:stochastic_representation_Existence_elliptic_BVP_beta}. For 
completeness, we outline the main steps. Let $\sO_k$ be
an increasing sequence of $C^{2+\alpha}$ subdomains of $\sO$ as in the proof of Theorem 
\ref{thm:stochastic_representation_Uniqueness_BVP_elliptic}, and let $Q_k:=(0, T) \times
\sO_k$. We notice that on each cylindrical domain, $Q_k$, the operator $A$ is uniformly elliptic, and its coefficients are 
$C^{\infty}(\bar{Q}_k)$ functions. By hypothesis, there
is an $\alpha \in (0, 1)$ such that $f \in C^{\alpha}(\bar{Q}_k)$ and $g\in C(\bar Q_k)$. Therefore, by \cite[Theorem 
$3.4.9$]{FriedmanPDE}, the terminal value problem
\begin{align*}
-u_t+Au &= f \quad\hbox{ on }Q_k,
\\
u &= g \quad\hbox{ on }(0,T)\times\partial\sO_k \cup \{T\}\times\bar\sO_k,
\end{align*}
has a unique solution $u_k \in C(\bar{Q}_k) \cap C^{2+\alpha}(Q_k)$, and by\footnote{See also \cite[Theorem 
$5.7.6$]{KaratzasShreve1991}.} \cite[Theorem $6.5.2$]{FriedmanSDE}
it has the stochastic representation
\begin{equation}
\label{eq:stochastic_representation_Approximate_stochastic_representation_parabolic_BVP}
\begin{aligned}
u_k(t, z) &= \EE^{t,z}_\QQ \left[e^{-r(\tau_{\sO_k}\wedge T-t)} g(\tau_{\sO_k}\wedge T, Z(\tau_{\sO_k}\wedge T))\right]           
\\
          &\quad+ \EE^{t,z}_\QQ\left[\int_{t}^{\tau_{\sO_k}\wedge T} e^{-r(s-t)} f(s,Z(s))\,ds \right],
          \quad\forall (t,z)\in\bar Q_k.
\end{aligned}
\end{equation}
Because $\tau_{\sO_k}$ converges to $\tau_{\sO}$ a.s. as $k\rightarrow\infty$, the integral term in
\eqref{eq:stochastic_representation_Approximate_stochastic_representation_parabolic_BVP} converges to the integral term of 
$u^*$ in
\eqref{eq:stochastic_representation_Stochastic_representation_parabolic_BVP_1}, by the same argument as that used in the 
proof of Theorem
\ref{thm:stochastic_representation_Uniqueness_BVP_parabolic}. By the continuity of $g$ and of the paths of the Heston 
process $Z$, we also know that
\[
e^{-r(\tau_{\sO_k}\wedge T)} g(\tau_{\sO_k}\wedge T, Z(\tau_{\sO_k}\wedge T))
\rightarrow
e^{-r(\tau_{\sO}\wedge T)} g(\tau_{\sO}\wedge T, Z(\tau_{\sO}\wedge T)), \quad\hbox{ as } k \rightarrow \infty.
\]
In order to show that the preceding convergence takes place in expectation also, it is enough to show that the collection 
of random variables,
\[
\left\{e^{-r(\tau_{\sO_k}\wedge T)} g(\tau_{\sO_k}\wedge T, Z(\tau_{\sO_k}\wedge T)): k \in \NN\right\},
\]
is uniformly integrable, but this follows by the same argument as that used for the collections 
\eqref{eq:stochastic_representation_Collection_rv_BVP_parabolic} in the proof of
Theorem \ref{thm:stochastic_representation_Uniqueness_BVP_parabolic}, by bounding their $p_0$-th order moments ($p_0>1$). 
Therefore, the sequence $\{u_k: k \in \NN\}$ converges
to $u^*$ pointwise on $Q$. By interior Schauder estimates for parabolic equations \cite[Theorem $3.3.5$]{FriedmanSDE} and 
Lemma 
\ref{lem:stochastic_representation_WellDefinedStochRep_parabolic_BVP}, there is a subsequence of $\{u_k: k \in \NN\}$ 
which converges to $u^*$ in $C^{2+\alpha'}(Q)$, when
$0<\alpha'<\alpha$.
Using the Arzel$\grave{\text{a}}$-Ascoli Theorem, we obtain $u ^* \in C^{2+\alpha}(Q)$. The proof of continuity of $u$ up 
to $\eth_{\beta} Q$ follows by exactly the same argument
as that used in the proof of Step \ref{step:stochastic_representation_Elliptic_BVP_continuity} in Theorem 
\ref{thm:stochastic_representation_Existence_elliptic_BVP_beta}.
Therefore, $u^*$ is a solution to \eqref{eq:stochastic_representation_Parabolic_BVP_beta}.

From Theorem \ref{thm:stochastic_representation_Uniqueness_BVP_parabolic} and Lemma 
\ref{lem:stochastic_representation_WellDefinedStochRep_parabolic_BVP}, we see that $u^*$ in
\eqref{eq:stochastic_representation_Stochastic_representation_parabolic_BVP_1} is the unique solution to the parabolic 
terminal value problem
\eqref{eq:stochastic_representation_Parabolic_BVP_beta}, for all $\beta>0$.
\end{proof}

We now have the

\begin{proof}[Proof of Theorem \ref{thm:stochastic_representation_Existence_parabolic_BVP_beta_with_Holder_continuous_boundary_data}]
Just as in the proof of Theorem \ref{thm:stochastic_representation_Existence_parabolic_BVP_beta}, we can easily adapt the 
proof of Theorem
\ref{thm:stochastic_representation_Existence_elliptic_BVP_beta_Holder_continuous_boundary_data} for the elliptic case to 
the present parabolic case. For this purpose, we only
need to make use of the local boundary Schauder estimate in Proposition 
\ref{prop:stochastic_representation_Local_apriori_boundary_estimates} instead of \cite[Corollary
$6.7$]{GilbargTrudinger} for the elliptic case.
\end{proof}

\section{Parabolic obstacle problem}
\label{sec:stochastic_representation_Parabolic_Obstacle}
Problem \eqref{eq:stochastic_representation_Parabolic_obstacle_problem} with boundary condition
\eqref{eq:stochastic_representation_Parabolic_BVP_boundary_condition_whole_boundary}, when $0<\beta<1$, and with boundary 
condition
\eqref{eq:stochastic_representation_Parabolic_BVP_boundary_condition_Gamma_1}, when $\beta \geq 1$, can be formulated as
\begin{equation}
\label{eq:stochastic_representation_Parabolic_obstacle_problem_beta}
\begin{cases}
\min\left\{-u_t+Au-f, u-\psi\right\} = 0& \quad\hbox{ on }Q, \\
u = g& \quad\hbox{ on }\eth_{\beta} Q,
\end{cases}
\end{equation}
where $\eth_{\beta} Q$ is defined in \eqref{eq:stochastic_representation_Boundary_beta_parabolic}. According to Theorem
\ref{thm:stochastic_representation_Uniqueness_parabolic_OP}, the solution to problem 
\eqref{eq:stochastic_representation_Parabolic_obstacle_problem_beta} is given in
\eqref{eq:stochastic_representation_Stochastic_representation_POP_2}.

\begin{proof}[Proof of Theorem \ref{thm:stochastic_representation_Uniqueness_parabolic_OP}]
We choose $\widetilde T>0$ small enough so that it obeys Hypothesis \ref{hyp:stochastic_representation_T_small_parabolic}. 
For such $\widetilde T>0$, the proof of Theorem
\ref{thm:stochastic_representation_Uniqueness_elliptic_OP} adapts to the present case in the same way that the proof of 
Theorem 
\ref{thm:stochastic_representation_Uniqueness_BVP_elliptic} adapts to give a proof of Theorem 
\ref{thm:stochastic_representation_Uniqueness_BVP_parabolic}. Therefore, it remains
to show that the corresponding stochastic representation 
\eqref{eq:stochastic_representation_Stochastic_representation_POP_1} of the solution  to problem
\eqref{eq:stochastic_representation_Parabolic_obstacle_problem_beta} holds for $T$ arbitrarily large.

Let $N:= \lfloor T/\widetilde T \rfloor$ (the greatest integer in $T/\widetilde T$), and $T_i:=i\widetilde T$, for 
$i=0,\ldots, N-1$, and $T_N:=T$. Let $k$ be an integer such
that $1 \leq k \leq N-1$, and assume that the stochastic representation formula 
\eqref{eq:stochastic_representation_Stochastic_representation_POP_1} holds for any $t \in [T_i,
T]$, where $i=k,\ldots,N-1$. We want to show that it holds also for $t \in [T_{k-1},T]$. Notice that for $k=N-1$, we have 
$T-t \leq \widetilde T$, for all $t \in [T_{N-1}, T]$,
and so we know that the stochastic representation \eqref{eq:stochastic_representation_Stochastic_representation_POP_1} of 
the solution  to problem
\eqref{eq:stochastic_representation_Parabolic_obstacle_problem} holds, by the observation at the beginning of the present 
proof.

For any $t\leq v\leq T$,  stopping time $\theta \in \sT_{t,v}$ with values in $[t,v]$, and $\varphi \in C(\bar \sO)$, we 
denote
\begin{equation}
\label{eq:stochastic_representation_Functional_F}
\begin{aligned}
F^{\varphi}(t,z, v, \theta)
& := \int_{t}^{\tau_{\sO}\wedge\theta} e^{-r(s-t)} f(s,Z(s))\,ds
  + e^{-r(\theta-t)} \psi(\theta,Z(\theta)) 1_{\{ \theta < \tau_{\sO} \wedge v \}}  \\
&\quad
+ e^{-r(\tau_{\sO}-t)} g(\tau_{\sO}, Z(\tau_{\sO})) 1_{\{ \tau_{\sO} \leq \theta, \tau_\sO < v \}}
+ e^{-r(v-t)} \varphi(Z(v)) 1_{\{ \tau_{\sO}\wedge v\leq\theta, \tau_{\sO} \geq v \}}.
\end{aligned}
\end{equation}
Notice that by choosing $\varphi=g(T,\cdot)$ and $v=T$ in \eqref{eq:stochastic_representation_Functional_F}, we obtain, 
for any stopping time $\theta\in \sT_{t,T}$,
\begin{align*}
&e^{-r(\tau_{\sO}-t)} g(\tau_{\sO}, Z(\tau_{\sO})) 1_{\{ \tau_{\sO} \leq \theta, \tau_\sO < T \}}
+ e^{-r(T-t)} \varphi(Z(T)) 1_{\{ \tau_{\sO}\wedge T\leq\theta, \tau_{\sO} \geq T \}}\\
&=
e^{-r(\tau_{\sO}\wedge T-t)} g(\tau_{\sO}\wedge T, Z(\tau_{\sO}\wedge T)) 1_{\{ \tau_{\sO}\wedge T \leq \theta \}}
\end{align*}
and so,
\begin{equation}
\label{eq:stochastic_representation_Functional_F_with_g}
\begin{aligned}
F^{g(T,\cdot)}(t,z, T, \theta)
& = \int_{t}^{\tau_{\sO}\wedge\theta} e^{-r(s-t)} f(s,Z(s))\,ds
  + e^{-r(\theta-t)} \psi(\theta,Z(\theta)) 1_{\{ \theta < \tau_{\sO} \wedge T \}}  \\
&\qquad
+ e^{-r(\tau_{\sO}\wedge T-t)} g(\tau_{\sO}\wedge T, Z(\tau_{\sO}\wedge T)) 1_{\{ \tau_{\sO}\wedge T \leq \theta \}}.
\end{aligned}
\end{equation}
Because $u$ solves problem \eqref{eq:stochastic_representation_Parabolic_obstacle_problem_beta} on the interval $(T_{k-1}, 
T_k)$, and $T_k-T_{k-1}\leq \widetilde T$, we see that
$u$ has the stochastic representation \eqref{eq:stochastic_representation_Stochastic_representation_POP_1}, for any $t \in 
[T_{k-1}, T_k)$ and $z \in
\sO\cup\partial_{\beta}\sO$,
\begin{equation}
\label{eq:stochastic_representation_First_FormulaForU}
u(t,z) = \sup_{\theta \in \sT_{t,T_k}} \EE^{t,z}_\QQ  \left[F^{u^*(T_k,\cdot)}(t, z, T_k, \theta)\right].
\end{equation}
For any stopping time $\eta\in\sT_{t,T_k}$, we set
\begin{equation}
\begin{aligned}
\label{eq:stochastic_representation_Functional_F_1}
F_1(t,z, T_k,\eta) &:=  \int_{t}^{\tau_{\sO}\wedge\eta} e^{-r(s-t)}f(s,Z(s))ds \\
&\quad         + e^{-r(\eta-t)} \psi(\eta,Z(\eta))\mathbf{1}_{\{\eta<\tau_{\sO}\wedge T_k\}}  \\
&\quad  + e^{-r(\tau_{\sO}-t)} g(\tau_{\sO},Z(\tau_{\sO}))\mathbf{1}_{\{\tau_{\sO} \leq \eta, \eta < T_k \}},
\end{aligned}
\end{equation}
and for any stopping time $\xi\in\sT_{T_k,T}$, we let
\begin{equation}
\begin{aligned}
\label{eq:stochastic_representation_Functional_F_2}
F_2(t,z, T,\xi) &:=  \int_{T_k}^{\tau_{\sO}\wedge\xi} e^{-r(s-T_k)}f(s,Z(s))ds  \\
&\quad   + e^{-r(\xi-T_k)} \psi(\xi,Z(\xi))\mathbf{1}_{\{\xi<\tau_{\sO}\wedge T\}} \\
&\quad   + e^{-r(\tau_{\sO}\wedge T-T_k)} g(\tau_{\sO}\wedge T,Z(\tau_{\sO}\wedge T))\mathbf{1}_{\{\tau_{\sO}\wedge T \leq 
\xi\}}.
\end{aligned}
\end{equation}
For the rest of the proof, we \emph{fix} $z \in \sO\cup\partial_{\beta}\sO$ and $t \in [T_{k-1}, T_k)$.

Let $\eta \in \sT_{t,T_k}$ and $\xi \in \sT_{T_k, T}$. It is straightforward to see that
\begin{align*}
\theta:=
\begin{cases}
\eta & \hbox{ if } \eta <T_k,\\
\xi& \hbox{ if } \eta = T_k,
\end{cases}
\end{align*}
is a stopping time with values in  $[t,T]$. We denote by
\begin{equation}
\label{eq:stochastic_representation_New_set_stopping_times}
\begin{aligned}
\sS_{t,T} &= \left\{\theta \in \sT_{t,T}:
\theta=\eta\mathbf{1}_{\{\eta<T_k\}}+ \xi\mathbf{1}_{\{\eta = T_k\}},
\hbox{ where } \eta\in\sT_{t,T_k} \hbox{ and } \xi \in \sT_{T_k,T} \right\}.
\end{aligned}
\end{equation}
For any stopping time $\theta \in \sT_{t,T}$, we define the stopping times $\theta'\in\sT_{t,T_k}$ and 
$\theta''\in\sT_{T_k,T}$,
\begin{equation}
\label{eq:stochastic_representation_definition_theta_prim_secund}
\theta' := \mathbf{1}_{\{\theta<T_k\}} \theta + \mathbf{1}_{\{\theta \geq T_k\}} T_k
\quad\hbox{and}\quad \theta'' := \mathbf{1}_{\{\theta < T_k\}} T_k + \mathbf{1}_{\{\theta \geq T_k\}} \theta .
\end{equation}
Then, any stopping time $\theta\in\sT_{t,T}$ can be written as
\begin{align*}
\theta
&= \theta' \mathbf{1}_{\{\theta<T_k\}} + \theta'' \mathbf{1}_{\{\theta\geq T_k\}}\\
&= \theta' \mathbf{1}_{\{\theta'<T_k\}} + \theta'' \mathbf{1}_{\{\theta' = T_k\}}
\end{align*}
and so,
\[
\sT_{t,T}=\sS_{t,T}.
\]
The preceding identity and definitions \eqref{eq:stochastic_representation_Stochastic_representation_POP_1} of $u^*$ and 
\eqref{eq:stochastic_representation_Functional_F} of
$F^{\varphi}$ give us
\begin{equation}
\label{eq:stochastic_representation_SimpleEq2}
\begin{aligned}
u^*(t,z) = \sup_{\theta\in \sS_{t,T}} \EE^{t,z}_\QQ\left[F^{g(T, \cdot)}(t,z, T,\theta)\right].
\end{aligned}
\end{equation}
We shall need the following identities
\begin{claim}
\label{claim:stochastic_representation_split_terms_F}
For any stopping time $\theta=\eta\mathbf{1}_{\{\eta<T_k\}}+ \xi\mathbf{1}_{\{\eta = T_k\}}$, where $\eta\in\sT_{t,T_k}$ 
and  $\xi \in \sT_{T_k,T}$, we have the following
identities
\begin{align*}
\int_{t}^{\tau_{\sO}\wedge\theta} e^{-r(s-t)}f(s,Z(s))ds
&=\mathbf{1}_{\{\eta < T_k\}} \int_{t}^{\tau_{\sO}\wedge\eta} e^{-r(s-t)}f(s,Z(s))ds\\
&\quad + \mathbf{1}_{\{\eta = T_k\}} \int_{T_k}^{\tau_{\sO}\wedge\xi} e^{-r(s-t)}f(s,Z(s))ds,
\end{align*}
and
\begin{align*}
e^{-r(\theta-t)} \psi(\theta,Z(\theta))\mathbf{1}_{\{\theta<\tau_{\sO}\wedge T\}}
& =e^{-r(\eta-t)} \psi(\eta,Z(\eta))\mathbf{1}_{\{\eta<\tau_{\sO}\wedge T_k\}}\mathbf{1}_{\{\eta<T_k\}}\\
&\quad +e^{-r(\xi-t)} \psi(\xi,Z(\xi))\mathbf{1}_{\{\xi<\tau_{\sO}\wedge T\}}
\mathbf{1}_{\{\eta= T_k\}},
\end{align*}
and
\begin{align*}
&e^{-r(\tau_{\sO}\wedge T-t)} g(\tau_{\sO}\wedge T,Z(\tau_{\sO}\wedge T))\mathbf{1}_{\{\tau_{\sO}\wedge T \leq \theta 
\}}\\
&\qquad\qquad\quad =e^{-r(\tau_{\sO}-t)} g(\tau_{\sO},Z(\tau_{\sO}))\mathbf{1}_{\{\tau_{\sO} \leq \eta, \eta < T_k 
\}}\mathbf{1}_{\{\eta<T_k\}}\\
&\qquad\qquad\quad\quad+ e^{-r(\tau_{\sO}\wedge T-t)} g(\tau_{\sO}\wedge T,Z(\tau_{\sO}\wedge T))
\mathbf{1}_{\{\tau_{\sO}\wedge T \leq \xi\}} \mathbf{1}_{\{\eta = T_k\}}.
\end{align*}
\end{claim}
\begin{proof}
Notice that
\begin{equation}
\label{eq:stochastic_representation_sets_stopping_times}
\{\theta<T_k\} = \{\eta<T_k\} \hbox{  and  } \{\theta \geq T_k\} = \{\eta=T_k\}.
\end{equation}
The first identity is obvious because, by \eqref{eq:stochastic_representation_sets_stopping_times}, we see that
\begin{equation}
\label{eq:stochastic_representation_theta_prim_secund}
\theta=\eta \hbox{ on } \{\eta<T_k\} \hbox{ and } \theta=\xi \hbox{  on  }\{\eta=T_k\}.
\end{equation}
The second identity follows by the observation that
\begin{align*}
\{\theta<\tau_{\sO}\wedge T\} = \{\theta<\tau_{\sO}\wedge T, \theta<T_k\} \cup \{\theta<\tau_{\sO}\wedge T, \theta \geq 
T_k\},
\end{align*}
and using \eqref{eq:stochastic_representation_theta_prim_secund} and 
\eqref{eq:stochastic_representation_sets_stopping_times}, it follows
\begin{align*}
\{\theta<\tau_{\sO}\wedge T\} = \{\eta<\tau_{\sO}\wedge T_k, \eta < T_k \} \cup \{\xi<\tau_{\sO}\wedge T, \eta = T_k\}.
\end{align*}
For the last identity of the claim, we notice
\begin{align*}
\{\tau_\sO \wedge T \leq \theta\}
&= \{\tau_\sO \wedge T \leq \theta, \tau_\sO < T\} \cup \{\tau_\sO \wedge T \leq \theta, \tau_{\sO} \geq T\}\\
&= \{\tau_\sO \wedge T \leq \theta, \tau_\sO < T, \theta < T_k\}
\cup \{\tau_\sO \wedge T \leq \theta, \tau_\sO < T, \theta \geq T_k\}\\
&\quad\cup \{\tau_\sO \wedge T \leq \theta, \tau_{\sO} \geq T\}.
\end{align*}
By \eqref{eq:stochastic_representation_theta_prim_secund} and \eqref{eq:stochastic_representation_sets_stopping_times}, we 
obtain
\begin{align*}
\{\tau_\sO \wedge T \leq \theta\}
&= \{\tau_\sO  \leq \eta,  \tau_\sO < T, \eta < T_k\}
\cup \{\tau_\sO \wedge T \leq \xi, \tau_\sO < T, \eta = T_k\}\\
&\quad \cup \{\tau_\sO \wedge T \leq \xi, \tau_{\sO} \geq T\}\\
&= \{\tau_\sO  \leq \eta,  \eta < T_k\}
\cup \{\tau_\sO \wedge T \leq \xi, \eta = T_k\},
\end{align*}
which implies the last identity of the claim.
\end{proof}
We can write the expression for $F^{g(T, \cdot)}(t,z, T,\theta)$ as a sum,
\begin{equation}
\label{eq:stochastic_representation_decomposition_F}
F^{g(T, \cdot)}(t,z, T,\theta) = \mathbf{1}_{\{\eta<T_k\}} F_1(t,z, T_k,\eta) + \mathbf{1}_{\{\eta = T_k\}} 
e^{-r(T_k-t)}F_2(t,z, T,\xi).
\end{equation}
Because $\xi\in \sT_{T_k,T}$ and $F_2(t,z,T,\xi)$ depends only on $\left(Z^{t,z}(s)\right)_{T_k\leq s\leq T}$, and the 
Heston process has the (strong) Markov property
\cite[Theorem $1.15$ (c)]{Feehan_Pop_mimickingdegen}, we have a.s. that
\begin{align*}
\EE^{t,z}_{\QQ}\left[F_2(t,z, T,\xi)|\sF_{T_k}\right]
&= \EE^{T_k,Z^{t,z}(T_k)}_{\QQ}\left[ F_2(T_k,Z^{t,z}(T_k), T,\xi)\right]\\
&= \EE^{T_k,Z^{t,z}(T_k)}_{\QQ}\left[ F^{g(T,\cdot)}(T_k,Z^{t,z}(T_k), T,\xi)\right],
\end{align*}
by applying definitions \eqref{eq:stochastic_representation_Functional_F_with_g} and 
\eqref{eq:stochastic_representation_Functional_F_2}. Thus,
\begin{align*}
&\EE^{t,z}_{\QQ}\left[\mathbf{1}_{\{\eta = T_k\}} e ^{-r(T_k-t)}F_2(t,z, T,\xi)|\sF_{T_k}\right]\\
&\qquad= \EE^{t,z}_{\QQ}\left[\EE^{t,z}_{\QQ}\left[\mathbf{1}_{\{\eta = T_k\}} e ^{-r(T_k-t)}F_2(t,z, 
T,\xi)|\sF_{T_k}\right]\right]\\
&\qquad= \EE^{t,z}_{\QQ}\left[\mathbf{1}_{\{\eta = T_k\}} e ^{-r(T_k-t)}\EE^{t,z}_{\QQ}\left[F_2(t,z, 
T,\xi)|\sF_{T_k}\right]\right]\\
&\qquad= \EE^{t,z}_{\QQ}\left[\mathbf{1}_{\{\eta = T_k\}} e ^{-r(T_k-t)} \EE^{T_k,Z(T_k)}_{\QQ}\left[ 
F^{g(T,\cdot)}(T_k,Z(T_k), T,\xi)\right]\right].
\end{align*}
By the preceding identity, \eqref{eq:stochastic_representation_New_set_stopping_times} and 
\eqref{eq:stochastic_representation_decomposition_F}, the identity
\eqref{eq:stochastic_representation_SimpleEq2} yields
\begin{align*}
u^*(t,z)
& = \sup_{\substack{ \theta=\eta\mathbf{1}_{\{\eta<T_k\}}+\xi\mathbf{1}_{\{\eta=T_k\}}\\\theta \in 
\sS_{t,T},\eta\in\sT_{t,T_k}, \xi\in\sT_{T_k,T}}}
  \left\{  \EE^{t,z}_\QQ  \left[ \mathbf{1}_{\{\eta< T_k\}}F_1(t,z, T_k,\eta) \right.\right.\\
&\qquad\qquad\left.\left.
   + \mathbf{1}_{\{\eta = T_k\}} e^{-r(T_k-t)}  \EE^{T_k,Z(T_k)}_{\QQ}\left[ F^{g(T, 
   \cdot)}(T_k,Z(T_k),T,\xi)\right]\right] \right\}\\
& = \sup_{\eta \in \sT_{t,T_k}}
  \left\{  \EE^{t,z}_\QQ  \left[ \mathbf{1}_{\{\eta< T_k\}}F_1(t,z, T_k,\eta) \right.\right.\\
&\qquad\qquad\left.\left.
   + \mathbf{1}_{\{\eta = T_k\}} e^{-r(T_k-t)} \sup_{\xi \in \sT_{T_k,T}} \EE^{T_k,Z(T_k)}_{\QQ}\left[ F^{g(T, 
   \cdot)}(T_k,Z(T_k),T,\xi)\right]\right] \right\}.
\end{align*}
Using the definition \eqref{eq:stochastic_representation_Stochastic_representation_POP_1} of $u^*$, we have
\[
u^*(T_k,Z(T_k)) = \sup_{\xi \in \sT_{T_k,T}} \EE^{T_k,Z(T_k)}_\QQ\left[F^{g(T, \cdot)}(T_k,Z(T_k),T,\xi)\right],
\]
and so it follows that
\begin{align*}
u^*(t,z)
&=\sup_{\eta \in \sT_{t,T_k}}
   \EE^{t,z}_\QQ  \left[ \mathbf{1}_{\{\eta< T_k\}}F_1(t,z, T_k,\eta)  + \mathbf{1}_{\{\eta = T_k\}} e^{-r(T_k-t)}  
   u^*(T_k,Z(T_k))\right].
\end{align*}
Notice that, by the definitions \eqref{eq:stochastic_representation_Functional_F} of $F^{\varphi}$ and 
\eqref{eq:stochastic_representation_Functional_F_1} of $F_1$, we have
\[
F^{u^*(T,\cdot)}(t,z,T_k,\eta)=\mathbf{1}_{\{\eta< T_k\}}F_1(t,z, T_k,\eta)  + \mathbf{1}_{\{\eta = T_k\}} e^{-r(T_k-t)}  
u^*(T_k,Z(T_k)).
\]
The preceding two identities yield
\begin{align*}
u^*(t,z) &= \sup_{\eta \in \sT_{t,T_k}}  \EE^{t,z}_\QQ  \left[F^{u^*(T,\cdot)}(t,z,T_k,\eta)\right]\\
         &= u(t,z), \quad\hbox{(by \eqref{eq:stochastic_representation_First_FormulaForU}).}
\end{align*}
This concludes the proof of the theorem.
\end{proof}

\begin{proof}[Proof of Theorem \ref{thm:stochastic_representation_Uniqueness_parabolic_OP_beta_less_than_1}]
We omit the proof as it is very similar to the proofs of Theorems 
\ref{thm:stochastic_representation_Uniqueness_parabolic_OP} and
\ref{thm:stochastic_representation_Uniqueness_elliptic_OP_beta_less_than_1}.
\end{proof}

\appendix

\section{Local a priori boundary estimates}
\label{app:LocalBoundaryEstimates}
To complete the proof of Theorem 
\ref{thm:stochastic_representation_Existence_parabolic_BVP_beta_with_Holder_continuous_boundary_data} we need the 
following local a priori boundary estimate (a parabolic analogue of \cite[Corollary 6.7]{GilbargTrudinger})
for a solution to a parabolic terminal/boundary value problem and for which we were not able to find a suitable 
reference in the literature. 

\begin{prop}[Local a priori boundary estimates]
\label{prop:stochastic_representation_Local_apriori_boundary_estimates}
Let $\sO\subseteqq\HH$ be a domain such that the boundary portion $\Gamma_1$ is of class $C^{2+\alpha}$. For $z_0 \in 
\Gamma_1$ and $R>0$, let
\[
B_R(z_0):=\left\{z\in \RR^d: |z-z_0|<R \right\} \quad\hbox{and}\quad 
Q_{R,T}(z_0) := (0,T)\times \sO\cap B_R(z_0).
\]
Assume $B_{2R}(z_0) \Subset \HH$ 
and let $f\in C^{\alpha}(\bar Q_{2R,T}(z_0))$ and 
$g \in C^{2+\alpha}(\bar Q_{2R,T}(z_0))$. Then, there is a positive constant $C$,
depending only on $z_0$, $R$ and the coefficients of $A$, such that for any solution 
$u \in C^{2+\alpha}(\bar Q_{2R,T}(z_0))$ to
\begin{align*}
-u_t + A u &= f \quad \hbox{on  } Q_{2R,T}(z_0),
\\
u &= g \quad \hbox{on  } \left((0,T]\times \left(B_{2R}(z_0) \cap \Gamma_1\right)\right) \cup \left(\{T\}\times \sO\cap B_{2R}(z_0)\right),
\end{align*}
we have
\begin{align*}
\|u\|_{C^{2+\alpha}(\bar Q_{R,T}(z_0))}
&\leq C\left(
\|f\|_{C^{\alpha}(\bar Q_{2R,T}(z_0))}
+\|g\|_{C^{2+\alpha}(\bar Q_{2R,T}(z_0))}
+\|u\|_{C(\bar Q_{2R,T}(z_0))} \right).
\end{align*}
\end{prop}

\begin{proof}
The result follows by combining the global Schauder estimate \cite[Theorem $10.4.1$]{Krylov_LecturesHolder} and the 
localization procedure of \cite[Theorem
$8.11.1$]{Krylov_LecturesHolder}, exactly as in the proof of \cite[Theorem $3.8$]{Feehan_Pop_mimickingdegen}.
\end{proof}

\begin{rmk}
The interior version of Proposition \ref{prop:stochastic_representation_Local_apriori_boundary_estimates} can be found in 
\cite[Exercise $10.4.2$]{Krylov_LecturesHolder}.
\end{rmk}

\section{Regular points and continuity properties of stochastic representations}
\label{sec:RegularPoints}
For the purpose of this section, we let $d$ be a non-negative integer, $D\subset\RR^d$ a bounded domain and $t_1<t_2$. We 
denote by $Q:=(t_1,t_2)\times D$ and recall that $\eth Q
: = (t_1,t_2)\times\partial D \cup \{t_2\}\times \bar D$. We consider coefficients $a$, $b$ and $\sigma$ satisfying the 
following conditions.

\begin{hyp}
\label{hyp:stochastic_representation_Nondegenerate_coefficients}
Let
\begin{align*}
a :\bar Q \rightarrow \RR^{d\times d}\quad \hbox{  and  }\quad b:\bar Q \rightarrow \RR^d,
\end{align*}
be maps with component functions, $a^{ij}$, $b^i$, belonging to $C^{0,1}(\bar Q)$. Require that the matrix, $a$, be 
symmetric and obey
\begin{align}
\label{eq:stochastic_representation_Uniform_ellipticity}
\sum_{i,j=1}^d a^{ij}(t,z) \xi^i\xi^j \geq \delta |\xi|^2, \quad \forall \xi \in \RR^d, \quad \forall (t,z)\in\bar Q,
\end{align}
where $\delta$ is a positive constant.
\qed
\end{hyp}
Let $\sigma$ be a square root of the matrix $a$ such that $\sigma\in C^{0,1}(\bar Q; \RR^{d\times d})$. Such a choice 
exists by \cite[Lemma $6.1.1$]{FriedmanSDE}. We consider an
extension of the coefficients $b$ and $\sigma$ from $\bar Q$ to $\RR\times\RR^d$, such that these extensions are bounded 
and uniformly Lipschitz continuous, and condition
\eqref{eq:stochastic_representation_Uniform_ellipticity} is satisfied on $\RR\times\RR^d$. Then, by \cite[Theorems $5.2.5$ 
\& $5.2.9$]{KaratzasShreve1991}, for any $(t,z)\in
\RR\times\RR^d$, there is a unique strong solution to
\begin{equation}
\label{eq:stochastic_representation_Non_degenerate_SDE}
\begin{aligned}
d Z_i(s) &= b^i(s,Z(s)) dt + \sum_{j=1}^d \sigma^{ij}(s,Z(s)) d W_j(s), \quad \forall i=1, \ldots, d,\quad s>t,\\
Z(t) &=z,
\end{aligned}
\end{equation}
where $W$ is a $\RR^d$-valued Brownian motion.

We have the following consequence of\footnote{See also \cite[Theorem $4.2.12$]{KaratzasShreve1991}.} \cite[Theorem 
$2.4.2$]{Dynkin_diffsuperpde} and \cite[Theorem $2.4.1$ and the
Remark following Theorem $2.4.1$]{Dynkin_diffsuperpde}.

\begin{cor}[Continuity of stochastic representations with killing term]
\label{cor:stochastic_representation_Continuity_stochastic_representations_with_killing_term}
Assume Hypothesis \ref{hyp:stochastic_representation_Nondegenerate_coefficients} holds and
\begin{enumerate}
\item the function $g$ is a Borel measurable, bounded function on $\eth Q$ which is continuous at $(t,z)$,
\item the function $c:\bar Q\rightarrow [0,\infty)$ is non-negative, bounded and Borel measurable,
\item if there is $T>0$, such that $\tau_{Q} \leq T$ a.s., then the function $c:\bar Q\rightarrow \RR$ is bounded and 
    Borel measurable function.
\end{enumerate}
Then
\begin{equation}
\label{eq:stochastic_representation_Continuity_stochastic_representations_with_killing_term}
\lim_{Q \ni (t',z')\rightarrow (t,z)} \EE^{t',z'}_{\QQ}\left[\exp \left(-\int_{t'}^{\tau_Q} c(s,Z(s))ds\right)g(\tau_Q, 
Z(\tau_Q))\right] = g(t,z),
\end{equation}
for all regular points $(t,z)\in\eth Q$.
\end{cor}

\begin{proof}
We consider first the case when the stopping time $\tau_{Q}$ is not necessarily bounded by a positive constant $T$. Then, 
we let $c_0$ be a positive constant such that
\begin{equation}
\label{eq:stochastic_representation_inequality_c}
0 \leq c \leq c_0, \quad \hbox{a.e. on Q}.
\end{equation}
Let $(t,z)\in\eth Q$ be a fixed regular point. We fix $\eps>0$ and consider $t'\in [t_1,t_2]$ such that $|t-t'|<\eps/2$. 
Then, using the fact that $\tau^{t',z'}_Q \geq
t'>t-\eps/2$, we see that
\[
\left\{\tau^{t',z'}_Q <t-\eps\right\} \subseteq \left\{t-\eps/2< t' \leq \tau^{t',z'}_Q <t-\eps\right\} = \emptyset,
\]
and so, we obtain
\begin{align*}
\left\{|\tau^{t',z'}_Q -t|>\eps\right\} &\subseteq  \left\{\tau^{t',z'}_Q >t+\eps\right\} \cup \left\{\tau^{t',z'}_Q 
t<t-\eps\right\}\\
&\subseteq  \left\{\tau^{t',z'}_Q >t+\eps\right\}.
\end{align*}
Applying \cite[Theorem $2.4.1$ and the Remark following Theorem $2.4.1$]{Dynkin_diffsuperpde}, with $t_0:=t+\eps$, it 
follows
\[
\lim_{Q \ni (t',z')\rightarrow (t,z)} \QQ^{t',z'}\left(|\tau_Q -t|>\eps\right)
\leq \lim_{Q \ni (t',z')\rightarrow (t,z)} \QQ^{t',z'}\left(\tau_Q >t+\eps\right)
 = 0,
\]
from where it follows that $\tau^{t',z'}_Q$ converges in probability to $0$. Similarly, we can argue that
\begin{equation}
\label{eq:stochastic_representation_exponential_rv}
\exp \left(-\int_{t'}^{\tau^{t',z'}_Q} c(s,Z^{(t',z')}(s))ds\right)
\end{equation}
converges in probability to $1$, as $(t',z')\in Q$ tends to $(t,z)$. We again fix $\eps\in(0,1)$ and consider $t'$ such 
that $|t'-t|<-1/(2c_0)\log(1-\eps)$. By inequality
\eqref{eq:stochastic_representation_inequality_c}, we see that
\begin{align*}
&\QQ^{t',z'}\left(\left|\exp \left(-\int_{t'}^{\tau_Q} c(s,Z(s))ds\right) -1\right|>\eps\right)\\
&\quad=
 \QQ^{t',z'}\left( \exp\left(-\int_{t'}^{\tau_Q} c(s,Z(s))ds\right) < 1-\eps \right), \quad \hbox{(as $c \geq 0$),}\\
&\quad\leq
 \QQ^{t',z'}\left( \exp\left( -c_0(\tau_Q-t')\right) <1-\eps\right), \quad \hbox{(as $0\leq c\leq c_0$),}\\
&\quad=  \QQ^{t',z'}\left(\tau_Q>t'-\frac{1}{c_0}\log(1-\eps)\right)\\
&\quad=  \QQ^{t',z'}\left(\tau_Q>t-\frac{1}{2c_0}\log(1-\eps)\right)
\quad \hbox{(because $|t'-t|<-1/(2c_0)\log(1-\eps)$)}.
\end{align*}
Choosing $t_0:=t-\log(1-\eps)/(2c_0)$ in \cite[Theorem $2.4.1$ and the Remark following Theorem 
$2.4.1$]{Dynkin_diffsuperpde}, we see that the last term in the preceding sequence
of inequalities converges to $0$, and so the collection of random variables 
\eqref{eq:stochastic_representation_exponential_rv} converges in probability to $1$, as $(t',z')\in Q$
tends to $(t,z)$. The sequence is uniformly bounded by the constant $1$, and so \cite[Exercise $2.4.34$ 
(b)]{Folland_realanalysis} implies that the sequence converges to $1$ in
expectation also, that is
\begin{equation}
\label{eq:stochastic_representation_killing_term}
\lim_{Q \ni (t',z')\rightarrow (t,z)}\EE_\QQ^{t',z'}\left[\left|\exp\left(-\int_{t'}^{\tau_Q} c(s,Z(s))ds\right) 
-1\right|\right]= 0.
\end{equation}
From the sequence of inequalities,
\begin{align*}
&\left|\EE^{t',z'}_{\QQ}\left[ \exp\left(-\int_{t'}^{\tau_Q} c(s,Z(s))ds\right) g(\tau_Q, Z(\tau_Q))\right] - 
g(t,z)\right|\\
&\quad
\leq \left|\EE^{t',z'}_{\QQ}\left[g(\tau_Q, Z(\tau_Q))\right] - g(t,z)\right|+
\left|\EE^{t',z'}_{\QQ}\left[\left(1-\exp\left(-\int_{t'}^{\tau_Q} c(s,Z(s))ds\right)\right)g(\tau_Q, 
Z(\tau_Q))\right]\right|\\
&\quad
\leq\left|\EE^{t',z'}_{\QQ}\left[g(\tau_Q, Z(\tau_Q))\right] - g(t,z)\right|+
\|g\|_{L^{\infty}(\eth Q)}\EE^{t',z'}_{\QQ}\left[\left|1-\exp\left(-\int_{t'}^{\tau_Q} c(s,Z(s))ds\right)\right|\right],
\end{align*}
the conclusion \eqref{eq:stochastic_representation_Continuity_stochastic_representations_with_killing_term} follows from 
\eqref{eq:stochastic_representation_killing_term} and
\cite[Theorem $2.4.2$]{Dynkin_diffsuperpde} which shows that
\begin{equation*}
\lim_{Q \ni (t',z')\rightarrow (t,z)} \EE^{t',z'}_{\QQ}\left[g(\tau_Q, Z(\tau_Q))\right] = g(t,z).
\end{equation*}

We next consider the case when the stopping time $\tau_{Q}$ is bounded a.s. by a positive constant $T$. We fix $(t,z) \in 
\eth Q$. Without loss of generality, we may assume that
$t\in[0,T]$ and $Q \subseteq [0,T]\times\RR^d$. Because $c$ is a bounded function on $Q$, we let $c_1,c_2$ be two positive 
constants such that
\[
-c_1\leq c\leq c_2 \quad\hbox{a.e. on Q},
\]
and we set
\begin{align*}
\tilde c &:= c+ c_1 \quad\hbox{ on Q},
\end{align*}
and
\begin{align*}
\tilde g(t',z') &:= e^{c_1 (t'-t)} g(t',z'), \quad \forall (t',z') \in \eth Q.
\end{align*}
Notice that $\tilde c$ is a non-negative, bounded Borel measurable function on $Q$. Also, $\tilde g$ is a bounded, Borel 
measurable function on $\eth Q$, and it is continuous at
$(t,z)$ with
\begin{equation}
\label{eq:stochastic_representation_coincidence_g_tilde_g}
\tilde g(t,z) = g(t,z).
\end{equation}
In addition, we have for all $(t',z')\in Q$,
\begin{equation}
\label{eq:stochastic_representation_rewriting_e_c_g}
\begin{aligned}
&\exp \left(-\int_{t'}^{\tau_Q} c(s,Z^{t',z'}(s))ds\right)g(\tau_Q, Z^{t',z'}(\tau_Q))\\
&\quad=
\exp \left(-\int_{t'}^{\tau_Q} \tilde c(s,Z^{t',z'}(s))ds\right)\tilde g(\tau_Q, Z^{t',z'}(\tau_Q))\\
&\qquad+
\left(\exp\left(c_1(t-t')\right)-1\right)\exp \left(-\int_{t'}^{\tau_Q} \tilde c(s,Z^{t',z'}(s))ds\right)\tilde g(\tau_Q, 
Z^{t',z'}(\tau_Q)).
\end{aligned}
\end{equation}
The functions $\tilde c: \bar Q \rightarrow [0,\infty]$ and $\tilde g : \eth Q  \rightarrow \RR$ satisfy the requirements 
of the preceding case, and so, we have that
\[
\lim_{Q \ni (t',z')\rightarrow (t,z)} \EE^{t',z'}_{\QQ}\left[\exp \left(-\int_{t'}^{\tau_Q} \tilde 
c(s,Z(s))ds\right)\tilde g(\tau_Q, Z(\tau_Q))\right] = g(t,z),
\]
using \eqref{eq:stochastic_representation_coincidence_g_tilde_g}. By the boundedness of $\tilde c$ on $Q$, of $\tilde g$ 
on $\eth Q$, and the fact that $\tau_Q \leq T$ a.s., we
also have
\[
\lim_{Q \ni (t',z')\rightarrow (t,z)} \EE^{t',z'}_{\QQ}\left[\left(\exp\left(c_1(t-t')\right)-1\right)\exp 
\left(-\int_{t'}^{\tau_Q} \tilde c(s,Z(s))ds\right)\tilde g(\tau_Q,
Z(\tau_Q))\right] = 0.
\]
Therefore, the conclusion of the corollary follows from the preceding two limits and identity 
\eqref{eq:stochastic_representation_rewriting_e_c_g}.
\end{proof}

%
%

\bibliography{master,mfpde}
\bibliographystyle{amsplain}

\end{document}

%% file: latexformat.tex
\newtheorem{thm}{Theorem}[section]
\newtheorem*{thm*}{Theorem}
\newtheorem{lem}[thm]{Lemma}
\newtheorem*{lem*}{Lemma}

\newtheorem{cor}[thm]{Corollary}
\newtheorem{claim}[thm]{Claim}
\newtheorem{prop}[thm]{Proposition}

\theoremstyle{definition}

\newtheorem{assump}[thm]{Assumption}
\newtheorem*{case*}{Case}

\newtheorem{defn}[thm]{Definition}
\newtheorem*{defn*}{Definition}

\newtheorem*{exmp*}{Example}

\newtheorem{hyp}[thm]{Hypothesis}

\newtheorem{step}{Step}\renewcommand{\thestep}{}

\theoremstyle{remark}

\newtheorem{case}{Case}\renewcommand{\thecase}{}

\newtheorem{rmk}[thm]{Remark}
\newtheorem*{rmk*}{Remark}

\makeatletter
\def\alphenumi{
  \def\theenumi{\alph{enumi}}
  \def\p@enumi{\theenumi}
  \def\labelenumi{(\@alph\c@enumi)}}
\makeatother

\makeatletter
\def\thecase{\@arabic\c@case}
\makeatother

\makeatletter
\def\thestep{\@arabic\c@step}
\makeatother

%% file: latexnewmacros.tex
\newcount\hh
\newcount\mm
\mm=\time
\hh=\time
\divide\hh by 60
\divide\mm by 60
\multiply\mm by 60
\mm=-\mm
\advance\mm by \time
\def\hhmm{\number\hh:\ifnum\mm<10{}0\fi\number\mm}

\let\oldmarginpar\marginpar
\renewcommand\marginpar[1]{\-\oldmarginpar[\raggedleft\footnotesize #1]%
{\raggedright\footnotesize #1}}

\renewcommand\emptyset{\varnothing}

\newcommand\EE{\mathbb{E}}
\newcommand\FF{\mathbb{F}}

\newcommand\HH{\mathbb{H}}

\newcommand\NN{\mathbb{N}}

\newcommand\PP{\mathbb{P}}
\newcommand\QQ{\mathbb{Q}}
\newcommand\RR{\mathbb{R}}

\newcommand\sC{{\mathscr{C}}}
\newcommand\sD{{\mathscr{D}}}

\newcommand\sF{{\mathscr{F}}}

\newcommand\sO{{\mathscr{O}}}

\newcommand\sS{{\mathscr{S}}}
\newcommand\sT{{\mathscr{T}}}
\newcommand\sU{{\mathscr{U}}}

\newcommand\eps{\varepsilon}

\DeclareMathOperator{\erfc}{erfc}

\newcommand\loc{\operatorname{loc}}

%% file: stochastic_representation_PF1.20.2013b.bbl
\def\cprime{$'$} \def\cprime{$'$} \def\cprime{$'$} \def\cprime{$'$}
  \def\lfhook#1{\setbox0=\hbox{#1}{\ooalign{\hidewidth
  \lower1.5ex\hbox{'}\hidewidth\crcr\unhbox0}}} \def\cprime{$'$}
  \def\cprime{$'$} \def\cprime{$'$}
\providecommand{\bysame}{\leavevmode\hbox to3em{\hrulefill}\thinspace}
\providecommand{\MR}{\relax\ifhmode\unskip\space\fi MR }
\providecommand{\MRhref}[2]{%
  \href{http://www.ams.org/mathscinet-getitem?mr=#1}{#2}
}
\providecommand{\href}[2]{#2}
\begin{thebibliography}{10}

\bibitem{AbramStegun}
M.~Abramovitz and I.~A. Stegun, \emph{Handbook of mathematical functions},
  Dover, New York, 1972.

\bibitem{Adams}
R.~A. Adams, \emph{Sobolev spaces}, Academic Press, Orlando, FL, 1975.

\bibitem{Andersen_2007}
L.~B.~G. Andersen, \emph{Efficient simulation of the {H}eston stochastic
  volatility model}, \url{ssrn.com/abstract=946405}.

\bibitem{Bayraktar_Kardaras_Xing_2012}
E.~Bayraktar, C.~Kardaras, and H.~Xing, \emph{Valuation equations for
  stochastic volatility models}, SIAM J. Financial Math. \textbf{3} (2012),
  351--373. \MR{2968038}

\bibitem{Bensoussan_Lions}
A.~Bensoussan and J.~L. Lions, \emph{Applications of variational inequalities
  in stochastic control}, North-Holland, New York, 1982.

\bibitem{Billingsley_1986}
P.~Billingsley, \emph{Probability and measure}, second ed., Wiley, New York,
  1986.

\bibitem{Constantin_Iyer_2008}
P.~Constantin and G.~Iyer, \emph{A stochastic {L}agrangian representation of
  the three-dimensional incompressible {N}avier-{S}tokes equations}, Comm. Pure
  Appl. Math. \textbf{61} (2008), 330--345.

\bibitem{Daskalopoulos_Feehan_optimalregstatheston}
P.~Daskalopoulos and P.~M.~N. Feehan, \emph{${C}^{1,1}$ regularity for
  degenerate elliptic obstacle problems in mathematical finance},
  \url{arXiv:1206.0831v1}.

\bibitem{Daskalopoulos_Feehan_statvarineqheston}
\bysame, \emph{Existence, uniqueness, and global regularity for variational
  inequalities and obstacle problems for degenerate elliptic partial
  differential operators in mathematical finance}, \url{arXiv:1109.1075v1}.

\bibitem{Daskalopoulos_Feehan_evolvarineqheston}
\bysame, \emph{Existence, uniqueness, and global regularity for variational
  inequalities and obstacle problems for degenerate parabolic partial
  differential operators in mathematical finance}, in preparation.

\bibitem{DaskalHamilton1998}
P.~Daskalopoulos and R.~Hamilton, \emph{{$C^\infty$}-regularity of the free
  boundary for the porous medium equation}, J. Amer. Math. Soc. \textbf{11}
  (1998), 899--965.

\bibitem{Dynkin_diffsuperpde}
E.~B. Dynkin, \emph{Diffusions, superdiffusions and partial differential
  equations}, American Mathematical Society, Providence, RI, 2002.

\bibitem{Ekstrom_Tysk_bssvm}
E.~Ekstr{\"o}m and J.~Tysk, \emph{The {B}lack-{S}choles equation in stochastic
  volatility models}, J. Math. Anal. Appl. \textbf{368} (2010), 498--507.

\bibitem{Ekstrom_Tysk_bcsftse}
\bysame, \emph{Boundary conditions for the single-factor term structure
  equation}, Ann. Appl. Probab. \textbf{21} (2011), 332--350.

\bibitem{Feehan_Pop_regularityweaksoln}
P.~M.~N. Feehan and C.~A. Pop, \emph{Degenerate elliptic operators in
  mathematical finance and {H\"o}lder continuity for solutions to variational
  equations and inequalities}, \url{arXiv:1110.5594v2}.

\bibitem{Feehan_Pop_mimickingdegen}
\bysame, \emph{Degenerate-parabolic partial differential equations with
  unbounded coefficients, martingale problems, and a mimicking theorem for
  {I}t\^o processes}, \url{arXiv:1112.4824v1}.

\bibitem{Feehan_Pop_elliptichestonschauder}
\bysame, \emph{Schauder a priori estimates and regularity of solutions to
  degenerate-elliptic linear second-order partial differential equations},
  \url{arXiv:1210.6727v1}.

\bibitem{Feehan_Pop_mimickingdegen_pde}
\bysame, \emph{A {S}chauder approach to degenerate-parabolic partial
  differential equations with unbounded coefficients}, \url{arXiv:1112.4824v2}.

\bibitem{Feynman_thesis}
R.~P. Feynman, \emph{The principle of least action in quantum mechanics},
  {Ph.D}. thesis, Princeton University, Princeton, NJ, 1942, reprinted in
  ``Feynmans Thesis: a New Approach to Quantum Theory'', L. M. Brown (ed.),
  World Scientific, New York, 2005.

\bibitem{Folland_realanalysis}
G.~B. Folland, \emph{Real analysis}, second ed., Wiley, New York, 1999.

\bibitem{FriedmanPDE}
A.~Friedman, \emph{Partial differential equations of parabolic type}, Prentice
  Hall, New York, 1964.

\bibitem{Friedman_1973}
\bysame, \emph{Uniqueness for the {C}auchy problem for degenerate parabolic
  equations}, Pacific J. Math. \textbf{46} (1973), 131--147.

\bibitem{Friedman_1974}
\bysame, \emph{Fundamental solutions for degenerate parabolic equations}, Acta
  Math. \textbf{133} (1974), 171--217.

\bibitem{FriedmanSDE}
\bysame, \emph{Stochastic differential equations and applications}, vol. I, II,
  Academic, New York, 1975 and 1976.

\bibitem{GilbargTrudinger}
D.~Gilbarg and N.~Trudinger, \emph{Elliptic partial differential equations of
  second order}, second ed., Springer, New York, 1983.

\bibitem{Glasserman}
P.~Glasserman, \emph{{M}onte {C}arlo methods in financial engineering},
  Springer, New York, 2003.

\bibitem{Heston1993}
S.~Heston, \emph{A closed-form solution for options with stochastic volatility
  with applications to bond and currency options}, Review of Financial Studies
  \textbf{6} (1993), 327--343.

\bibitem{Ikeda_Watanabe}
N.~Ikeda and S.~Watanabe, \emph{Stochastic differential equations and diffusion
  processes}, North-Holland, Amsterdam, 1981.

\bibitem{Kac_1949}
M.~Kac, \emph{On distributions of certain {W}iener functionals}, Trans. Amer.
  Math. Soc. \textbf{65} (1949), 1--13.

\bibitem{KaratzasShreve1991}
I.~Karatzas and S.~E. Shreve, \emph{Brownian motion and stochastic calculus},
  second ed., Springer, New York, 1991.

\bibitem{KaratzasShreve1998}
\bysame, \emph{Methods of mathematical finance}, Springer, New York, 1998.

\bibitem{KarlinTaylor2}
S.~Karlin and Taylor, \emph{A second course on stochastic processes}, Academic,
  New York, 1981.

\bibitem{Koch}
H.~Koch, \emph{Non-{E}uclidean singular integrals and the porous medium
  equation}, Habilitation Thesis, University of Heidelberg, 1999,
  \url{www.mathematik.uni-dortmund.de/lsi/koch/publications.html}.

\bibitem{Krylov_LecturesHolder}
N.~V. Krylov, \emph{Lectures on elliptic and parabolic equations in {H}\"older
  spaces}, American Mathematical Society, Providence, RI, 1996.

\bibitem{Lieberman}
G.~M. Lieberman, \emph{Second order parabolic differential equations}, World
  Scientific Publishing Co. Inc., River Edge, NJ, 1996.

\bibitem{Lord_Koekkoek_vanDijk_2010}
R.~Lord, R.~Koekkoek, and D.~Van~Dijk, \emph{A comparison of biased simulation
  schemes for stochastic volatility models}, Quant. Finance \textbf{10} (2010),
  177--194.

\bibitem{Lorinczi_Hiroshima_Betz_2011}
J.~L{\H{o}}rinczi, F.~Hiroshima, and V.~Betz, \emph{Feynman-{K}ac-type theorems
  and {G}ibbs measures on path space}, de Gruyter, Berlin, 2011.

\bibitem{Oksendal_2003}
B.~{\O}ksendal, \emph{Stochastic differential equations}, sixth ed., Springer,
  Berlin, 2003.

\bibitem{Oleinik_Radkevic}
O.~A. Ole{\u\i}nik and E.~V. Radkevi{\v{c}}, \emph{Second order equations with
  nonnegative characteristic form}, Plenum Press, New York, 1973.

\bibitem{Revuz_Yor}
D.~Revuz and M.~Yor, \emph{Continuous martingales and {B}rownian motion}, third
  ed., Springer, New York, 1999.

\bibitem{Simon_functionalintegration}
B.~Simon, \emph{Functional integration and quantum physics}, second ed., AMS
  Chelsea Publishing, Providence, RI, 2005.

\bibitem{Skorokhod}
A.~V. Skorokhod, \emph{Studies in the theory of random processes},
  Addison-Wesley, Reading, MA, 1965.

\bibitem{StroockVaradhan1972}
D.~W. Stroock and S.~R.~S. Varadhan, \emph{On degenerate elliptic-parabolic
  operators of second order and their associated diffusions}, Comm. Pure Appl.
  Math. \textbf{25} (1972), 651--713.

\bibitem{Yamada_Watanabe_1971a}
S.~Watanabe and T.~Yamada, \emph{On the uniqueness of solutions of stochastic
  differential equations}, J. Math. Kyoto Univ. \textbf{11} (1971), 155--167.

\bibitem{Yamada_Watanabe_1971b}
\bysame, \emph{On the uniqueness of solutions of stochastic differential
  equations. {II}}, J. Math. Kyoto Univ. \textbf{11} (1971), 553--563.

\bibitem{Yamada_1978}
T.~Yamada, \emph{Sur une construction des solutions d'{\'e}quations
  diff{\'e}rentielles stochastiques dans le cas non-lipschitzien},
  S{\'e}minaire de {P}robabilit{\'e}s, {XII} ({U}niv. {S}trasbourg,
  {S}trasbourg, 1976/1977), Lecture Notes in Math., vol. 649, Springer, Berlin,
  1978, pp.~114--131.

\bibitem{Zhou_quasiderivdegenpde}
W.~Zhou, \emph{Quasiderivative method for derivative estimates of solutions to
  degenerate elliptic equations}, \url{arXiv:1112.5689}.

\end{thebibliography}
